\documentclass[11pt]{article}
\usepackage[margin = 1 in]{geometry}
\usepackage{amsmath}
\usepackage{amsfonts}
\usepackage{amssymb}
\usepackage{amsthm}
\usepackage{xcolor}
\usepackage{verbatim}
\usepackage[colorlinks = true, citecolor=blue]{hyperref}
\usepackage{bbm, bm}
\usepackage{mathtools}
\usepackage{upgreek}
\usepackage{scalerel}
\usepackage{mathrsfs}

\newtheorem{theorem}{Theorem}[section]

\newtheorem{prop}[theorem]{Proposition}
\newtheorem{lemma}[theorem]{Lemma}
\newtheorem{cor}[theorem]{Corollary}

\numberwithin{equation}{section}

\DeclareMathOperator*{\comp}{\bigcirc}
\newcommand{\cE}{\mathcal{E}}
\newcommand{\R}{\mathbb{R}}
\newcommand{\E}{\mathbb{E}}
\renewcommand{\P}{\mathbb{P}}
\newcommand{\N}{\mathbb{N}}
\newcommand{\V}{\mathbb{V}}
\newcommand{\av}{\mathrm{Av}}
\newcommand{\M}{\mathcal{M}}
\newcommand{\var}{\mathrm{Var}}
\newcommand{\1}{\mathbbm{1}}
\newcommand{\cP}{\mathcal{P}}
\newcommand{\T}{\mathcal{T}}
\newcommand{\X}{\mathcal{X}}
\newcommand{\gse}{\mathrm{GSE}}
\newcommand{\la}{\langle}
\newcommand{\ra}{\rangle}
\renewcommand{\i}{\mathbf{i}}
\newcommand{\G}{\mathcal{G}}
\newcommand{\f}{\mathbf{f}}
\newcommand{\Z}{\mathbb{Z}}
\newcommand{\h}{\mathrm{hc}}
\renewcommand{\Pr}{\mathrm{Pr}}

\begin{document}
\title{On the Replica Symmetric Solution in General Diluted Spin Glasses}
\author{Ratul Biswas, Wei-Kuo Chen and Arnab Sen}
\date{}
\maketitle
\begin{abstract}
     We present a unifying approach to studying the replica symmetric solution in general diluted spin glass models on random $p$-uniform hypergraphs with sparsity parameter $\alpha$. Our result shows that there exist two key regimes in which the model exhibits replica symmetry and the free energy can be explicitly represented as the evaluation of an energy functional at the unique fixed point of a recursive distributional equation. One is called the high temperature regime, where the temperature and the sparsity parameter are essentially inversely proportional to each other; the other is the subcritical regime defined as $\alpha p (p-1)\leq 1$. In particular, the fact that the second regime is independent of the temperature parameter further allows us to deduce an analogous representation of the ground state energy in the subcritical regime. Along the way, we revisit several well-known formulas and also derive new ones for the free and ground state energies in the constraint satisfaction problem, Potts model, XY model, and continuous hardcore model.
    
\end{abstract}
\allowdisplaybreaks
{
\hypersetup{linkcolor = black}
\tableofcontents
}
\section{Introduction}

Understanding the unusual magnetic behavior of spin glasses has intrigued physicists for decades. Thanks to theoretical predictions and experimental observations made in the physics literature, significant progress in the study of spin glass models has been achieved largely for mean field models, in which the spin interactions are uniformly defined on fully connected graphs. We refer the reader to \cite{Mzard1986SpinGT, Panchenko2013TheSM, book, book2} for developments. Naturally, for more realistic models, one would expect that on average, interactions between spins are diluted in the sense that the spin at a given site interacts only with the spins at a bounded number of sites in its vicinity. This gives rise to spin glass models where the underlying interaction graph of the Hamiltonian is given by sparse random graphs of constant average degree, such as the diluted Erd\"os-R\'enyi graph and the random $d$-regular graph. Important examples originating from statistical physics include the Ising and Potts models on sparse random graphs \cite{dembo2010gibbs, Dembo_2013}, the Viana-Bray model \cite{viana1985phase}, diluted Hopfield model \cite{bovier1992rigorous} and diluted $V$-statistics \cite{Talagrand2016AMS}.  In addition, diluted models have been prominently featured in computer science, particularly in the context of random constraint satisfaction problems, such as $K$-SAT and NAE-$K$-SAT as well as the $q$-coloring and the hardcore model on sparse random graphs, see \cite{Achlioptas2005RigorousLO,  ding2021proof, Krz_aka_a_2007, mezard2009information} and the  references therein.

In the investigation of diluted spin glass models, an important component is concerned with the so-called M\'ezard-Parisi replica symmetry ansatz \cite{mezard2001bethe}. Roughly speaking, it states that there exists a regime of system parameters, in which the overlap between two spin configurations, sampled independently from the Gibbs measure, is approximately constant. This implies that the spins are in a `pure state', i.e., under the Gibbs measure, a collection of finitely many spins of the system is asymptotically uncorrelated (see \cite[Proposition 1.4.14 and Theorem 6.7.8]{book} for example) and the quenched distribution of a spin is expected to be characterized as the fixed point of a certain distributional operator. 

However, developing a strategy that establishes these two properties for a general diluted spin glass model is a challenging task. To this end, Talagrand developed an approach to studying general diluted models with Ising spin configurations and validated the M\'ezard-Parisi replica symmetry ansatz in a regime where the temperature essentially varies inversely with the sparsity parameter, see \cite[Chapter 6]{book} for details. Nevertheless, there are examples of models, where by exploiting the specifics of the models under consideration, the replica-symmetric free energy has been obtained over much larger regimes of temperature and sparsity parameters. For example, the replica-symmetric free energy of the random $K$-SAT model \cite{Montanari-Shah, EJP2963} and Viana-Bray model \cite{Guerra_2004} can be obtained at any temperature if the underlying graph is sufficiently sparse. In some other diluted models, the replica symmetric free energy was obtained for all positive temperatures and sparsities; see \cite{Dembo_2010} for the ferromagnetic Ising model with an external field, \cite{dembo_2012} for the ferromagnetic $q$-Potts model on random regular graphs with even degree, and \cite{10.1214/22-AOP1597} for the Shcherbina-Tirozzi model with a quadratic Hamiltonian.

In this paper, we aim to present a unifying approach to studying the replica-symmetric solution for general diluted spin glasses following  Panchenko's framework of asymptotic Gibbs measures \cite{Panchenko2010SPINGM}. Owing to the symmetry between the spin index $i$ and the replica index $l$, in the limit, the two-dimensional array of spin configurations $(\sigma_i^l)_{i,l\geq 1}$ under the annealed Gibbs measure is exchangeable and hence, due to the Aldous-Hoover representation \cite[Section 1.4]{Panchenko2013TheSM}, it can be characterized as the array $\sigma(w,u_l,v_i,x_{l,i})_{i,l\geq 1}$ for some measurable function $\sigma$, where $w, u_l, v_i, x_{l,i}$ are i.i.d.\ copies of uniform random variables on $[0,1]$. Rather than working with finite spin configurations, we consider the limiting exchangeable spin array to establish the free energy representation in terms of a functional evaluated at the unique fixed point of a recursive distributional equation in two key replica symmetric regimes defined by the temperature and sparsity parameters. 

The first is known as the high temperature regime, where the temperature and the sparsity parameter are essentially inversely proportional to each other. Our results in this regime are similar to those obtained by Talagrand \cite[Chapter 6]{book} who used a sophisticated inductive approach in finite systems. 

The second regime is new, permitting any temperature as long as the graph is subcritically or critically dense. Instead of proving the key uniqueness property of the distributional fixed point equation governing the limiting spin distribution via the usual contractive approach as in \cite{book}, we obtain this by analyzing certain invariant processes on finite trees that describe the local neighborhoods of the interaction graph. A major advantage of our approach is that our uniqueness argument is impervious to temperature as long as the sparsity parameter is in the subcritical or critical regime. In this case, we obtain a formula for the limiting free energy that holds for any temperature which further allows us to derive an analogous representation for the ground state energy by letting the temperature go to infinity.

To demonstrate the applicability of our results, we revisit several important diluted models and derive a number of new formulas for the related quantities. One major class of models we consider is the constraint satisfaction problems, including the $K$-SAT and the symmetric perceptron models, and establish the limit for the logarithmic number of spin configurations that satisfy the underlying constraints with an average fraction $t\in (0,1).$ Furthermore, we apply our results to two examples with  continuous spins, the XY model and the continuous hardcore model (introduced in \cite{gamarnik2017uniqueness}). In the subcritical regime, we demonstrate the solvability of the fixed point operator in the XY model and obtain explicit formulas for the free energy and ground state energy. For the continuous hardcore model, we derive the limit for the logarithmic volume for the spin configurations that satisfy the hardcore constraint.

\subsection{Main results}

For any $k\in \N = \{1, 2, \ldots\}$ (we shall use $\Z_+ := \N \cup\{0\})$, let $[k]=\{1,2,\ldots,k\}$ and for $A\subset \N$ and $r\geq 1,$ let $\binom{A}{r}$ be the collection of all subsets of $A$ with $r$ distinct elements. For some $R > 0$, we let $\Sigma$ be a compact subset of $[-R,R]$. Denote by $\Pr(\Sigma)$ the set of probability measures on $\Sigma$ and let $\nu$ be an element of $\Pr(\Sigma)$. For $p\geq 2,$ assume that $\theta:\Sigma^p\to \R$ is a random function that satisfies the following  conditions:\begin{align}\label{bdd theta}
	& \E \|\theta\|^2_\infty<\infty
 \end{align}
 and
 \begin{align}
	& \E\, \mbox{Lip}(\theta) <\infty,\label{eq reg theta}
\end{align}
where \begin{align}\label{eq lip}
    \mbox{Lip}(\theta):=\sup_{x, y \in \Sigma^p: x\neq y}\frac{|\theta(x) - \theta(y)|}{\|x-y\|_2}.
\end{align}
The Hamiltonian of the diluted model is defined as 
 \begin{align}\label{eq:orig hamiltonian}
	 H_N(\sigma) = \sum_{k \leq \pi(\alpha N)} \theta_k(\sigma_{I(k,1)}, \ldots, \sigma_{I(k,p)}) + \sum_{i=1}^N \psi(\sigma_i)
\end{align}
for $\sigma=(\sigma_1,\ldots,\sigma_N)\in \Sigma^N.$ Here, $\alpha >0$ is the sparsity parameter, $\psi : \Sigma\to \R$ is a (bounded) Lipschitz function that measures the external field, $(\theta_k)_{k\geq 1}$ are i.i.d. copies of $\theta$, $\pi(\alpha N)$ is a Poisson random variable with mean $\alpha N$, and for each $k\geq 1,$ $I_k: = \{I(k,1),\ldots,I(k,p)\}$ is uniformly sampled from $\binom{[N]}{p}$. These are all independent of each other. At the (inverse) temperature $\beta >0$, define the Gibbs measure as
$$
G_{N,\nu,\beta}(d\sigma)=\frac{1}{Z_N(\beta)}e^{\beta H_N(\sigma)}\nu^{\otimes N}(d\sigma)
$$
where the normalizing constant $
Z_N(\beta) =\int e^{\beta H_N(\sigma)}\nu^{\otimes N}(d\sigma)$ is called the partition function. The free energy at $\beta$ and ground state energy are defined respectively as
\begin{align}\label{freeandGSE}
   F_N(\beta) =\frac{1}{N}\log Z_N(\beta)\;\;\text{and}\;\;\gse_N=\frac{1}{N}\max_{\sigma\in \Sigma^N}H_N(\sigma). 
\end{align}

Our first result establishes the replica symmetric formula for the limiting free energy. To describe this limit, we introduce a distributional operator. Let $\X:=C(\Sigma)$ be the space of real-valued continuous functions on $\Sigma$ equipped with the uniform metric, i.e., $$\|f_1-f_2\|_\infty = \sup_{t\in \Sigma}|f_1(t)-f_2(t)|,\;\; f_1,f_2\in \X.$$ Note that $(\X,\|\cdot\|_\infty)$ is complete and separable. Let  $\mathcal{B}$ be the corresponding Borel $\sigma$-field. Denote by $\Pr_1(\X)$ the collection of all probability measures $\lambda$ defined on $(\X,\mathcal{B})$ with 
$
\int \|f\|_\infty\lambda(df)<\infty.
$
Define the Wasserstein $1$-distance on $\Pr_1(\X)$ as
\begin{align*}
	W_1(\lambda _1,\lambda _2) = \inf_{\Pi \in \Pi(\lambda _1, \lambda _2)}\int \|f_1- f_2\|_\infty d\Pi(f_1,f_2)
\end{align*}
for any $\lambda _1,\lambda _2\in \Pr_1(\X)$, where $\Pi(\lambda _1,\lambda _2)$ is the set of all couplings of $\lambda _1$ and $\lambda _2$. 
For any $\beta <\infty$, $\lambda\in \Pr_1(\X)$ and a continuous function $f:\Sigma^n\to \R$, set
$$
\la f(\sigma)\ra_{\beta, X}=\int f(\sigma)\prod_{k=1}^ne^{\beta X_k(\sigma_k)}\nu^{\otimes n}(d\sigma),
$$
where $X, X_1,\ldots,X_n$ are i.i.d. samples drawn according to $\lambda$. Define, for $\beta < \infty$, the distributional operator $\T_{\nu, \beta} :\Pr_1(\X)\to \Pr_1(\X)$ as the law of the random function $T_{\nu,\beta,\pi(\alpha p)}(X_1,\ldots,X_{(p-1)\pi(\alpha p)})$, where
\begin{align}
	\label{opeartor:eq1}
	T_{\nu,\beta, \pi(\alpha p)}(X_1,\ldots,X_{(p-1)\pi(\alpha p)})(t):=\frac{1}{\beta}\log\frac{\la\cE_{\beta, \pi(\alpha p)}(\sigma, t) \ra_{\beta,X}}{\la\int  \cE_{\beta, \pi(\alpha p)}(\sigma, s) \nu(ds)\ra_{\beta,X}},\;\;t\in \Sigma,
\end{align}
with $\pi(\alpha p)$ a Poisson random variable with mean $\alpha p$ sampled independently of $(X_k)_{k\geq 1}$ and for $r\geq 0$, $\sigma\in \Sigma^{r(p-1)},$ and $t\in \Sigma,$
\begin{align}\label{def:E}
	\cE_{\beta,r}(\sigma, t):=  \exp\Bigl(\beta \sum_{k=1}^{r}\theta_k(\sigma_{(k-1)(p-1)+1}, \ldots, \sigma_{k(p-1)},t) + \beta \psi(t)\Bigr).
\end{align}
Observing that \begin{align}\label{l_infty bound}
	\bigl\|T_{\nu,\beta, \pi(\alpha p)}(X_1,\ldots,X_{(p-1)\pi(\alpha p)})\bigr\|_{\infty}\leq  2\sum_{k=1}^{\pi(\alpha p)}\|\theta_k\|_\infty + 2\|\psi\|_\infty,
\end{align}
we have by \eqref{bdd theta} that $\T_{\nu, \beta}(\lambda)\in \Pr_1(\X)$. We also note that \begin{align}\label{lip_bound}
    \mbox{Lip}\bigl(T_{\nu,\beta,\pi(\alpha p)}(f_1,\ldots,f_{(p-1)\pi(\alpha p)})\bigr)&\leq \sum_{k=1}^{\pi(\alpha p)}\mbox{\rm Lip}(\theta_k) + \mbox{\rm Lip}(\psi).
\end{align}
 For $\lambda \in \Pr_1(\X)$, define the functional \begin{align}\label{free_energy}
 	\begin{split}
 		\cP_{\nu, \beta}(\lambda) & = \E \log \Bigl\la \int \cE_{\beta,\pi(\alpha p)}(\sigma,\varepsilon )\nu(d\varepsilon)\Bigr\ra_{\beta,X}- \alpha(p-1)\E \log \bigl\la e^{\beta \theta(\sigma)} \bigr\ra_{\beta,X}.
 	\end{split}
 \end{align}
The following theorem establishes the replica symmetric formula for the limiting free energy.

\begin{theorem}[Free energy]\label{thm1}
In the regime
	\begin{align}\label{ht}
		 \min(1, 6\beta e^{4\beta \|\psi\|_\infty}\E \|\theta\|_\infty e^{4\beta \|\theta\|_\infty}) \alpha p(p-1)\leq 1,
	\end{align} 
the following statements hold true.
\begin{enumerate}
    \item  The operator $\T_{\nu, \beta}$ admits a unique fixed point $\lambda_{\nu,\beta} \in \Pr_1(\X)$.
    \item  The free energy converges to $F(\beta): = \cP_{\nu, \beta}(\lambda_{\nu, \beta})$ in $L^1$, i.e.,
    $$
		\lim_{N\to \infty} \E|F_N(\beta)-F(\beta)| = 0.
    $$
    \end{enumerate}
\end{theorem}

Let us now discuss our results in detail. The inequality \eqref{ht} consists of two parts: the `high temperature regime'
$$
6\alpha \beta p(p-1)e^{4\beta \|\psi\|_\infty}\E \|\theta\|_\infty e^{4\beta \|\theta\|_\infty} \leq 1,$$
and the `critical-subcritical regime'
$$\alpha p(p-1)\leq 1.$$ 
For brevity, we will just call the latter the `subcritical regime'. First of all, the fixed point $\lambda_{\nu,\beta}$ should be understood as the log-density of the limiting spin distribution with respect to the measure $\nu.$ Our result in the high temperature regime is in spirit the same as that obtained in \cite[Theorem 6.4.13]{book} and although our result improves on the extent of this regime, it is unlikely that our inequality describing this regime is tight.

The validity of the replica symmetric formula for a general choice of $\theta$, $\nu$, and $\psi$ in the subcritical regime is new. We follow  Panchenko's approach \cite{Panchenko2010SPINGM, EJP2963} to work with subsequential limits of spin configurations, which can be viewed as an exchangeable array with respect to the spin and replica coordinates. This avoids tedious computations involving finite $N$. Under the subcriticality assumption, the limiting spin configuration is in a pure state, which means that the corresponding exchangeable array does not depend on the replica coordinates. This was done in \cite[Lemma 1]{EJP2963} by a clever contractive approach. In this paper, we provide a more intuitive proof in the subcritical regime (see Section \ref{sec3.1}) by analyzing a recursion of the moments of the limiting spins on the subcritical Galton-Watson tree. By using the pure state condition, the log-density of the limiting spin is shown to satisfy the same fixed point distributional equation \eqref{opeartor:eq1}. However, instead of i.i.d.\ variables, the equation now involves a one-dimensional exchangeable sequence.  The convergence follows if we can show that this exchangeable fixed point equation has a unique solution. Such uniqueness is not trivial even in the i.i.d.\ case, which was usually proved using the contractivity of the distributional operator (in an appropriate metric) under the assumption that both $\beta$ and $\alpha$ are small, e.g., see \cite{book}. However, since the subcritical regime allows $\beta$ to be arbitrarily large, it is not clear how to establish such contractivity. 
Instead, our approach relies on the systematic treatment of the general recursive distributional equations (RDE)  in \cite{aldous2005survey}. In the i.i.d.\ case, our operator is an RDE driven by a Galton-Watson tree with $(p-1)\pi(\alpha p)$ offspring distribution.
The uniqueness of the solution of the RDE is a consequence of the fact that this tree has finite depth under the subcriticality assumption, i.e.,  $\alpha p(p-1) \le 1$. More importantly, this approach also yields uniqueness in the exchangeable case, which does not belong to the RDE setup; the proof is carried out in Section \ref{sec3.2} below. 

Next, we obtain and expression for the limiting ground state energy analogous to the limiting free energy. We define an operator $\T_\infty:\Pr_1(\X)\to \Pr_1(\X)$ as follows. For $\lambda \in \Pr_1(\X)$, $\T_\infty(\lambda)$ denotes the law of the random function $T_{\infty,\pi(\alpha p)}(X_1, \ldots, X_{(p-1)\pi(\alpha p)})$, which is defined as 
 \begin{align}\label{op_infty}
 \begin{split}
	T_{\infty,\pi(\alpha p)}(X_1, \ldots, X_{(p-1)\pi(\alpha p)})(t) 
	& = \sup_{\sigma \in \Sigma^{(p-1)\pi(\alpha p)}}\Big(\mathcal{I}_{\pi(\alpha p)}(\sigma,t)+ \sum_{k=1}^{(p-1)\pi(\alpha p)}X_k(\sigma_k)\Big)\\
	&\qquad  - \sup_{(s,\sigma) \in \Sigma^{(p-1)\pi(\alpha p) + 1}}\Big(\mathcal{I}_{\pi(\alpha p)}(\sigma,s) + \sum_{k=1}^{(p-1)\pi(\alpha p)}X_k(\sigma_k)\Big),
 \end{split}
\end{align}
where $(X_k)_{k\geq 1}$ are i.i.d. sampled from $\lambda$ independent of $\pi(\alpha p)$, a Poisson random variable with mean $\alpha p$, and for $r\geq 0$, $\sigma \in \Sigma^{r(p-1)r}$, and $t\in \Sigma$,
\begin{align*}
	\mathcal{I}_r(\sigma,t):=\sum_{k=1}^r \theta_k(\sigma_{(k-1)(p-1) + 1}, \ldots, \sigma_{k(p-1)},t) +\psi(t).
\end{align*}
Here, as before, it is easy to see that \eqref{l_infty bound} and \eqref{lip_bound} hold for the operator $\T_\infty$, so $\T_\infty(\lambda)\in \Pr_1(\X)$. Using the same notation as above, let us define the functional $\cP_\infty$ on $\Pr_1(\X)$ by \begin{align} \label{functional, infty}
	\begin{split}
	\cP_\infty(\lambda) & = \E \sup_{(s,\sigma) \in \Sigma^{(p-1)\pi(\alpha p) + 1}}\Big(\mathcal{I}_{\pi(\alpha p)}(\sigma,s) + \sum_{k=1}^{(p-1)\pi(\alpha p)}X_k(\sigma_k)\Big)\\
	&\qquad - \alpha(p-1)\E \sup_{\sigma \in \Sigma^p} \Big(\theta(\sigma)+ \sum_{k=1}^p X_k(\sigma_k)\Big).
\end{split}
\end{align}

\begin{theorem}[Ground state energy]\label{thm:GSE}
 In the regime \begin{align}\label{small alpha}
		\alpha p(p-1) \leq 1,
	\end{align}
the following statements hold true.
\begin{enumerate}
    \item The operator $\T_\infty$ admits a unique fixed point $\lambda_\infty \in \Pr_1(\X)$.
    \item  The ground state energy converges (in $L^1$) to $\gse:=\cP_\infty(\lambda_\infty)$, i.e., 
\end{enumerate}
\begin{align*}
		\gse_N\xrightarrow{L_1}\gse.
	\end{align*}
\end{theorem}

\subsection{Results for models with Ising spins}
We now apply our theorems above to diluted models with Ising spins. For the rest of the paper, we let $\upnu$ denote the uniform probability measure on $\Sigma=\{-1,1\}.$ 

Suppose that $\theta$, in addition to being invariant with respect to permutations of inputs and satisfying properties \eqref{bdd theta} and \eqref{eq reg theta}, also satisfies the \textit{symmetry property}, i.e., $\theta(\sigma) = \theta(-\sigma)$ for any $\sigma \in \{-1,1\}^p$. In the absence of an external field, i.e., $\psi \equiv 0$, it can be verified directly that the Dirac measure at the zero function is a fixed point of $\mathcal{T}_{\upnu,\beta}$ and $\mathcal{T}_\infty$, from which the free energy and the ground state energy can be explicitly computed, as illustrated in the result below.

\begin{prop}\label{prop1}
Let $\upnu$ be the uniform probability measure on $\{-1,1\}$ and $\psi \equiv 0.$ For any $\theta$ that satisfies \eqref{bdd theta}, \eqref{eq reg theta}, and the symmetry property, 
	\begin{itemize}
	\item[$(i)$] if \eqref{ht} holds, then 
	$$
	F_N(\beta)\xrightarrow{L_1} F(\beta) = \alpha \E\log \sum_{\sigma\in \{-1,1\}^p}e^{\beta\theta(\sigma)} - \alpha p \log 2,$$
	\item[$(ii)$] if \eqref{small alpha} holds, then 
	$$\gse_N\xrightarrow{L_1} \gse = \alpha \E\max_{\sigma\in \{-1,1\}^p}\theta(\sigma).$$
\end{itemize}
\end{prop} 

Our second consequence of Theorems \ref{thm1} and \ref{thm:GSE} establishes the Parisi formula for the free energy and ground state energy if $\theta$ satisfies the Franz-Leone identity:
\begin{align}\label{identity}
	\exp \beta \theta(x_1,\ldots,x_p)=a(1+bf_1(x_1)\cdots f_p(x_p))
\end{align}
for all $x_1,\ldots,x_p\in \{-1,1\},$ where $a = a(\beta), b = b(\beta)$ are random functions and $f_1, \ldots, f_p$ are copies of the random function $f$ that does not depend on $\beta$, all sampled independently of each other. Here $a$ can be any random variable, while $b, f_1, \ldots, f_p$ satisfy \begin{align*}
\E(-b)^n\geq 0 \text{ for all } n\geq 1\;\;\text{and } \;\; |bf_1(x_1)\cdots f_p(x_p)| < 1 \text{ a.s. }
 \end{align*}
\begin{prop}
	\label{Parisiformula}
	 Suppose that $\theta$ satisfies \eqref{identity}. With $F(\beta)$ and $\gse$ being as in Theorems \ref{thm1} and \ref{thm:GSE}, 
	\begin{itemize}
		\item[$(i)$] if \eqref{ht} holds, then $$
			F_N(\beta)\xrightarrow{L_1}\min_{\lambda\in \Pr_1(\X)}\cP_{\upnu,\beta}(\lambda) = F(\beta),$$
			\item[$(ii)$] if \eqref{small alpha} holds, then $$\gse_N\xrightarrow{L_1}\min_{\lambda\in \Pr_1(\X)}\cP_\infty(\lambda) = \gse.$$
	\end{itemize}
\end{prop}

\begin{proof}
    Under the hypothesis of Proposition \ref{Parisiformula}, it is well-known that the Franz-Leone bound holds, i.e., $\E F_N(\beta)\leq \cP_{\upnu,\beta}(\lambda)$  and $\E\, \gse_N\leq \cP_\infty(\lambda)$ for all $\lambda\in \Pr_1(\X)$ (see, e.g., \cite[Theorem 6.5.1]{book}). These bounds, in combination with Theorems \ref{thm1} and \ref{thm:GSE}, yield Proposition \ref{Parisiformula}.
\end{proof}

Often, the free energy for models satisfying the Franz-Leone identity \eqref{identity} is better interpreted in terms of spin magnetizations rather than the log-densities of the spin distribution as in the previous proposition. We explain this in detail below.
 
Firstly, note that for any $r\geq 0$ and $g_1, \ldots, g_r \in \X$, the following pointwise equality of functions holds, \begin{align*}
    T_{\upnu,\beta, r}(g_1, \ldots, g_r) = T_{\upnu,\beta, r}(\bar g_1, \ldots, \bar g_r)
\end{align*} where $\bar g$ is the function $g$ normalized with respect to $\upnu$, i.e., $\bar g: = g - \frac{1}{\beta}\log \int e^{\beta g} d \upnu$. Thus, $\lambda_{\upnu,\beta}$ is supported on the set $\X_\beta: = \{g\in \X: \int e^{\beta g} d\upnu = 1\}$. Therefore, for the rest of this section, we will restrict our attention to functions in $\X_\beta$. Since any function $f$ in $\X_\beta$ takes values in $\{-1,1\}$, we may write $f(\varepsilon) = \xi + \varepsilon \zeta$ for $\varepsilon \in \{-1,1\}$ where $\xi$ and $\zeta$ are random variables independent of $a$ and $b$. In the same vein, we can assume that $\psi$ is linear. In fact, shifting $\psi$ by a constant alters the free and ground state energies by the same constant, so without loss of generality, we can consider $\psi(x) = h x$, where $h \in \R$. It will be useful to extend the definition of $f$ and $\psi$ by linear interpolation to all $\varepsilon \in \Sigma': = [-1,1]$. Equip $\Sigma'$ with the $L^1$ metric and let $\Pr_1(\Sigma')$ denote the set of probability measures on $(\Sigma', \mathcal{B}(\Sigma'))$, metrized by the Wasserstein 1-distance, i.e., for $\lambda_1, \lambda_2 \in \Pr_1(\Sigma')$, \begin{align*}
    W_1(\lambda_1, \lambda_2) = \inf_{\Pi \in \Pi(\lambda_1, \lambda_2)} \int |x_1 - x_2| d\Pi(x_1, x_2)
\end{align*}
where $\Pi(\lambda_1, \lambda_2)$ is the set of all couplings of the measure $\lambda_1$ and $\lambda_2$.

Define the operator $\T'_{\upnu,\beta}:\Pr_1(\Sigma') \to \Pr_1(\Sigma')$ as follows: for $\lambda \in \Pr_1(\Sigma')$, let $(m_{k,i})_{k\geq 1, i\leq p-1}$ be i.i.d. samples from $\lambda$. Then \begin{align}\label{fp_mag}
    \T'_{\upnu,\beta}(\lambda) = \text{ Law of } \Big( \tanh\Big[\sum_{k\leq \pi(\alpha p)} \tanh^{-1}\frac{b_k\zeta_k\prod_{i\leq p-1}f_{k,i}(m_{k,i})}{1 + b_k\xi_k\prod_{i\leq p-1}f_{k,i}(m_{k,i})} + \beta h\Big] \Big) 
\end{align}
where $(b_k)_{k\geq 1}, (f_{k,i})_{k\geq 1, i\leq p-1}$ are copies of the random functions $b = b(\beta)$ and $f$ respectively, $(\zeta_k)_{k\geq 1}$ and $(\xi_k)_{k\geq 1}$ are copies of the random variables $\zeta$ and $\xi$ respectively, and $\pi(\alpha p)$ is a Poisson random variable with mean $\alpha p$, all independent of each other and the sequence $(m_{k,i})_{k\geq 1, i\leq p-1}$. 

Let us also, using the same notation as above, define the following functional: for $\lambda \in \Pr_1(\Sigma')$, \begin{align}\label{fe_old} \begin{split}
\cP'_{\upnu,\beta}(\lambda) & = -\log 2 + \alpha \E \log a - \alpha(p-1)\E\log \Big[1 + b\prod_{i\leq p}f_i(m_i)\Big] \\
& \qquad + \E \log \Big[\sum_{\varepsilon \in \{-1,1\}} e^{\varepsilon \beta h} \prod_{k\leq \pi(\alpha p)}\Big(1 + b_kf_k(\varepsilon) \prod_{i\leq p-1}f_{k,i}(m_{k,i}) \Big)\Big]
\end{split}
\end{align}
where $(f_k)_{k \geq 1}$ are i.i.d. copies of $f$ and $(m_i)_{i\leq p}$ are i.i.d. samples from $\lambda$, independent of everything else. Fixing $\beta <\infty$ and identifying each element $g \in \X_\beta$ uniquely with its \textit{magnetization} $$m(g): = \int \varepsilon e^{\beta g(\varepsilon)} \upnu(d\varepsilon) = \frac{1}{2}\Big(e^{\beta g(1)} - e^{\beta g(-1)}\Big) \in \Sigma'$$
we obtain the following corollary of Theorem \ref{thm1}.

\begin{cor}\label{cor2}
Suppose that $\theta$ satisfies \eqref{identity}. If \eqref{small alpha} holds, then 
\begin{enumerate}
    \item  the operator $\T'_{\upnu,\beta}$ admits a unique fixed point $\lambda'_{\upnu,\beta} \in \Pr_1(\Sigma')$, and
    \item  the free energy converges (in $L^1$) to $F(\beta)= \cP'_{\upnu,\beta}(\lambda'_{\upnu,\beta})$, i.e.,
    $$
		\lim_{N\to \infty} \E|F_N(\beta)-F(\beta))| = 0.
    $$
    \end{enumerate}
\end{cor}

Note that, restricted to $\X_\beta$, $\T_{\upnu,\beta}$ is the operator $(g\mapsto m(g))^{-1}\circ \T'_{\upnu,\beta} \circ(g\mapsto m(g))$, so the uniqueness of the fixed point of $\T'_{\upnu,\beta}$ follows from that of $\T_{\upnu,\beta}$ and the invertibility of the map $g\mapsto m(g)$. The expression for the free energy follows by changing variables in an analogous manner.

In instances when the distribution of $b$, $\zeta$ and $\xi$ are atomic, as in the case of Bernoulli disorder, we obtain that the fixed point $\lambda'_{\upnu,\beta}$ is an atomic probability measure.
\begin{prop}
    Suppose that \eqref{small alpha} holds and $\theta$ satisfies $\eqref{identity}$, where $f_1, \ldots, f_p$ are i.i.d. copies of the function $f(\varepsilon) = \xi + \varepsilon \zeta$. If the distribution of $b$, $\zeta$ and $\xi$ are atomic, then $\lambda_{\upnu,\beta}'$ is an atomic probability measure on $\Sigma'$.
\end{prop}
\begin{proof}
    Suppose that the total atomic mass of the measure $\lambda_{\upnu,\beta}'$ is $c \in [0,1]$. Let $(m_{k,i})_{k\geq 1, i\leq p-1}$ be i.i.d. samples from $\lambda'_{\upnu,\beta}$. Then for any $l \geq 0$, the total atomic mass of the law of the random variable
    \begin{align*}
        \sum_{k\leq l} \tanh^{-1}\frac{b_k\zeta_k\prod_{i\leq p-1}f_{k,i}(m_{k,i})}{1 + b_k\xi_k\prod_{i\leq p-1}f_{k,i}(m_{k,i})}
    \end{align*}
    is equal to $c^{(p-1)l}$ and hence the total atomic mass of the law of the random variable\begin{align*}
        \tanh\Big[\sum_{k\leq \pi(\alpha p)} \tanh^{-1}\frac{b_k\zeta_k\prod_{i\leq p-1}f_{k,i}(m_{k,i})}{1 + b_k\xi_k\prod_{i\leq p-1}f_{k,i}(m_{k,i})} + \beta h\Big]
    \end{align*}
    is equal to $\E c^{(p-1)\pi(\alpha p)} = e^{\alpha p(c^{p-1}-1)}$. Thus, from \eqref{fp_mag}, we obtain that $c = e^{\alpha p (c^{p-1}-1)}$. We claim that under \eqref{small alpha}, the only solution to this equation in $[0,1]$ is $c=1$. Firstly, observe that $c>0$. Changing variables $x = -\log c \geq 0$ we obtain the new equation as $\alpha p (1-e^{-x(p-1)}) = x$. However, note that \begin{align*}
        1-e^{-x(p-1)} \leq x(p-1) = \alpha p (p-1) (1-e^{-x(p-1)}) \leq 1-e^{-x(p-1)},
    \end{align*}
    where the final equality follows from \eqref{small alpha}. The conclusion now follows easily,
\end{proof}

\subsection{Structure of the paper}
The rest of the paper is structured as follows. In Section \ref{examples}, we visit several important diluted models in the context of the results that we obtained above. Section \ref{distributional_operator} is dedicated to the study of the random distributional operator $\T_{\nu,\beta}$. Utilizing Panchenko's invariance principle, we establish the model's replica symmetric behavior and identify the limiting spin distribution in Section \ref{rs_behavior}. The free and ground state energies of the dilute model are derived in Sections \ref{fe_dilute_model} and \ref{gse_dilute_model} respectively. In Section \ref{sec6} we count the number of spin configurations that approximately satisfy a fraction of the constraints in a random constraint satisfaction problem. Finally, the zero-temperature free energy of the dilute continuous hardcore model is established in Section \ref{proof_hc}.

\section{Examples}\label{examples}
\subsection{Constraint satisfaction problems}

Random constraint satisfaction problems are formulated in terms of a random function $\theta$ which indicates whether a spin configuration satisfies the instance of a random criterion specified by $\theta$. More precisely, suppose that $\theta:\{-1,1\}^p \to\{-1,0\}$ is a random function and $(\theta_k)_{k\geq 1}$ are i.i.d. copies of $\theta$. Let $I(k,1), \ldots, I(k,p)$ be distinct indices chosen uniformly from $[N]$ for each $k\geq 1$, independent of $(\theta_k)_{k\geq 1}$. We say that the $k$th constraint is satisfied by the configuration $\sigma \in \{-1,1\}^N$ if \begin{align*}
    \theta_k(\sigma_{I(k,1)}, \ldots, \sigma_{I(k,p)}) = 0,
\end{align*}
otherwise, we say that the constraint is unsatisfied. In such problems, one is interested in the number of solutions $\sigma\in \{-1,1\}^N$ that simultaneously satisfy $\pi(\alpha N)$ constraints, i.e.,
\begin{equation} \label{eq:number_sol}
 \mathcal{N}_N=\#\bigl\{\sigma\in \{-1,1\}^N:\theta_k(\sigma_{I(k,1)},\ldots,\sigma_{I(k,p)})=0\text{ for all } 1\leq k\leq \pi(\alpha N)\bigr\}.   
\end{equation}
In general, the computation of $\mathcal{N}_N$ is an extremely formidable task. We address a related but simpler question for which we consider the set of $\sigma$'s that satisfy approximately $\lfloor t\pi(\alpha N)\rfloor$ many equations on average for $0 < t< 1$, namely, for $0<\epsilon<1$, 
\begin{align}\label{approx_num_sol}
  \mathcal{AN}_{N,\epsilon}(t):=\#\Bigl\{\sigma\in \{-1,1\}^N:\frac{1}{\pi(\alpha N)}\sum_{k=1}^{\pi(\alpha N)}\big(1+\theta_k(\sigma_{I(k,1)},\ldots,\sigma_{I(k,p)}) \bigr)\in (t-\epsilon,t + \epsilon) \Bigr\} . 
\end{align}
The following result establishes the logarithmic behavior of $\mathcal{AN}_{N,\epsilon}$. Observe that by replacing $\theta$ with $-\theta$, we may define $F(\beta)$ for all $\beta \in \R$.
\begin{theorem}\label{csp:thm1}
    Let $\alpha p (p-1) \le 1$. Let $F(\beta): = \cP_{\upnu,\beta}(\lambda_{\upnu,\beta})$. If $F$ is differentiable at some $\beta \in \R$, then almost surely,
\begin{align*}
	\lim_{\epsilon\downarrow 0}	\limsup_{N\to \infty}\Bigl|\frac{1}{N}\log\mathcal{AN}_{N,\epsilon}(1+\alpha^{-1} F'(\beta))- \bigl(\log 2 + F(\beta)-\beta F'(\beta)\bigr) \Bigr|=0.
	\end{align*}
 Additionally, if $\beta_1 \leq \beta_2$ are two points of differentiability of $F$, we have \begin{enumerate}
     \item $F'(\beta_1) \leq F'(\beta_2)$,
     \item $F(\beta_1) - \beta_1 F'(\beta_1) \geq F(\beta_2) - \beta_2 F'(\beta_2)$ if $\beta_1 > 0$, and
     \item $F(\beta_1) - \beta_1 F'(\beta_1) \leq F(\beta_2) - \beta_2 F'(\beta_2)$ if $\beta_2 < 0$.
 \end{enumerate}
\end{theorem}

Note that under the assumption that $\theta$ is symmetric, i.e., $\theta(x)=\theta(-x)$ for all $x\in \{-1,1\}^p$, $\lambda_{\upnu,\beta}$ is a Dirac measure at the zero function and the differentiability of $F(\beta)$ can be established. In particular, let $\phi$ be a deterministic and non-constant symmetric function and $(g_i)_{i\leq p}$ be i.i.d. samples from an arbitrary distribution supported on $\{-1,1\}$. If we take \begin{align}\label{sym_theta}
    \theta(x_1, \ldots, x_p) = \phi(g_1 x_1, \ldots,g_p x_p),
\end{align}
then an explicit expression for the limit of $\mathcal{AN}_{N,\epsilon}$ can be obtained.

\begin{cor}\label{csp:cor1}
  Let $\alpha p (p-1)\le 1$. If $\theta$ is symmetric, then $F (\beta)$ is differentiable for any $\beta \in \R$. Furthermore, if $\theta$ is of the form \eqref{sym_theta}, then for any $0<t< 1,$ almost surely,
\begin{align*}
	\lim_{\epsilon\downarrow 0}	\limsup_{N\to \infty}\Bigl|\frac{1}{N}\log\mathcal{AN}_{N,\epsilon}(t)-\Bigl(\log 2+\alpha(1-t)\log \frac{1-\varrho}{1-t}+\alpha t\log \frac{\varrho}{t}\Bigr)\Bigr|=0,
\end{align*}
where $\varrho:=\upnu^{\otimes p}\bigl(\phi(x_1,\ldots,x_p)=-1\bigr) \in (0,1)$.
\end{cor}

We present the proof of Theorem \ref{csp:thm1} and Corollary \ref{csp:cor1} in Section \ref{sec6}.

\subsubsection{Perceptron model}
The perceptron model is a type of constraint satisfaction problems, for which one is interested in estimating the number of spin configurations that lie in the intersection of a collection of random sets. For its formulation, consider a random set $A$ in $\mathbb{R}^p$ and take $\theta(x)=\mathbbm{1}_A(x)-1$ for $x\in \{-1,1\}^p$. Two interesting examples are
\begin{equation} \label{asymper}
 A=\Bigl\{x\in \{-1,1\}^p:\sum_{i=1}^pg_i x_i\leq \kappa\Bigr\}   
\end{equation}
and 
\begin{align}\label{symper}
A=\Bigl\{x\in \{-1,1\}^p:\Bigl|\sum_{i=1}^pg_ix_i\Bigr|\leq \kappa\Bigr\}\ \ \mbox{or}\ \ \Bigl\{x\in\{-1,1\}^p:\Bigl|\sum_{i=1}^pg_ix_i\Bigr|\geq \kappa\Bigr\},
\end{align}
where $g_1, \ldots, g_p$ are i.i.d.\ samples drawn from some distribution and $\kappa >0$ is a threshold variable.
The latter in particular is known as the symmetric perceptron model. Generally, the fundamental question in the perceptron model is concerned with the behavior of the solution space as defined in \eqref{eq:number_sol} such as its logarithmic scaling-limit, see, for example, \cite[Research Problems 6.7.1 and 6.7.2]{book}. 
While a major progress remains missing, one can obtain the limiting free energy (see the definition in \eqref{freeandGSE}) associated to the perceptron model in a regime where the temperature and the sparsity parameters are essentially inversely proportional to each other, as a direct consequence of Talagrand's general result \cite[Theorem 6.4.13]{book}. 

In light of this direction, an application of Theorem \ref{thm1} also provides a formula for the limiting free energy in the diluted perception model in the presence of an external field $\psi(t) = \beta^{-1} h t$, $h>0$ for any choice of the random set $A$ under the assumption,
\begin{equation}\label{eq:high_temp_perc}
 \min(1, 6 \beta e^{4 (\beta + h)}) \alpha p(p-1) \leq 1.   
\end{equation}
Additionally, the corresponding ground state energy can be computed using Theorem~\ref{thm:GSE}, provided that the underlying graph is in the subcritical regime, $\alpha p(p-1) \leq 1$. When specialized to the symmetric model \eqref{symper} with vanishing external field, i.e., $h = 0$, Proposition~\ref{prop1} readily yields much simpler expressions: under \eqref{eq:high_temp_perc},
\begin{align*}
     F(\beta) & = \alpha \E \log \sum_{\sigma\in \{-1,1\}^p} \big(e^{-\beta} + (1-e^{-\beta})\1_A(\sigma)\big) - \alpha p \log 2
\end{align*}
and when $\alpha p(p-1)\leq 1,$
\begin{align*}
    \gse & = -\alpha  + \alpha \P(A\,\,\mbox{is nonempty}).
\end{align*}
We remark that it is an open problem to compute the asymptotic number of solutions \eqref{eq:number_sol} for the diluted perceptron model \eqref{asymper} and \eqref{symper} for all sparsity level $\alpha >0;$ Theorem~\ref{csp:thm1} and Corollary~\ref{csp:cor1}  provide some partial information on this question. In contrast, for the fully connected version of the perception model, more complete results are available for analogous \eqref{asymper} and \eqref{symper} including the logarithmic scaling-limit of the solution space, see \cite{ding2018capacity, huang2024capacity, perkins2021frozen, shcherbina2001rigorous}.

\subsubsection{\textit{K}-SAT model}
The $K$-SAT model is another fundamental example of the constraint satisfaction problems. Denoting $K=p$ and letting $(J^{(k)})_{k\geq 1}$ be i.i.d.\ copies of a random vector $J$ in $\{-1,1\}^p$ called the clauses, the aim is to understand whether there exist configurations $\sigma \in \{-1,1\}^N$ that satisfy the clauses $(J^{(k)})_{k\leq \pi(\alpha N)}$, i.e., $(\sigma_{I(k,1)}, \ldots, \sigma_{I(k,p)}) \neq J^{(k)}$ for all $k\leq \pi(\alpha N)$. The constraint function $\theta$ takes the following form $$\theta(x_1, \ldots, x_p) = -\1_{(x_1, \ldots, x_p) \neq J} = -\prod_{i=1}^p \frac{1 + x_iJ_i}{2}.$$  
The replica symmetric free energy for the $K$-SAT model at high temperature was established in \cite{Talagrand_high_temp}, and subsequently improved in \cite[Chapter 6]{book}. The model was further studied in \cite{Montanari-Shah}, where the free energy was expressed as the evaluation of a Bethe functional at the unique solution of a system of belief propagation equations for a regime of $\alpha$ independent of temperature, more precisely, $\alpha <\alpha_*: =  2 p^{-1}\log p(1+o_p(1))$, which is much larger than our subcritical regime $\alpha \leq p^{-2}(1+o_p(1))$. However, our general framework allows for the possibility that the constraints $(J^{(k)})_{k\geq 1}$ are sampled independently from an arbitrary distribution on $\{-1,1\}^p$, in contrast to the traditional $K$-SAT model, where the constraints are typically sampled uniformly at random from the same space.

Thanks to Corollary \ref{cor2}, the free energy for the $K$-SAT model can now be expressed in terms of the unique solution to certain RDE (equation~\eqref{RDE:k-sat} below) for all $\alpha \leq (p(p-1))^{-1}$. This contrasts with the Parisi-type variational formula provided in \cite{EJP2963}, which admitted a unique minimizer only under a high-temperature condition.
 To write down the free energy, let us first note that for $\beta < \infty$, \begin{align}\label{pos_ksat}
    e^{\beta \theta (x_1, \ldots, x_p)} = 1 + (e^{-\beta}-1)\prod_{i=1}^p \frac{1 + x_iJ_i}{2} ,
\end{align}
so the Franz-Leone identity \eqref{identity} is satisfied by this model. Thus, from Corollary \ref{cor2}, the free energy of the $K$-SAT model subject to the external field $\psi(t) = ht$ for any $\alpha, \beta, h > 0$ satisfying $\min(1,6\beta e^{4\beta(h+1)})\alpha p (p-1) \leq 1$ is given by\begin{align*}
     F(\beta) & = -\log 2 + \E \log \sum_{\varepsilon \in \{-1,1\}}e^{\beta h\varepsilon }\prod_{k\leq \pi(\alpha p)}\Big(1+ 2^{-p}(e^{-\beta}-1)(1+J^{(k)}_p\varepsilon) \prod_{i\leq p-1}(1+ J^{(k)}_im_{k,i})\Big)\\
            & \qquad - \alpha (p-1)\E \log \Big(1 + 2^{-p}(e^{-\beta}-1)\prod_{i\leq p}(1 + J_im_i)\Big),
    \end{align*}
    where $(m_i)_{i\geq 1}$, $(m_{k,i})_{k,i\geq 1}$ are i.i.d.\ copies of $m$ that satisfies \begin{align}\label{RDE:k-sat}
        m \stackrel{d}{=} \tanh \Big[ \sum_{k\leq \pi(\alpha p)}\tanh^{-1}\frac{(e^{-\beta}-1)J^{(k)}_p\prod_{i\leq p-1}(1+J^{(k)}_im_{k,i})}{2^p+(e^{-\beta}-1)\prod_{i\leq p-1}(1+J^{(k)}_i m_{k,i})} +\beta h\Big].
    \end{align}
    
Besides investigating the existence of a solution to the $K$-SAT problem, one is also interested in finding the maximum number of clauses that can be satisfied by a Boolean configuration $\sigma \in \{-1,1\}^N$. This is known as the MAXSAT problem in the computer science literature \cite{makarychev_et_al:DFU.Vol7.15301.287} and can be interpreted as the ground state of the $K$-SAT problem with zero external field. From Proposition \ref{Parisiformula}, we obtain that for $\alpha \leq (p(p-1))^{-1}$, the ground state  is given by 
\begin{align*}
	\begin{split}
	\gse & = \E \sup_{(\sigma,s) \in \{-1,1\}^{\pi(\alpha p)\times (p-1) + 1}}\Big[\sum_{k\leq \pi(\alpha p)}\Big(\theta_k(\sigma_{k,1}, \ldots, \sigma_{k,p-1},s) + \sum_{i=1}^{p-1}X_{k,i}(\sigma_{k,i})\Big) + hs\Big] \\
 & \qquad - \alpha(p-1)\E \sup_{\sigma \in \{-1,1\}^p} \Big(\theta(\sigma_1, \ldots, \sigma_p) + \sum_{k=1}^p X_k(\sigma_k)\Big),
\end{split}
\end{align*}
where $(X_k)_{k\leq p}$, $(X_{k,i})_{k,i\geq 1}$ are i.i.d.\ copies of $X:\{-1,1\}\to \R$ whose distribution is a unique solution of the following distributional identity. 
\begin{align*}
    &\{X(t)\}_{t\in \{-1,1\}}\\
    & \stackrel{d}{=} \left\lbrace\sup_{\sigma \in \{-1,1\}^{\pi(\alpha p) \times (p-1) }}\Big[\sum_{k\leq \pi(\alpha p)}\Big(\theta_k(\sigma_{k,1}, \ldots, \sigma_{k,p-1},t)+ \sum_{i=1}^{p-1} X_{k,i}(\sigma_{k,i})\Big) + ht\Big]\right. \\
    & \qquad\left.- \sup_{(\sigma,s) \in \{-1,1\}^{\pi(\alpha p)\times (p-1) + 1}}\Big[\sum_{k\leq \pi(\alpha p)}\Big(\theta_k(\sigma_{k,1}, \ldots, \sigma_{k,p-1},s)+ \sum_{i=1}^{p-1} X_{k,i}(\sigma_{k,i})\Bigr) + hs\Big]\right\rbrace_{t\in \{-1,1\}}.
\end{align*}
In the absence of the external field, it can be readily verified that $X \equiv 0$ is the unique solution to the above distributional equation, and consequently, the limiting ground state energy is equal to zero.

Note that  the $K$-SAT model is not symmetric since $\theta(x) = -1$ for some $x \in \{-1,1\}^p$ does not necessarily imply $\theta(-x) =-1$. 
The NAE-$K$-SAT model is a symmetrized version of the $K$-SAT model, where for $J$ a random vector in $\{-1,1\}^p$, we let \begin{align*}
    \theta(x_1, \ldots, x_p) = -\1_{(x_1, \ldots,x_p)\neq \pm J}. 
\end{align*}
The NAE-$K$-SAT model has been extensively studied, with the threshold for the existence of solutions being rigorously established for the model on Erd\"os-R\'enyi graphs in \cite{achlioptas2003randomksatmomentssuffice, Coja_Oghlan_2016, Coja_Oghlan_2012} and on random regular graphs in \cite{ding2013satisfiabilitythresholdrandomregular, sly2023number}. Noting that the activation function for the NAE-$K$-SAT model satisfies the symmetry property, Proposition \ref{prop1} yields the following simple expressions for the free and ground state energies when $\alpha$ is in the subcritical regime and the external field is absent:
\begin{align*}
    F(\beta) & = \alpha \log [1+2^{-(p-1)}(e^{-\beta}-1)]\;\; \text{and}\;\;\gse = 0.
\end{align*}

\subsection{Diluted models with Ising \textit{p}-spin interactions}
The diluted Ising model with $p$-spin interactions has the Hamiltonian \eqref{eq:orig hamiltonian} that comprises interactions of $p$-tuples of spins arising from an underlying $p$-uniform hypergraph, \begin{align*}
    \theta(x_1, \ldots, x_p) = J x_1 \cdots x_p \;\; \text{and}\;\; \psi(x) = hx
\end{align*}
for $(x_1,\ldots,x_p)\in\{-1,1\}^p$ and $x\in \{-1,1\},$
where $J$ is a random variable with finite second moment.
The ferromagnetic Ising model corresponds to $J\equiv 1$ and $p=2$, which has been studied in the case of locally tree-like graphs in \cite{Dembo_2010, Dembo_2013, Dommers_2010}, where the Bethe prediction for the free energy was rigorously established for all sparsities $\alpha >0$, temperatures $\beta \geq 0$ and external fields $h \in \R$. A crucial step in these works involve showing a non-uniform decay of correlation of the root spin in the presence of an external field by using tools such as the Griffith-Hurst-Sherman and the Griffith-Kelly-Sherman inequalities. However, these tools are not available for spin glass models, and consequently, our results recover the same free energy for a regime of parameters smaller than those prescribed in these works.

The model with $J$ a symmetric random variable is known as the diluted $p$-spin model. Observe that $\theta$ satisfies the Franz-Leone identity \eqref{identity} with $a = \cosh(\beta J), b = \tanh(\beta J)$ and $f(x) = x$. When $J$ is supported on $[-1,1]$, Corollary \ref{cor2} implies that in the regime $\alpha p (p-1)\min(1,6\beta e^{4\beta(h+1)}) \leq 1$, the free energy is given as \begin{align*}
    F(\beta) & = -\log 2 + \alpha E \log \cosh(\beta J) - \alpha (p-1)\E \log \Big[1+\tanh(\beta J)\prod_{i\leq p}m_i\Big] \\
    & + \E \log\Big[e^{\beta h} \prod_{k\leq \pi(\alpha p)}\Big(1+\tanh(\beta J_k)\prod_{i\leq p-1}m_{k,i}\Big) + e^{-\beta h} \prod_{k\leq \pi(\alpha p)}\Big(1-\tanh(\beta J_k)\prod_{i\leq p-1}m_{k,i}\Big)\Big].
\end{align*}
Here, as usual, $(m_k)_{k\geq 1}, (m_{k,i})_{k, i\geq 1}$ are independent samples from the unique law that satisfies \begin{align*}
        m & \stackrel{d}{=} \tanh \Big( \sum_{k\leq \pi(2 \alpha)} \tanh^{-1}\Big(\eta\tanh(\beta) m_k\Big) + h\Big).
    \end{align*}
Furthermore, when $p$ is an even integer, $\theta$ is symmetric and Proposition \ref{prop1} readily implies that in the absence of an external field, the free and ground state energies for the diluted $p$-spin model can be simplified as
\begin{align*}
    F(\beta) & = \alpha \E \log \cosh (\beta J)\;\; \text{and}\;\;\gse = \alpha \E |J|.
\end{align*}
We remark that the diluted $2$-spin model with $J$ supported in $[-1,1]$ was considered earlier in \cite{Guerra_2004}, where the replica symmetric free energy was studied in the regime of parameters $2\alpha \E \tanh^2(\beta J) <1$, which has a better dependence on $J$ than our regime. Nevertheless, if $\beta=\infty,$ while their result allows to obtain the same ground state energy mentioned above for $\alpha<1/2$, ours extends to $\alpha=1/2$.

\subsection{Diluted models with general spins}
In the previous two sections, we outlined the applicability of our results to some well-known models where the spins take values $\{-1,1\}$. In this section, we mention some models where the spins are not Ising and Franz-Leone identity \eqref{identity} does not hold, but our results from Theorem \ref{thm1} still apply. 

\subsubsection{Potts model}
The $q$-Potts model is a generalization of the Ising model, where the spins take values in $\Sigma = \{1, \ldots, q\}$ for some positive integer $q\geq 2$. The Hamiltonian for the diluted version of the Potts spin glass is defined as \begin{align*}
        -H_N(\sigma) = -\sum_{k\leq \pi(\alpha N)} J_k\1_{\sigma_{I(k,1)} = \sigma_{I(k,2)}} + h\sum_{i=1}^N\1_{\sigma_i = 1},\,\; \sigma \in \Sigma^N,
    \end{align*}
where $h\in \R$ and $(J_k)_{k\geq 1}$ are i.i.d. copies of a random variable $J$ that satisfies $\E J^2 < \infty$. The external field imposes a cost of $h$ on configurations per coordinate attaining the value 1. As usual, $J \equiv 1$ or $J\equiv -1$ corresponds to the ferromagnetic or antiferromagnetic model respectively.

Besides being an important model in statistical physics, the Potts model shares an important connection with combinatorics; the antiferromagnetic Potts model at zero temperature is equivalent to the graph coloring problem, which studies the existence of a coloring of the vertices so that adjacent vertices have different colors, see for example \cite{coja2019bethe, coja2019spin} and the references therein. In particular, \cite{coja2016potts, Contucci_2013} establish that the antiferromagnetic $q$-Potts model on the Erd\"os-R\'enyi graph is replica symmetric for all $\beta \geq 0$ as long as $\alpha \leq (q-1/2)\ln q -1 +o_q(1)$. For values of $\alpha$ exceeding this threshold, bounds were obtained for the critical temperature at which a phase transition from the replica-symmetric to the replica-symmetry breaking phase occurs. Later in \cite{coja2018information}, the question of precisely identifying the critical temperature for this phase transition was settled.

In contrast to the antiferromagnetic model, the ferromagnetic $q$-Potts model on graphs that converge locally to $d$-regular trees is believed to be replica symmetric at all temperatures $\beta \geq 0$ for all $q, d\geq 3$. However, this has been rigorously shown when $d$ is even \cite{dembo_2012,Dembo_2013}, or when the external field is absent \cite{bencs2023random, helmuth2023finite}. In line with the study of ferromagnetic Ising models on locally tree-like graphs, where the Gibbs measure was shown to converge locally weakly to a symmetric mixture of the Ising measure on the infinite tree with $+1$ and $-1$ boundary conditions, the weak limit of the Gibbs measure of the ferromagnetic Potts on locally tree-like graphs was shown to be the Potts measure on the corresponding infinite tree with free or wired boundary conditions or a mixture thereof, see \cite{basak2012ferromagnetic, montanari2012weak} for details.

 It is also natural to consider the Potts spin glass model, where the interactions are independent random variables. In the mean-field setting this has been studied previously, in particular, Parisi-type variational formulas for the free energy are known through the papers \cite{bates2023parisiformulabalancedpotts, c6f03298-21dc-337c-b65d-4c74d6e7603a}. To the best of our knowledge, the diluted Potts spin glass has not been analyzed previously. Theorem \ref{thm1} yields that in the regime
    \begin{equation}\label{eq:potts_rs_regime}
        \min \big (1, 6 \beta e^{4 \beta |h|} \E |J|  e^{4 \beta |J|} \big) \alpha  \leq 1/2,
    \end{equation} 
    the Potts spin glass is replica symmetric, with the corresponding free energy being\begin{align*}
        F(\beta) & = \E \log \Big[e^{\beta h}\prod_{k\leq \pi(2\alpha)}\Big(1+X_k(1)(e^{\beta J_k}-1)\Big) + \sum_{\varepsilon \in [q]\setminus \{1\}}\prod_{k\leq \pi(2\alpha)}\Big(1 + X_k(\varepsilon)(e^{\beta J_k} - 1) \Big)\Big]\\
        & \qquad -\alpha \E \log \Big[1 + (e^{\beta J}-1)\sum_{\varepsilon \in [q]}X_1(\varepsilon)X_2(\varepsilon)\Big] - \log q,
    \end{align*}
    where $(X_k)_{k\geq 1}$ are i.i.d. functions on $[q]$ sampled according to the unique law that satisfies \begin{align*}
        X(t) \stackrel{d}{=}\frac{e^{\beta h\1_{t = 1}}\prod_{k\leq \pi(2\alpha)}(1+X_k(t)(e^{\beta J_k}-1))}{e^{\beta h}\prod_{k\leq \pi(2\alpha)}(1+X_k(1)(e^{\beta J_k}-1)) + \sum_{s\in [q]\setminus \{1\}}\prod_{k\leq \pi(2\alpha)}(1+X_k(s)(e^{\beta J_k}-1))}.
    \end{align*}
   Under the assumption $\alpha \le 1/2$, the ground state is given by Theorem \ref{thm:GSE} as \begin{align*}
        \gse & = \E \sup_{(\varepsilon, \sigma)\in [q]^{\pi(2\alpha) + 1}}\Big[h\1_{\varepsilon = 1} + \sum_{k\leq \pi(2\alpha)} \big(J_k\1_{\sigma_k = \varepsilon} +  Y_k(\sigma_k)\big)\Big] \\
        & \qquad - \alpha \E \sup_{\sigma \in [q]^2}\Big[J\1_{\sigma_1 = \sigma_2} + Y_1(\sigma_1) + Y_2(\sigma_2)\Big],
    \end{align*}
     where $(Y_k)_{k\geq 1}$ are i.i.d.\ functions on $[q]$ sampled according to the unique law that satisfies \begin{align*}
     Y(t) & \stackrel{d}{=} \sup_{\sigma \in [q]^{\pi(2\alpha)}}\Big[h\1_{t = 1} + \sum_{k\leq \pi(2\alpha)}\big(J_k\1_{\sigma_k = t} + Y_k(\sigma_k)\Big)\Big] \\
     & \qquad -  \sup_{(s,\sigma) \in [q]^{\pi(2\alpha) + 1}}\Big[h\1_{s = 1} + \sum_{k\leq \pi(2\alpha)}\big(J_k\1_{\sigma_k = s} + Y_k(\sigma_k)\Big)\Big].
     \end{align*}
 In the absence of an external field, one can verify that under \eqref{eq:potts_rs_regime},  the Dirac measure at the zero function is a fixed point of the operator $\T_{\nu,\beta}$ and consequently, the free energy is expressed as a much simpler formula \begin{align}\label{potts}
        F(\beta) = \alpha \E \log [e^{\beta J} + q-1] - \alpha \log q.
    \end{align}
    In a similar fashion, if $\alpha \le 1/2$, one obtains that the ground state energy in the absence of an external field is \begin{align*}
        \gse & = \alpha \E \max(J,0).
    \end{align*}
As a simple corollary, note that the expression for the free energy obtained in \cite{Contucci_2013} can be obtained from \eqref{potts} by setting $J = -1$. The price we pay for considering a disordered model is that the replica symmetric free energy is valid over a stricter regime for $\alpha$ and $\beta$, and we leave it as a problem for future study to improve this regime.

\subsubsection{XY model}  In the XY model, the spins associated with each site take values in the unit circle $\Sigma^1$, and the spin vectors at adjacent sites
interact via their inner product. When the interactions are ferromagnetic, the  XY model (or more generally, the spin $O(n)$ model)
is one of the classical statistical physics models with continuous spins which has been well-studied on Euclidean lattices. Due to the presence of continuous symmetry, the XY model exhibits remarkably different behaviors compared to the standard (discrete spin) Ising model. For example, as a consequence of the celebrated Mermin-Wagner Theorem \cite{mermin1967absence, mermin1966absence}, the XY model does not have an orientational long-range order in dimension two at low temperature. However, the two-point correlation function still undergoes certain phase transition in temperature which is known as Berezinskii–Kosterlitz–Thouless (BKT) transition \cite{berezinskii1971destruction, kosterlitz1973ordering}. In high temperature, it decays exponentially whereas it only shows algebraic decay in low temperature \cite{frohlich1981kosterlitz, mcbryan1977decay}. We refer to the book \cite[Chapter 9]{friedli2017statistical} and the lecture note \cite{peled2019lectures} for a comprehensive list of references on this subject. 

In the mean-field setting, the disordered version of the XY model is a special case of the mixed $p$-spin models with vector spins.  They were studied in \cite{panchenko2018free} where a Parisi-like formula for their limiting free energy was obtained. To the best of our knowledge, the diluted version of XY spin glasses has not been studied before rigorously. After representing the $i$th vector spin on the unit circle as $(\cos (2\uppi \sigma_i), \sin (2\uppi \sigma_i))$ through the angle $\sigma_i \in \Sigma:=[0, 1]$, the Hamiltonian of the diluted XY spin glass model can be expressed as
\begin{align*}
        H_N(\sigma) = \sum_{k\leq \pi(2\alpha)}J_k\cos(2\uppi(\sigma_{I(k,1)} - \sigma_{I(k,2)})) + h\sum_{i=1}^N \cos(2\uppi \sigma_i).
    \end{align*}
    Here $h\in \R$ and $(J_k)_{k\geq 1}$ are drawn independently from a distribution with finite second moment (to ensure that \eqref{bdd theta} holds), and the Poisson number of terms given by $\pi(2\alpha)$ is not to be confused with the numerical constant $\uppi$. The measure $\nu$ is the uniform measure on $[0,1]$ (henceforth, for this section, we drop $\nu$ from the subscripts). When $J_k = 1$, we obtain the ferromagnetic XY model, see \cite{AIZENMAN1980281} for details. For $J_k$ taking both positive and negative values, we obtain the XY model in a spin glass phase. In either case we can compute the free energy when $(\alpha, \beta)$ satisfy $\min(1,6\beta \E|J|e^{4\beta|J|})\alpha p(p-1) \leq 1$. This is given by 
    \begin{align*}
        \cP_\beta(\lambda_\beta) & = \E \log \int e^{\beta h\cos(2\uppi t)}\prod_{k\leq \pi(2\alpha)}\exp\Big(\beta J_k\cos(2\uppi(\rho_k - t))\Big) Y_k(\rho_k)d\rho dt \\
        & \qquad - \alpha \E \log \int \exp\Big(\beta J \cos(2\uppi(\rho_1 - \rho_2))\Big)Y_1(\rho_1)Y_2(\rho_2)d\rho,
    \end{align*}
    where $(Y_k)_{k\geq 1}$ are i.i.d. samples from $\lambda_\beta$, the fixed point of the distributional equation \begin{align*}
        \big(Y(t)\big)_{t\in [0,1]} \stackrel{d}{=} \Big(   \frac{\int \exp\big[ \sum_{k\leq \pi(2\alpha)} \beta J_k\cos(2\uppi(\rho_k - t)) + \beta h\cos(2\uppi t)\big]   Y_k(\rho_k) d\rho}{\int \exp\big[\sum_{k\leq \pi(2\alpha)} \beta J_k\cos(2\uppi(\rho_k - s))  + \beta h\cos(2\uppi s)\big]Y_k(\rho_k) d\rho ds} \Big)_{t\in [0,1]}.
    \end{align*}
    Note that the fixed point equation for the distributional equation above is unique by virtue of Theorem \ref{thm1} and a change of variables $Y(t) = e^{\beta X(t)}$. It is easy to verify that $\lambda_\beta$ is the Dirac measure on $Y(t) \equiv 1$ on the interval $[0, 1]$. To see this, note that for any $1$-periodic function $g$, $\int_0^1 g(x+ a) dx  = \int_0^1 g(x) dx$ for any $a \in \R$. Hence, for any $t$,
\[ \int_0^1 e^{\beta J\cos(2\uppi(x - t))} d x = \int_0^1 e^{\beta J\cos(2\uppi x )} d x.  \]
Thus, for $\alpha \leq 1/2$, we have the free energy at any temperature $\beta<\infty$ to be given by 
\begin{align*}
    \cP_\beta(\lambda_\beta) = \alpha \E \log  \int_0^1 e^{\beta J\cos(2\uppi x )} d x.
\end{align*}

Let $I_0$ be the modified Bessel function of the first kind, which is given by
\[ I_0(z) = \frac{1}{\uppi} \int_0^{\uppi} e^{z \cos(x)} dx = 2 \int_0^{1/2} e^{z \cos(2 \uppi x)} dx, \]
where the latter expression follows from a simple change of variable. Therefore,
\begin{align*}
   \int_0^1 e^{\beta J\cos(2\uppi x )} d x &=   \int_0^{1/2} e^{\beta J\cos(2\uppi x )} d x +  \int_0^{1/2} e^{-\beta J\cos(2\uppi x )} d x =  \frac{1}{2} \big(I_0(\beta J) + I_0( -\beta J)\big).  
\end{align*}
Thus the free energy can be expressed in terms of the modified Bessel functions as \begin{align*}
    \cP_\beta(\lambda_\beta) & = \alpha \E \log [I_0(\beta J) + I_0(-\beta J)] - \alpha \log 2.
\end{align*}
For $\alpha \leq 1/2$, the ground state of this model can be computed using Theorem \ref{thm:GSE}. To that end, we first note that the Dirac measure on the function $X\equiv 0$, is the unique solution to the fixed point equation \begin{align*}
    X(t) \stackrel{d}{=} & \sup_{\rho \in \Sigma^{\pi(2\alpha)}}\Big(\sum_{k\leq \pi(2\alpha)}J_k\cos(2\uppi(\rho_k - t)) + \sum_{k=1}^{\pi(2\alpha)}X_k(\rho_k)\Big)\\
	& \qquad - \sup_{(s,\rho) \in \Sigma^{\pi(2\alpha) + 1}}\Big(\sum_{k\leq \pi(2\alpha)}J_k\cos(2\uppi(\rho_k - s)) + \sum_{k=1}^{\pi(2\alpha)}X_k(\rho_k)\Big)
\end{align*}
and thus, the ground state is obtained to be \begin{align*}
    \cP_\infty(\lambda_\infty) & = \alpha \E |J|.
\end{align*}

\subsubsection{Continuous hardcore model} 
Let $\G$ be a finite graph with vertex set $\V$. Denote by $\mathcal{I}(\G)$ the set of independent sets of $\G$, i.e., $\mathcal{I}(\G)$ is the collection of all $I \subset \V$ such that no two vertices in $I$ are neighbors in $\G$. Given the fugacity (or activity) parameter $\eta > 0$, the (discrete) hardcore model is a probability measure on  $\mathcal{I}(\G)$ that assigns to $I \in \mathcal{I}(\G)$ a mass proportional to $\eta^{|I|}$. Equivalently, the hardcore model can be viewed as the Gibbs measure on $\{0, 1\}^\V$
with  mass function proportional to 
$\eta^{\sum_{v\in\V}\sigma_v} \1_{ \{ \sigma \in \mathfrak{X}_\G  \}}$, where 
 \begin{align*}
    \mathfrak{X}_\G: = \big\{\sigma: = (\sigma_v)_{v\in\V} \in \{0,1\}^\V: \sigma_u + \sigma_v \leq 1 \text{ for all edges }\{u,v\} \in \G\big\}.
\end{align*}

For the hardcore model, depending on the geometry of the underlying graph, there exists a critical value of the fugacity $\eta_c$, known as the uniqueness threshold, such that for $\eta < \eta_c$, a unique Gibbs state exists in the thermodynamic limit, while for $\eta > \eta_c$, multiple  Gibbs states coexist in the limit.
For a $d$-regular tree, it is well known that $\eta_c(d) = (d-1)^{d-1}/(d-2)^d$ (see \cite{kelly1985stochastic}), and the uniqueness threshold for any graph with maximum degree $d$ is at least $\eta_c(d)$ \cite{weitz2006counting}. The phase transition of the hardcore model has also been studied on Euclidean lattices \cite{BLANCA_CHEN_GALVIN_RANDALL_TETALI_2019, GALVIN_KAHN_2004, Sinclair2014SpatialMA}.
In the uniqueness regime, the limit of the free energy (log-partition function)  of the hardcore model on random $d$-regular graphs was computed in \cite{bandyopadhyay2006counting, bandyopadhyay2008counting}, and this result was later extended to more general locally tree-like graphs, including diluted Erd\"os-R\'enyi graphs, in \cite{Dembo_2013}.

The work \cite{gamarnik2017uniqueness} introduced a continuous version of the hardcore model, where the spins are allowed to take value in the interval $[0, 1].$ For a fugacity parameter $\eta >0$, let $\nu$ be the probability measure defined on $\Sigma = [0,1]$ in the following manner. For any Borel set $B \subset \Sigma$, set $\nu(B) = \upsilon_0^{-1}\int_B \eta^x dx$ for $\upsilon_0 := \int_0^1 \eta^x dx$. In particular, when $\eta = 1$, $\nu$ is the uniform measure on $[0,1]$. Given a finite graph $\G$, the continuous hardcore model is the probability measure on $[0, 1]^\V$ whose density is proportional to $ \1_{\mathfrak{P}_\G} (\sigma) \nu^{\otimes \V}(d\sigma)$
where  \begin{align*}
    \mathfrak{P}_\G: = \big\{\sigma:  = (\sigma_v)_{v\in\V}  \in [0,1]^\V: \sigma_u + \sigma_v \leq 1  \text{ for all edges }\{u,v\} \in \G\big\}.
\end{align*}

It was shown in \cite{gamarnik2017uniqueness} that, in contrast to the discrete hardcore model, the continuous model on $d$-regular trees does not exhibit a phase transition, meaning there exists a unique limiting Gibbs state for any value of the fugacity parameter $\eta > 0$. 
 The limiting free energy was also computed for all values of $\eta$ for any sequence of $d$-regular graph on $N$ vertices, with girth growing to infinity. Let us remark that even though the continuous model does not have a phase transition, it remains challenging to estimate, via deterministic algorithms, the volume of the polytope $\mathfrak{P}_\G$ or equivalently, the partition function of the continuous hardcore model when $\eta =1$. Recent works \cite{bencs2024approximating, gamarnik2023computing} show that for any graph with maximum degree $\Delta$, there is a polynomial-time deterministic algorithm to approximate the volume of $\mathfrak{P}_\G \cap [0, 1/2 + \alpha]^\V$, where $\alpha = O(\Delta^{-1}).$

As a consequence of Theorem \ref{thm1}, we can derive an explicit expression for the limiting free energy of the continuous hardcore model on the random graph $\G_N$, which has $N$ vertices and edges given by $\{(I(k,1), I(k,2))\}_{k \leq \pi(\alpha N)}$ for $\alpha \leq 1/2$ corresponding to the subcritical regime.
To see this, let $\X_0$ denote the set of all continuous density functions (w.r.t. the Lebesgue measure) on $[0,1]$. Consider the operator $\T_\infty^\h:\Pr_1(\X_0)\to \Pr_1(\X_0)$ defined as follows. For $\lambda \in \Pr_1(\X_0)$, let $(f_k)_{k\geq 1}$ be i.i.d. samples from $\lambda$ and $F_k(t) = \int_0^t f_k(x) dx$. Then 
        \begin{equation} \label{eq:fixpt_beta_infty_hardcore}
        \T_\infty^\h(\lambda) = \text{ law of }\left ( \frac{\eta^t\prod_{k\leq \pi(2\alpha)} F_k (1-t)  }{\int_0^1 \eta^s\prod_{k\leq \pi(2\alpha)} F_k (1-s) ds} \right)_{ t \in [0, 1]}.
        \end{equation}
Also, for $\lambda \in \Pr_1(\X_0)$, define \begin{align}\label{add:eq01}
 \cP_\infty^\h(\lambda) & = - \log \upsilon_0+ \E\log \int \eta^t\prod_{k\leq \pi(2\alpha)}F_k(1-t) dt - \alpha \E \log \int F_1(1-t)f_2(t) dt.
 \end{align}
\begin{theorem}\label{thm_hardcore}
Suppose that $\alpha\leq 1/2$ and $\G_N$ are the graph described above. Then for any $\eta > 0$, the following statements hold.
\begin{enumerate}
    \item There exists a unique $\lambda_\infty^\h \in \Pr_1(\X_0)$ such that $\T_\infty(\lambda_\infty^\h) = \lambda_\infty^\h$.
    \item As $N \to \infty$,  $N^{-1} \log \nu^{\otimes N}(\mathfrak{P}_{\G_N})$ converges to $ \cP_\infty^\h(\lambda_\infty^\h)$ in $L^1$.
\end{enumerate} 
\end{theorem}

The proof of this theorem is presented in Section \ref{proof_hc}. Two key observations about our result are the absence of a phase transition and the expression of the asymptotic volume of the polytope $\mathfrak{P}_\G$ as replica symmetric free energy formula, which hold true for random locally-tree like graphs. In the case of regular trees, the limiting Gibbs state is the unique solution to a certain first-order ODE (proved for $\eta = 1$ in \cite[Theorem 3.3]{gamarnik2017uniqueness} and conjectured for general $\eta > 0$ in \cite[Conjecture 5.2]{gamarnik2017uniqueness}). The price we pay for considering random tree-like graphs is that our result is true only for subcritically sparse graphs, where $\alpha \leq 1/2$. This, we believe, is a reasonable compromise because the limiting Gibbs state is the unique solution to a certain distributional equation, which, as we mention in the introduction, is challenging to establish. Perhaps it is true that our result holds for any $\alpha >0$, as in the case of the 2-spin Ising model considered in \cite{Dembo_2010}.

\section{The distributional operator}\label{distributional_operator}

There are two goals in this section. First, we will establish the uniqueness of the fixed point of the operator $\mathcal{T}_{\nu,\beta}$ stated in Theorem \ref{thm1}. Next, we will continue to show the tightness and continuity of the fixed point that will be used later when we turn to the proof of the asserted formula for the limiting free energy in the diluted model. 

To begin with, we recall two properties of the Wasserstein 1-metric that we shall use throughout this section. First of all, since $(\X,\|\cdot\|_\infty)$ is complete and separable, the metric space $(\Pr_1(\X),W_1)$ also enjoys the same topological properties.  Second, if $(\lambda _n)_{n\geq 1}\subset \Pr_1(\X)$ and $\lambda \in \Pr_1(\X
)$, then
\begin{align*}
	\lim_{n\to\infty}W_1(\lambda _n,\lambda )=0\;\; \Longleftrightarrow\;\;\lambda _n \text{ converges weakly in } \Pr_1(\X) \text{ to } \lambda
\end{align*}
in the sense that $\lambda _n\stackrel{d}{\to}\lambda $ and $\int \|f\|_\infty\lambda _n(df)\to \int \|f\|_\infty\lambda (df).$ These two results can be found, for example, in \cite[Theorems 6.9 and 6.18]{villani2008optimal}. 

\subsection{Existence and uniqueness of the fixed point}
The main result of this section is the following theorem.
\begin{theorem}\label{uniqueness}
	Suppose that $0 < \beta \leq \infty$ and \begin{align} \label{ht1}
		 4\alpha \beta \E \|\theta\|_\infty e^{2\beta \|\theta\|_\infty}p(p-1)<1 \;\;\text{or}\;\;\alpha p(p-1)\leq 1.
	\end{align}
    Then there exists a unique $\lambda_{\nu,\beta} \in \Pr_1(\X)$ such that $\T_{\nu,\beta}(\lambda_{\nu,\beta} )=\lambda_{\nu,\beta} $ and $\lambda_{\nu,\beta}$ satisfies \begin{align}\label{eq2}
		\begin{split}
			\int \|f\|_\infty\lambda_{\nu,\beta}(df)&\leq 2\alpha p \mathbb{E}\|\theta\|_\infty + 2 \|\psi\|_\infty, \\
			\int \mbox{\rm Lip}(f)\lambda_{\nu,\beta} (df)&\leq \alpha p\E \mbox{\rm Lip}(\theta) + \mathrm{Lip}(\psi).
		\end{split}
	\end{align}
\end{theorem}

The proof of this theorem is split into two parts corresponding to the subcritical regime, $\alpha p (p-1) \leq 1$, and the high temperature regime, $4\alpha \beta\E \|\theta\|_\infty e^{2\beta\|\theta\|_\infty} p(p-1) <1$.

Following \cite[Lemma 14]{aldous2005survey}, we establish the existence and uniqueness of the solution to a general distributional fixed point equation. Our result for the subcritical regime will follow immediately from this general result.
\begin{lemma}\label{gen_fixed_pt}
 Let $g: [0,1]\times \cup_{n\in \N} \X^{(p-1)n} \to \X$ be a continuous measurable function and $f_0 \in \X$ be a fixed element. Let $\T: \Pr_1(\X) \to \Pr_1(\X)$ be an operator defined as follows. For $\lambda\in \Pr_1(\X)$, \begin{align*}
     \T(\lambda) = \text{ Law of } \begin{cases}
     g(\xi, X_1, X_2, \ldots, X_{(p-1)\pi(\alpha p)}), & \mbox{if $\pi(\alpha p) \neq 0$},\\
     \delta_{f_0}, & \mbox{if $\pi(\alpha p) = 0$},
     \end{cases}
 \end{align*}
where $\pi(\alpha p)$ is a ${\rm Pois}(\alpha p)$ random variable, $\xi$ is ${\rm Unif}([0,1])$ distributed, and $(X_i)_{i\geq 1}$ are sampled according to $\lambda$, all independent of each other, and $\delta_{f_0}$ is the Dirac measure on $f_0$. If $\alpha p(p-1) \leq 1$, then $\T$ admits a unique fixed point.
\end{lemma}

\begin{proof}
First of all, we establish the existence of the fixed point. We begin by defining a random forest with vertex set $\V:= \{\varnothing\}  \cup(\cup_{r \ge 1} \N^r)$. We call the vertex $\varnothing$ the root and set the generation number of this vertex to be zero. For $r\geq 1$, every vertex in $\N^r$ is represented by a word of the form $\i = i_1 i_2 \cdots i_r$ for some $i_1, i_2, \ldots, i_r \in \N$ and we write $|\i| = r$ for the length of the word $\i$. For any $r\geq 0,$ denote $ \V_{\le r} = \{ \i \in \V: |\i| \le r\}$ and $ \V_{= r} = \{ \i \in \V: |\i| = r\}$. To construct our random forest, let $(\pi_\i (\alpha p))_{\i\in \V}$ be independent Poisson random variables with mean $\alpha p$. Starting from $\varnothing$, we place an edge (no edge) between $\varnothing$ and $ i$ for any $1\leq i\leq (p-1)\pi_{\varnothing}(\alpha p)$ if $\pi_{\varnothing}(\alpha p)\geq 1$ (respectively if $\pi_{\varnothing}(\alpha p)=0$). Recursively, for $r\geq 1$ and any $\i\in \mathbb{N}^r,$ we place an edge (no edge) between $\i$ and $\i j$ for $1\leq j\leq (p-1)\pi_{\i}(\alpha p)$ if $\pi_{\i}(\alpha p)\geq 1$ (respectively if $\pi_{\i}(\alpha p)=0$). The resulting random forest on the vertex set $\V$ is denoted by $\G$. Denote by $\G_\varnothing$ the connected component of $\G$ containing $\varnothing$, which is a random tree rooted at $\varnothing$. Denote by $\text{ht}(\G_\varnothing)$ the height of the tree $\G_\varnothing$. 
Note that $\mathcal{G}_{\varnothing}$ is a Galton-Watson tree with offspring distribution $(p-1)\pi(\alpha p).$ The assumption $\alpha p(p-1)\leq 1$ ensures that $\mathcal{G}_{\varnothing}$ is finite a.s. and consequently $\text{ht}(\G_\varnothing)<\infty$ a.s. 

Now, let $\lambda_0 \in \Pr_1(\X)$ be arbitrary. For each $\i \in \V$, let $(\xi_\i)_{\i \in \V}$ be i.i.d. ${\rm Unif}([0,1])$, independent of $(\pi_\i(\alpha p))_{\i\in \V}$. For any $r\geq 1$, define $(X^{(r)}_\i)_{i\in \V_{\leq r}}$ by letting $(X^{(r)}_\i)_{\i \in \V_{=r}}$ be i.i.d. samples from $\lambda_0$, independent of everything else, and if $(X_\i^{(r)})_{\i\in \V_{=r'}}$ is already defined for some $1\leq r'\leq r,$ we set $(X_\i^{(r)})_{\i\in \V_{=r'-1}}$ by 
 \begin{align*}
      X^{(r)}_\i =\begin{cases}
       g(\xi_\i, X_{\i 1}^{(r)}, X_{\i 2}^{(r)}, \ldots,  X_{\i (p-1)\pi_\i (\alpha p)}^{(r)}), & \mbox{if $\pi_\i(\alpha p) \neq 0$},\\
       f_0, & \mbox{if $\pi_\i(\alpha p)  = 0$}.
      \end{cases}
  \end{align*}
   Observe that by definition $X_\varnothing^{(r)}=X_\varnothing^{(\text{ht}(\G_\varnothing))}$ whenever $r\geq \text{ht}(\G_\varnothing),$ and thus it follows that a.s. $\lim_{r\to\infty}X_\varnothing^{(r)}=X_\varnothing^{(\text{ht}(\G_\varnothing))}$, which, in particular, is measurable with respect to the $\sigma$-field generated by $\{\pi_\i(\alpha p) , \xi_\i: \i \in \G_\varnothing\}$ and is independent of the law of $\lambda_0$. Set $X_\varnothing=X_\varnothing^{(\text{ht}(\G_\varnothing))}.$ Owing to the i.i.d. feature of $(\pi_\i(\alpha p))_{\i\in \V}$ and $(\xi_\i)_{\i\in \V}$, we have that $(X_j^{(r+1)})_{j\in \N}$ are i.i.d. copies of $X_\varnothing^{(r)}$ and thus, their limits, $(X_j)_{j\in \N}$, are i.i.d. copies of $X_\varnothing$. On the other hand, by construction, 
  \begin{align*}
  X_\varnothing^{(r)} = \begin{cases}
      g(\xi_\varnothing, X_1^{(r)}, \ldots, X_{(p-1)\pi_\varnothing(\alpha p)}^{(r)}), & \mbox{if $\pi_\varnothing(\alpha p)  \neq 0$},\\
       f_0, & \mbox{if $\pi_\varnothing(\alpha p)  = 0$},
  \end{cases}
  \end{align*} so taking $r\to \infty$, we have that $X_\varnothing \stackrel{d}{=} \T(X_\varnothing)$, establishing that the operator $\T$ admits a fixed point. Finally, the uniqueness follows since for any fixed point $\lambda$, the $X_\varnothing$ associated to $\lambda$ has the same law as $\lambda$ and conditioned on $\{\pi_\i(\alpha p), \xi_\i: \i \in \G_\varnothing\}$, $X_\varnothing$ is uniquely determined by recursion from the leaves upwards to the root.
\end{proof}

\begin{proof}[Proof of Theorem \ref{uniqueness} (Subcritical regime)]
The existence and uniqueness of the fixed point follow immediately from Lemma \ref{gen_fixed_pt} by taking  \begin{align*}
    g(\xi, X_1, X_2, \ldots, X_{(p-1)\pi(\alpha p)}) = 
    \begin{cases} 
    \T_{\nu, \beta}(X_1, \ldots, X_{(p-1)\pi(\alpha p)}), & \beta <\infty\\
    \T_\infty(X_1, \ldots, X_{(p-1)\pi(\alpha p)}), & \beta = \infty,
    \end{cases}
\end{align*}
where $\xi$ captures the randomness of $(\theta_k)_{k\geq 1}$ and \begin{align*}
    f_0(t) = \begin{cases}
        \psi(t) - \beta^{-1}\log \int e^{\beta \psi(s)}\nu(ds), & \beta <\infty,\\
        \psi(t) - \sup_{s\in \Sigma}\psi(s), & \beta = \infty.
    \end{cases}
\end{align*}
As for \eqref{eq2}, it holds obviously due to \eqref{l_infty bound} and \eqref{lip_bound}.
\end{proof}

We now turn to the high temperature regime case. Although the proof is fairly standard (see \cite[Chapter 6]{book}), we include it below for completeness.
\begin{proof}[Proof of Theorem \ref{uniqueness} (High temperature regime)]
 For any $\lambda ,\lambda '\in \Pr_1(\X),$ let $\Pi$ be an arbitrary coupling of $\lambda$ and $\lambda'$ and let $(f_k, f'_k)_{k\geq 1}$ be i.i.d. samples from $\Pi.$ Conditional on $\pi(\alpha p)$, for $0\leq i \leq \pi(\alpha p)$, let \begin{align*}
        \Delta_i(t): = T_{\nu, \beta, \pi(\alpha p)}(f_1, \ldots, f_i, f_{i+1}', \ldots, f'_{(p-1)\pi(\alpha p)})(t),
    \end{align*}
    where $T_{\nu,\beta,\pi(\alpha p)}$ is defined in \eqref{opeartor:eq1}. Thus, the law of $\Delta_0$ is $\T_{\nu,\beta}(\lambda')$ while that of $\Delta_{(p-1)\pi(\alpha p)}$ is $\T_{\nu,\beta}(\lambda)$.  We claim that for all $1\leq j\leq \pi(\alpha p)$ and $(j-1)(p-1)+1\leq i \leq j(p-1)$,
    \begin{align}\label{almost_contraction}
        \|\Delta_{i-1}- \Delta_i\|_\infty \leq 4\beta\E \|\theta\|_\infty e^{2\beta\|\theta\|_\infty} \|f_i- f_i'\|_\infty.
    \end{align}
    Note that \begin{align*}
        W_1(\T_{\nu, \beta}(\lambda), \T_{\nu, \beta}(\lambda')) \leq \E \|\Delta_0 -\Delta_{(p-1)\pi(\alpha p)}\|_\infty \leq \E\sum_{j=1}^{\pi(\alpha p)}\sum_{i = (j-1)(p-1)+1}^{j(p-1)}\|\Delta_{i-1}- \Delta_i\|_\infty.
    \end{align*}
    If \eqref{almost_contraction} holds, then 
  \begin{align*}
        W_1(\T_{\nu, \beta}(\lambda), \T_{\nu, \beta}(\lambda')) \leq 4\beta\E \|\theta\|_\infty e^{2\beta\|\theta\|_\infty} \alpha p(p-1) W_1(\lambda, \lambda'),
    \end{align*}
    which, under the high temperature condition of \eqref{ht1}, will imply that $\T_{\nu, \beta}$ is a contraction on the complete space $\Pr_1(\X)$, and the statement of the lemma will follow from the Banach fixed-point theorem.

    To establish \eqref{almost_contraction}, firstly, note that from the structure of the operator $\T_{\nu,\beta}$, we can, without loss of generality, let $f_k, f_k'$ be normalized, i.e., $\int e^{\beta f_k} d\nu = \int e^{\beta f_k'}d\nu = 1$ for all $k\geq 1$. Fix $j$ and $i$ in the range above and recall the definition of $\cE_{\beta, r}$ from \eqref{def:E}. Following \cite[Lemma 6.6.2]{book}, let us denote \begin{align*}
        \cE' _\beta(\rho, t)= \exp \beta \theta_j(\rho_{(j-1)(p-1)+1}, \ldots, \rho_{j(p-1)}, t)
    \end{align*}
    and $\cE''_\beta (\rho,t)= \cE_{\beta,\pi(\alpha p)}(\rho, t)/\cE'_\beta(\rho, t)$. Also for a function $\phi = \phi(\rho, t)$, let \begin{align*}
        \langle \phi\rangle_i&= \int \phi(\rho, t) \exp \Big( \beta \sum_{r \leq i} f_r(\rho_r) + \beta \sum_{r>i} f'_r(\rho_r)\Big)\nu^{\otimes (p-1)\pi(\alpha p)}(d\rho)
        \end{align*}
    and $\langle \av \phi\rangle_i = \int \langle \phi \rangle_i \nu(dt)$. With this notation, we have \begin{align*}
        \Delta_i(t) = \frac{1}{\beta}\log \frac{\langle \cE_{\beta,\pi(\alpha p)}\rangle_i}{\langle \av \cE_{\beta, \pi(\alpha p)}\rangle_i}. 
    \end{align*}
    Then \begin{align}\label{eqn diff i, i-1}
        \Delta_i - \Delta_{i-1} & = \frac{1}{\beta}\log \frac{\langle \cE_{\beta,\pi(\alpha p)}\rangle_i}{\langle \av \cE_{\beta,\pi(\alpha p)}\rangle_i} - \frac{1}{\beta}\log \frac{\langle \cE_{\beta,\pi(\alpha p)}\rangle_{i-1}}{\langle \av \cE_{\beta,\pi(\alpha p)}\rangle_{i-1}} \notag\\
        & = \frac{1}{\beta}\log \frac{\langle \cE_{\beta,\pi(\alpha p)}\rangle_i}{\langle \cE_{\beta,\pi(\alpha p)}\rangle_{i-1}} - \frac{1}{\beta}\log \frac{\av \langle \cE_{\beta,\pi(\alpha p)}\rangle_i}{\av \langle  \cE_{\beta,\pi(\alpha p)}\rangle_{i-1}}.
    \end{align}
     Note that $\langle \cE_{\beta,\pi(\alpha p)}\rangle_i = \langle \cE'_\beta\rangle_i\langle \cE''_\beta\rangle_i$ since $\cE_\beta'$ and $\cE''_\beta$ depend on disjoint sets of coordinates. Further, since $\cE''_\beta$ is independent of $\rho_i$, we have that $\langle \cE''_\beta\rangle_i = \langle \cE''_\beta\rangle_{i-1}$. These observations yield that \begin{align*}
       \frac{1}{\beta}\log \frac{\langle \cE_{\beta,\pi(\alpha p)}\rangle_i}{\langle \cE_{\beta,\pi(\alpha p)}\rangle_{i-1}} &  =\frac{1}{\beta}\log \frac{\int B(\rho_i, t)e^{\beta f_i(\rho_i)}\nu(d\rho_i)}{\int B(\rho_i, t)e^{\beta f'_i(\rho_i)}\nu(d\rho_i)},
    \end{align*}
    where \begin{align*}
        B(\rho_i,t) = \int \cE'_\beta \exp \Big(\beta \sum_{k<i} f_k(\rho_k) + \beta \sum_{k>i}f'_k(\rho_k)\Big)\nu^{\otimes (p-2)}(d\rho).
    \end{align*}
  For $f_i \in \X$ normalized, we can use the inequality $\big|\log\frac{x}{y}\big| = \frac{|x-y|}{\min(x,y)}$ for $x, y >0$ and the fact that $e^{-\beta \|\theta_j\|_\infty} \leq B(\rho_i, t) \leq e^{\beta \|\theta_j\|_\infty}$ to write
    \begin{align*}
        \Big|\frac{1}{\beta}\log \frac{\int B(\rho_i, t)e^{ \beta f_i(\rho_i)}\nu(d\rho_i)}{\int B(\rho_i, t)e^{ \beta f'_i(\rho_i)}\nu(d\rho_i)}\Big| & \leq e^{\beta \|\theta_j\|_\infty} \Big| \frac{1}{\beta}\int B(\rho_i, t) \big( e^{\beta f_i(\rho_i)} - e^{\beta f'_i(\rho_i)}\big) \nu(d\rho_i) \Big|\\
        &  = e^{\beta \|\theta_j\|_\infty} \Big| \frac{1}{\beta}\int \big(B(\rho_i, t)-1\big) \big( e^{\beta f_i(\rho_i)} - e^{\beta f'_i(\rho_i)}\big) \nu(d\rho_i) \Big|\\
        & \leq  e^{\beta \|\theta_j\|_\infty} \int \frac{1}{\beta} \big|B(\rho_i, t)-1\big| \big| e^{\beta f_i(\rho_i)} - e^{\beta f'_i(\rho_i)}\big| \nu(d\rho_i)\\
        & \leq \|\theta_j\|_\infty e^{2\beta \|\theta_j\|_\infty} \|e^{\beta f_i} - e^{\beta f_i'}\|_1,
    \end{align*}
    where in the last line we used the inequality $|e^x-1| \leq |x| e^{|x|}$ with $x = \log B(\rho_i, t)$ and the fact that $|\log B(\rho_i, t)| \leq \beta \|\theta_j\|_\infty$.    
    Now, \begin{align*}
        \|e^{\beta f_i} - e^{\beta f_i'}\|_1 & \leq \beta\int e^{\beta\max( f_i, f_i')}| f_i -  f'_i| d\nu \leq \beta\| f_i- f_i'\|_\infty  \int e^{\beta\max( f_i, f_i')}d\nu \leq 2\beta\|f_i - f_i'\|_\infty,
    \end{align*}
    so we obtain that \begin{align*}
        \Big|\frac{1}{\beta}\log \frac{\langle \cE_{\beta,\pi(\alpha p)}\rangle_i}{\langle \cE_{\beta,\pi(\alpha p)}\rangle_{i-1}}\Big| & \leq 2\beta\|\theta_j\|_\infty e^{2\beta\|\theta_j\|_\infty}\|f_i - f_i'\|_\infty.
    \end{align*}
       This bound holds even when $\cE_{\beta,\pi(\alpha p)}$ is replaced by $\av \cE_{\beta,\pi(\alpha p)}$. Thus, we have bounded the right-hand side of \eqref{eqn diff i, i-1} to finish the proof of \eqref{almost_contraction}. Finally, note that the assertions in \eqref{eq2} follow directly from  \eqref{l_infty bound} and \eqref{lip_bound}.
    \end{proof}

\subsection{Tightness of the fixed point measures}
\begin{theorem}[Tightness]\label{tight}
    Let $0 <\beta < \infty$ be fixed. Then the family of probability measures $$\{\T_{\mu, \beta}(\lambda):\lambda\in \Pr_1(\X),\mu \in \Pr(\Sigma)\}$$ is tight.
\end{theorem}

\begin{proof}
	Let $\beta$ be fixed. For any $C,L>0,$ let \begin{align}\label{defn_K}
        K(C,L): = \{f \in \X: \|f\|_\infty \leq C, {\rm Lip}(f)\leq L\}.
    \end{align}
 Since $K(C,L)$ is a family of equicontinous and uniformly bounded functions, the Arzela-Ascoli theorem ensures that $K(C,L)$ is precompact. Let $0<\epsilon<1$. Take $k_0\geq 1$ with 
	\begin{align*}
		\Big(1-\frac{\epsilon}{2}\Big)\mathbb{P}(\pi(\alpha p)\leq k_0)\geq 1-\epsilon.
	\end{align*} Observe that conditionally on $\pi(\alpha p)\leq k_0$, using the assumptions \eqref{bdd theta} and \eqref{eq reg theta} and the Markov inequality, we see that there exist two positive constants $L_0 \geq 2\mbox{\rm Lip}(\psi)$ and $C_0\geq 4\|\psi\|_\infty$, depending on $\epsilon$ and $k_0$ such that with probability at least $1-\epsilon/2$, for any $k\leq k_0$, $\|\theta_k\|_\infty\leq {C_0}/(4 k_0)$ and $\mbox{Lip}(\theta_k)\leq{L_0}/{2k_0},$ which, together with \eqref{l_infty bound} and \eqref{lip_bound}, imply that for any $f_1,\ldots,f_{k_0}\in \X,$
	\begin{align*}
	    \|T_{\mu,\beta, \pi(\alpha p)}(f_1,\ldots,f_{(p-1)\pi(\alpha p)})\|_\infty\leq C_0 \;\;\text{and}\;\;\mbox{Lip}\bigl(T_{\mu,\beta, \pi(\alpha p)}(f_1,\ldots,f_{(p-1)\pi(\alpha p)})\bigr)\leq L_0.
	\end{align*}
	In other words, conditionally on $\pi(\alpha p)\leq k_0$, with probability at least $1-\epsilon/2,$ \begin{align*}
	    T_{\mu,\beta,\pi(\alpha p)}(f_1,\ldots,f_{(p-1)\pi(\alpha p)})\in K(C_0,L_0)
	\end{align*} for all $\mu\in \mbox{Pr}(\Sigma)$ and $f_1,\ldots,f_{k_0}\in \X.$ It follows that $$
	\sup_{\mu\in \mathrm{Pr}(\Sigma),\;\lambda\in \Pr_1(\X)}\bigl(\T_{\mu,\beta}(\lambda)\bigr)(K(C_0,L_0))\geq \Big(1-\frac{\epsilon}{2}\Big)\mathbb{P}(\pi(\alpha p)\leq k_0)\geq 1-\epsilon,
	$$
	completing our proof.
\end{proof}

\subsection{Continuity of the fixed point measures}
With a slightly abuse of notation, we also equip $\Pr(\Sigma)\times \Pr(\Sigma)$ with the Wasserstein $1$-distance $W_1$, i.e., 
$$
W_1(\nu,\nu')=\inf_{\Pi\in\Pi(\nu,\nu')}\int |s-t|\Pi(ds,dt),
$$
where $\Pi(\nu,\nu')$ is the collection of all couplings between $\nu$ and $\nu'$. Since $\Sigma$ is bounded, this quantity admits the following dual representation,
\begin{align}\label{dual}
	W_1(\nu,\nu')=\sup\Bigl\{\int_Sf(s)(\nu-\nu')(ds):f\in \X\;\;\mbox{and}\;\;\mbox{Lip}(f)\leq 1\Bigr\}.
\end{align}
Also, note that the boundedness of $\Sigma$ ensures that $W_1(\nu_n,\nu)\to 0$ if and only if $\nu_n\stackrel{d}{\to} \nu$.

\begin{theorem}\label{continuity}
	Let $(\nu_n)_{n\geq 1}$ be a sequence in $\Pr(\Sigma)$ that converges weakly to some $\nu_0\in \Pr(\Sigma)$. Then for every $0<\beta < \infty$, $(\lambda_{\nu_n, \beta})_{n\geq 1}$ converges to $\lambda_{\nu_0,\beta}$ with respect to the $W_1$-distance.
\end{theorem}

To prove this theorem, we need the following lemma.
   
   \begin{lemma}\label{lem1}
   	Let $C,L>0.$ Assume that $\lambda\in \Pr_1(\X)$ satisfies 
   	\begin{align}\label{eq5}
   	\int \|f\|_\infty\lambda(df)\leq C \;\;\text{and}\;\;\int \mbox{\rm Lip}(f)\lambda(df)\leq L.
   	\end{align} 
   For any $\epsilon,\delta>0,$ there exists a constant $C'>0$ depending only on $C,L,\epsilon,\delta, \beta$ and $\psi$ such that 
   	\begin{align*}
   		\P\Bigl(\Bigl\|	T_{\nu,\beta, \pi(\alpha p)}(f_1,\ldots,f_{(p-1)\pi(\alpha p)})-T_{\nu',\beta,\pi(\alpha p)}(f_1,\ldots,f_{(p-1)\pi(\alpha p)})\Bigl\|_\infty\geq \delta\Bigr)\leq \epsilon
   	\end{align*} 
   for any $\nu,\nu'\in {\rm Pr}(\Sigma)$ satisfying $W_1(\nu,\nu')<\delta/C',$ where $(f_k)_{k\geq 1}$ are i.i.d. sampled from $\lambda$ and are independent of other randomness.
   \end{lemma}

\begin{proof}
	Let $0<\epsilon<1.$ 
From \eqref{bdd theta}, \eqref{eq reg theta} and \eqref{eq5}, by using the Markov inequality and the union bound, there exist two positive constants $C_0$ and $L_0$ and an integer $k_0\geq 1$ such that with probability at least $1-\epsilon/2$, the following inequalities hold simultaneously: $\pi(\alpha p)\leq k_0,$ 
\begin{align*}
\sup_{k\leq k_0}\|\theta_k\|_\infty\leq \frac{C_0}{k_0},&\;\;\sup_{k\leq k_0}\mbox{Lip}(\theta_k)\leq L_0,\\
\sup_{k\leq (p-1)k_0}\|f_k\|_\infty\leq \frac{C_0}{(p-1)k_0},&\;\; \sup_{k\leq (p-1)k_0}\mbox{Lip}(f_k)\leq L_0.
\end{align*} Denote by $\mathcal{A}$ the event on which these inequalities hold.
Write
\begin{align}
\nonumber	&T_{\nu,\beta,\pi(\alpha p)}(f_1,\ldots,f_{(p-1)\pi(\alpha p)})(t)-T_{\nu',\beta,\pi(\alpha p)}(f_1,\ldots,f_{(p-1)\pi(\alpha p)})(t)\\
\label{eq7}	\begin{split}
    &=\frac{1}{\beta}\log \frac{\int\mathcal{E}_{\beta, \pi(\alpha p)}(\rho,t)\prod_{k=1}^{(p-1)\pi(\alpha p)}e^{\beta f_k(\rho_k)}\nu^{\otimes (p-1)\pi(\alpha p)}(d\rho)}{\int\mathcal{E}_{\beta, \pi(\alpha p)}(\rho,t)\prod_{k=1}^{(p-1)\pi(\alpha p)}e^{\beta f_k(\rho_k)}{\nu'}^{\otimes (p-1)\pi(\alpha p)}(d\rho)}\\
& \qquad -\frac{1}{\beta}\log \frac{\int\mathcal{E}_{\beta, \pi(\alpha p)}(\rho,s)\prod_{k=1}^{(p-1)\pi(\alpha p)}e^{\beta f_k(\rho_k)}\nu^{\otimes (p-1)\pi(\alpha p)}(d\rho)\nu(ds)}{\int\mathcal{E}_{\beta,\pi(\alpha p)}(\rho,s)\prod_{k=1}^{(p-1)\pi(\alpha p)}e^{\beta f_k(\rho_k)}{\nu'}^{\otimes (p-1)\pi(\alpha p)}(d\rho)\nu'(ds)}.
\end{split}
\end{align}
To treat the first term, without loss of generality, we assume that
\begin{align*}
    & \int\mathcal{E}_{\beta, \pi(\alpha p)}(\rho,t)\prod_{k=1}^{(p-1)\pi(\alpha p)}e^{\beta f_k(\rho_k)}\nu^{\otimes (p-1)\pi(\alpha p)}(d\rho)\\
    & \geq\int\mathcal{E}_{\beta, \pi(\alpha p)}(\rho,t)\prod_{k=1}^{(p-1)\pi(\alpha p)}e^{\beta f_k(\rho_k)}{\nu'}^{\otimes (p-1)\pi(\alpha p)}(d\rho)
\end{align*}
and write, by using $|\log(1+x)|=\log(1+x)\leq x$ for $x\geq 0,$
\begin{align*}
	&\frac{1}{\beta}\Bigl|\log \frac{\int\mathcal{E}_{\beta,\pi(\alpha p)}(\rho,t)\prod_{k=1}^{(p-1)\pi(\alpha p)}e^{\beta f_k(\rho_k)}\nu^{\otimes(p-1) \pi(\alpha p)}(d\rho)}{\int\mathcal{E}_{\beta, \pi(\alpha p)}(\rho,t)\prod_{k=1}^{(p-1)\pi(\alpha p)}e^{\beta f_k(\rho_k)}{\nu'}^{\otimes (p-1)\pi(\alpha p)}(d\rho)}\Bigr|\\
	&\leq \frac{1}{\beta}\frac{\int\mathcal{E}_{\beta, \pi(\alpha p)}(\rho,t)\prod_{k=1}^{(p-1)\pi(\alpha p)}e^{\beta f_k(\rho_k)}\bigl(\nu^{\otimes (p-1)\pi(\alpha p)}-{\nu'}^{\otimes (p-1)\pi(\alpha p)})(d\rho)}{\int\mathcal{E}_{\beta, \pi(\alpha p)}(\rho,t)\prod_{k=1}^{(p-1)\pi(\alpha p)}e^{\beta f_k(\rho_k)}{\nu'}^{\otimes(p-1) \pi(\alpha p)}(d\rho)}\\
	&\leq \bar L\exp\beta \Big(\sum_{k=1}^{\pi(\alpha p)}\|\theta_k\|_\infty+\sum_{k=1}^{(p-1)\pi(\alpha p)}\|f_k\|_\infty +  \|\psi\|_\infty\Big)\cdot W_1(\nu^{\otimes (p-1)\pi(\alpha p)},{\nu'}^{\otimes (p-1)\pi(\alpha p)})\\
	&\leq \bar L\exp\beta \Big(\sum_{k=1}^{\pi(\alpha p)}\|\theta_k\|_\infty+\sum_{k=1}^{(p-1)\pi(\alpha p)}\|f_k\|_\infty + \|\psi\|_\infty\Big) \cdot(p-1)\pi(\alpha p) W_1(\nu ,\nu'),
\end{align*}
where noting that for any $t\in \Sigma$ and $\rho,\rho'\in \Sigma^{(p-1)\pi(\alpha p)},$
\begin{align*}
&\Bigl|	\mathcal{E}_{\beta, \pi(\alpha p)}(\rho,t)\prod_{k=1}^{(p-1)\pi(\alpha p)}e^{\beta f_k(\rho_k)}-\mathcal{E}_{\beta, \pi(\alpha p)}(\rho',t)\prod_{k=1}^{(p-1)\pi(\alpha p)}e^{\beta f_k(\rho_k')}\Bigr|\leq\beta \bar L\|\rho-\rho'\|_2
\end{align*}
with
\begin{align*}
    \bar L&:=\Bigl(\max_{k\leq \pi(\alpha p)}\mbox{Lip}(\theta_k)+\max_{k\leq (p-1)\pi(\alpha p)}\mbox{Lip}(f_k)\Bigr)\cdot\exp\beta \Big(\sum_{k=1}^{\pi(\alpha p)}\|\theta_k\|_\infty+ \sum_{k=1}^{(p-1)\pi(\alpha p)}\|f_k\|_\infty +\|\psi\|_\infty\Big),
\end{align*}
the numerator in the second inequality used the dual representation \eqref{dual}
and the third inequality used the bound $W_1(\mu^{\otimes  n},{\mu'}^{\otimes n})\leq nW_1(\mu ,\mu')$ for any $n\geq 1.$ The second logarithmic term in \eqref{eq7} can be treated in the same way. As a conclusion, we obtain that on $\mathcal{A}$,
\begin{align*}
&\Bigl\|	T_{\nu,\beta, \pi(\alpha p)}(f_1,\ldots,f_{(p-1)\pi(\alpha p)})(t)-T_{\nu',\beta, \pi(\alpha p)}(f_1,\ldots,f_{(p-1)\pi(\alpha p)})\Bigl\|_\infty\leq C'W_1(\nu,\nu'),
\end{align*}
where $C':=4L_0e^{2\beta(2C_0 + \|\psi\|_\infty)}(p-1)k_0$.
Therefore,
\begin{align*}
&\P\Bigl(\Bigl\|	T_{\nu,\beta,\pi(\alpha p)}(f_1,\ldots,f_{(p-1)\pi(\alpha p)})(t)-T_{\nu',\beta,\pi(\alpha p)}(f_1,\ldots,f_{(p-1)\pi(\alpha p)})\Bigl\|_\infty\geq \delta\Bigr)\\
& \leq \P(\mathcal{A}^c)+\P\bigl( C'W_1(\nu,\nu')\geq \delta\bigr)\leq \epsilon,
\end{align*}
as long as $W_1(\nu,\nu')<\delta/C'.$
\end{proof}

\begin{proof}[Proof of Theorem \ref{continuity}]
	From Theorem \ref{tight}, $(\lambda_{\nu_n, \beta})_{n\geq 1}$ is tight and since $\Pr_1(\X)$ is complete and separable,  this sequence has a convergent subsequence. From now on, to ease our notation, we assume that $(\lambda_{\nu_n, \beta})_{n\geq 1}$ converges to some $\lambda_*$ and we aim to show that $\T_{\nu_0,\beta}(\lambda_*)=\lambda_*.$ If this is valid, then from Theorem \ref{uniqueness}, we must have $\lambda_*=\lambda_{\nu_0,\beta}$ and this will complete our proof.
	
	 To establish that $\T_{\nu_0,\beta}(\lambda_*)=\lambda_*,$ it suffices to show that $\T_{\nu_n,\beta}(\lambda_{\nu_n,\beta})\to \T_{\nu_0,\beta}(\lambda_*)$ in $\Pr_1(\X)$. Recall that $\T_{\nu_n,\beta}(\lambda_{\nu_n,\beta})$ and $\T_{\nu_0,\beta}(\lambda_*)$ are respectively the laws of $$T_{\nu_n,\beta,\pi(\alpha p)}(f_1^n,\ldots,f_{(p-1)\pi(\alpha p)}^n)\;\;\text{and} \;\;T_{\nu_0,\beta,\pi(\alpha p)}(f_1^*,\ldots,f^*_{(p-1)\pi(\alpha p)}),$$ where $(f_k^n)_{k\geq 1}$ and $(f_k^*)_{k\geq 1}$ are i.i.d. samples from $\lambda_{\nu_n,\beta}$ and $\lambda_*$ respectively. Denote
	 \begin{align*}
	 	T_n&=T_{\nu_0,\beta,\pi( \alpha p)}(f_1^n,\ldots,f_{(p-1)\pi(\alpha p)}^n),\\
	 	 	T_*&=T_{\nu_0,\beta,\pi( \alpha p)}(f_1^*,\ldots,f_{(p-1)\pi(\alpha p)}^*),\\
	 	\Delta_n&=T_{\nu_n,\beta, \pi(\alpha p)}(f_1^n,\ldots,f_{(p-1)\pi(\alpha p)}^n)-T_{\nu_0,\beta,\pi(\alpha p)}(f_1^n,\ldots,f_{(p-1)\pi(\alpha p)}^n)
	 \end{align*}
so that
 \begin{align*}
 	T_{\nu_n,\beta,\pi(\alpha p)}(f_1^n,\ldots,f_{(p-1)\pi(\alpha p)}^n)&=\Delta_n+T_n.
 \end{align*}
Owing to the structure of the operator $\T_{\nu_0,\beta}$, note that without loss of generality, we can take the sequences $(f_k^n)_{k\geq 1}$ and $(f_k^*)_{k\geq 1}$ to be normalized, i.e., $\int e^{f_k^n} d\nu_0 = \int e^{f_k^*} d\nu_0 = 1$. Since $\lambda_{\nu_n,\beta}\to\lambda_*$ weakly in $\Pr_1(\X)$, we have that $T_n\to T_*$ weakly in $\Pr_1(\X)$ as well.  From \eqref{eq2} and the fact that $\T_{\nu_n,\beta}(\lambda_{\nu_n,\beta})=\lambda_{\nu_n,\beta}\to \lambda_*$ in $\Pr_1(\X)$, we obtain that for some positive constants $C$ and $ L$,\begin{align*}
    & \int \|f\|_\infty \T_{\nu_n,\beta}(\lambda_{\nu_n,\beta})(df) \leq  C\;\;\text{ and} \;\;\int \mbox{\rm Lip}(f)(\T_{\nu_n,\beta}(\lambda_{\nu_n,\beta}))(df) \leq L
\end{align*}
uniformly for all $n\geq 1$. Consequently, from Lemma \ref{lem1}, for any $\epsilon,\delta>0$, there exists a constant $C'>0$ such that $\P(\|\Delta_n\|_\infty\geq \delta)\leq \epsilon$ as long as $W_1(\nu_n,\nu_0)\leq \delta/C',$ which means that $\Delta_n$ converges to zero in probability. Hence, by Slutsky's theorem, we have arrived at $\T_{\nu_n, \beta}(\lambda_{\nu_n,\beta})\stackrel{d}{\to} \T_{\nu_0,\beta}(\lambda_*)$. Finally, from \eqref{l_infty bound}, we see that $(\T_{\nu_n,\beta}(f_1^n,\ldots,f_{(p-1)\pi(\alpha p)}^n))_{n\geq 1}$ is uniformly integrable; this together with the convergence $\T_{\nu_n,\beta}(\lambda_{\nu_n,\beta})\stackrel{d}{\to} \T_{\nu_0,\beta}(\lambda_*)$ leads to
	 \begin{align*}
	 	\lim_{n\to\infty}\int \|f\|_\infty(\T_{\nu_n,\beta}(\lambda_{\nu_n, \beta}))(df)=\int \|f\|_\infty(\T_{\nu_0,\beta}(\lambda_*))(df).
	 \end{align*}
 Hence, $\T_{\nu_n,\beta}(\lambda_{\nu_n,\beta})\to \T_{\nu_0,\beta}(\lambda_*)$ in $\Pr_1(\X)$ and our proof is completed.
\end{proof}

\section{Replica symmetric behavior}\label{rs_behavior}

The objective of this section is to establish two important properties of the limiting spin distribution. First, we show that the system is in a pure state, meaning that any finite collection of spins is asymptotically independent under the Gibbs distribution, see Corollary \ref{cor ps} for a precise statement. The second purpose is to identify the limiting distribution of the spins as the unique fixed point of the distributional operator $\T_{\nu, \beta}$ introduced in the previous section; this is done in Theorem \ref{thm fixed point}. Equipped with these two properties, we characterize the asymptotic joint density of a collection of finitely many spins as the product of random density functions which are sampled from the measure which is the unique fixed point of $\T_{\nu,\beta}$.

\subsection{The invariance principle}

Fix a (non-random) sequence of real numbers $(c_N)_{N\geq 1}$ that satisfies $c_N \to \infty$, $c_N/N \to 0$ and $|c_{N+1} - c_N| \to 0$. We introduce the following perturbation to our Hamiltonian \eqref{eq:orig hamiltonian}, \begin{align*}
    H_N^{(p)}(\sigma) = \sum_{l \leq \pi(c_N)}\log \av \exp \Big({\sum_{k\leq \pi_l(\alpha p)}\theta_{l,k}(\sigma_{J(l,k,1)}, \ldots, \sigma_{J(l,k,p-1)}, \varepsilon) + \psi(\varepsilon)}\Big),
\end{align*}
where $\av$ denotes averaging in the variable $\varepsilon$ distributed as $\nu$, $(\theta_{l,k})_{l,k\geq 1}$ and $(\pi_l(\alpha p))_{l\geq 1}$ are i.i.d. distributed as $\theta$ and $\pi(\alpha p)$ respectively, and $(J_{l,k})_{l,k\geq 1} := (\{J(l,k,r):r\in [p-1]\})_{l,k\geq 1}$ are uniformly chosen from $\binom{[N]}{p-1}$. All randomness are independent of each other and of those introduced earlier. We will work with the perturbed Hamiltonian \begin{align}\label{pert_H}
     \hat H_N(\sigma) = H_N(\sigma) + H_N^{(p)}(\sigma)
\end{align}
and denote the corresponding Gibbs measure by $\hat G_{N,\beta, \nu}$. Note that since the perturbation is of $o(N)$, the free energy corresponding to the original and the perturbed Hamiltonians are asymptotically the same. Further, this perturbation does not alter the Gibbs average of functions that depend only on a bounded number of spins (see \cite[Lemma 2]{Panchenko2010SPINGM}), so we shall drop the hat from the notation for the Gibbs measure. 

Denote by $\mu_{N,\beta,\nu}$ the joint distribution of the array of all spins on all replicas $(\sigma_i^l)_{i \leq N, l\geq 1}$ under the annealed product Gibbs measure $\E G_{N,\nu,\beta}^{\otimes \infty}$, i.e., for any $q\geq 1$ and an arbitrary collection $(A_i^l)_{i\in[N],l\leq q} $ of Borel sets of $\mathbb{R}$, \begin{align*}
    \mu_{N,\beta,\nu}(\{\sigma_i^l \in A_i^l: i\in [N], l\in [q]\}) = \E G_{N,\nu,\beta}^{\otimes \infty}(\{\sigma_i^l \in A_i^l: i\in [N], l\in [q]\}).
\end{align*}
We extend the definition of $\mu_{N,\beta,\nu}$ to all $(\sigma_i^l)_{i,l\geq 1}$ by setting $\sigma_i^l = 0$ for all $i> N$. Denote by $\M_{\beta,\nu}$ the set of all weak subsequential limits of $(\mu_{N,\beta,\nu})_{N\geq 1}$, which are invariant with respect to finite-dimensional permutations of the indices $i$ and $l.$ Based on this property, the Aldous-Hoover representation ensures that for any $\mu \in \M_{\beta,\nu}$, there exists a measurable function $\sigma = \sigma_{\mu, \beta}:[0,1]^4\to \Sigma$ such that \begin{align*}
    \mu = \text{law of }(\sigma(w, u_l, v_i, x_{l,i}))_{l,i\geq 1},
    \end{align*}
where $w, u_l, v_i, x_{l,i}$ are uniform on $[0,1]$ and independent of each other. 

Fix integers $n, m, q \geq 1$. Here $n$ denotes the number of cavity coordinates, $m$ is the number of non-cavity coordinates, and $q$ is the number of replicas. For $l, k \geq 1$ and indices $i_1, \ldots, i_k \geq 1$, we shall write $s_{i_1, \ldots, i_k} = \sigma(w, u, v_{i_1, \ldots, i_k}, x_{i_1, \ldots, i_k})$ and $s^l_{i_1, \ldots, i_k} = \sigma(w, u_l, v_{i_1, \ldots, i_k}, x_{l, i_1, \ldots, i_k})$. For $k \geq 1$, let $f_1, \ldots, f_k:\Sigma^{(m+n)\times q}\to \R$ and $F: \bigtimes_{i=1}^k f_i(\Sigma^{(m+n)\times q}) \to \R$ be continuous functions. We write the inputs to the functions $f_k$ as $(s', s^{(0)})\in\Sigma^{(m+n)\times q}$ where \begin{align*}
    s' = (s_1^l, \ldots, s_m^l)_{1\leq l\leq q} \in \Sigma^{m\times q} \;\;\text{and} \;\; s^{(0)} = (s_{m+1}^l, \ldots, s_{m+n}^l)_{1\leq l\leq q} \in \Sigma^{n\times q}.
\end{align*}
We further let $\varepsilon = (\varepsilon_1^l, \ldots, \varepsilon_n^l)_{1\leq l\leq q}$ and define \begin{align}\label{eq:defn_e}
   \cE_\beta (\varepsilon):= \exp \beta\sum_{l\leq q}\sum_{i\leq n}\Big(\sum_{k \leq \pi_i(\alpha p)}\theta_{k,i}(s_{k,i,1}^l, \ldots, s_{k,i,p-1}^l, \varepsilon_i^l)+\psi(\varepsilon^l_i)\Big).
\end{align}

The following version of the invariance principle is a generalization of \cite[Equation (34)]{EJP2963}. We only state this result here without proof since it does not involve any major change when it applies to models with general spins.

\begin{theorem}[Invariance principle]\label{thm inv}
    Let $0 < \beta <\infty$. For any $\mu \in \M_{\beta,\nu}$ and $\sigma = \sigma_{\mu,\beta}$, we have \begin{align}\label{eq:inv_new}\begin{split}
    & \E F\Bigl( \E_{u,x} f_1(s',s^{(0)}), \ldots,  \E_{u,x} f_k(s',s^{(0)})\Bigr)\\
    &  = \E  F\left(\frac{\E_{u,x} \av f_1(s',\varepsilon) \cE_\beta(\varepsilon)}{\E_{u,x} \av \cE_\beta(\varepsilon)}, \ldots, \frac{\E_{u,x} \av f_k(s',\varepsilon) \cE_\beta(\varepsilon)}{\E_{u,x} \av \cE_\beta(\varepsilon)}\right),
\end{split}
\end{align}
where the subscripts to $\E$ denote expectation with respect to only those variables and $\av$ denotes averaging with respect to $\varepsilon$.
\end{theorem}

\subsection{Pure states}\label{sec3.1}

The replica symmetric behavior we aim to show is that the spins are asymptotically uncorrelated with respect to the Gibbs measure, or equivalently, the spin distribution array $(s_{i}^l)_{i,l\geq 1}$ is equal in distribution to $(\sigma(w,u,v_i,x_{i,l}))_{i,l\geq 1}$, where $u$ is uniform on $[0,1]$ and independent of other randomness. To describe our result, we will need the operator $\Phi^{(0)}: \Sigma^{n\times q} \to  \Sigma^{n\times q}$ that is defined as \begin{align*}
    \Phi^{(0)}(s^{(0)}) = \bigl((s_{m+1}^2, \ldots s_{m+n}^2), (s_{m+1}^1, \ldots s_{m+n}^1), (s_{m+1}^l, \ldots s_{m+n}^l)_{3\leq l\leq q}\bigr).
\end{align*}
In other words, $\Phi^{(0)}$ swaps $(s_{m+1}^1, \ldots s_{m+n}^1)$ with $(s_{m+1}^2, \ldots s_{m+n}^2)$ in $s^{(0)}$ and leaves the other coordinates intact. A key result for what follows is the next theorem.

\begin{theorem}\label{thm_pure_states}
     Let $f_1,f_2:\Sigma^{(m+n)\times q}\to \R$ be continuous functions such that $$f_1 (s',\Phi^{(0)}(s^{(0)})) = -f_1(s',s^{(0)})\;\; \text{and}\;\; 0<|f_1|\leq f_2.$$ 
     Suppose that \begin{align}\label{ht3}
        2\beta\alpha p (p-1)\exp(4\beta \|\psi\|_\infty + \alpha p(\E e^{4\beta\|\theta\|_\infty} - 1)\big)\E \|\theta\|_\infty e^{4\beta\|\theta\|_\infty}  <1
     \;\;\text{or}\;\; \alpha p (p-1) \leq 1 \end{align} holds. Then we have that \begin{align}\label{pstate}
        \E\Big|\frac{\E_{u,x} f_1(s',s^{(0)})}{\E_{u,x} f_2(s',s^{(0)})}\Big| = 0.
    \end{align}
\end{theorem}
Again, as in Theorem \ref{uniqueness}, the proof of this theorem comprises two parts, corresponding to the subcritical regime and the high temperature regime. The proof provided below for the subcritical regime is an improvement of \cite[Lemma 1]{EJP2963}, which is achieved by considering the recursive tree-like process introduced in Lemma \ref{gen_fixed_pt}.

\begin{proof}[Proof of Theorem \ref{thm_pure_states} (Subcritical regime)]
  Label $f_i^{(0)} = f_i$ for $i = 1, 2$. Recall the random tree $\G_\varnothing$ introduced in Lemma \ref{gen_fixed_pt}. For $k> m$, identify $(s^1_k, \ldots, s^q_k)$ with $(s^1_\i, \ldots, s^q_\i)$, where $\i$ is the $(k-m)$th vertex of $\G_\varnothing$ discovered during a breadth-first exploration of $\G_\varnothing$. For a level $r\geq 0$ in the tree $\G_\varnothing$, let $s^{(r)} := (s^1_\i, \ldots, s^q_\i)_{\i \in \V_{=r}}$, (here the indices in $\V_{=r}$ are ordered in the breadth-first sequence that they were discovered) and let\begin{align*}
    \cE_\beta^{(r)}(\varepsilon, s^{(r+1)}) = \exp \beta \sum_{l\leq q} \sum_{\i \in \V_{=r}}\Big(\sum_{k\leq \pi_\i(\alpha p)}\theta_{\i,k}(s^l_{\i((p-1)(k-1)+1)}, \ldots, s^l_{\i(p-1)k}, \varepsilon^l_\i) + \psi(\varepsilon^l_\i)\Big),
\end{align*}
where $\varepsilon = (\varepsilon^1_\i, \ldots, \varepsilon^q_\i)_{\i \in \V_{=r}}$, and for $i = 1, 2$, \begin{align*}
    f_i^{(r+1)}(s',s^{(r+1)}) = \av f_i^{(r)}(s',\varepsilon)\cE^{(r)}_\beta(\varepsilon, s^{(r+1)}). 
\end{align*}
With $\Phi^{(r)}$ the switching operator acting as \begin{align*}
    s^{(r)} \mapsto \big((s^2_\i)_{\i \in \V_{=r}}, (s^1_\i)_{\i \in \V_{=r}}, (s^3_\i)_{\i \in \V_{=r}}\ldots, (s^q_\i)_{\i \in \V_{=r}}\big),
\end{align*} it is easy to verify that conditioned on the randomness of $(\pi_\i(\alpha p))_{\i \in \V_{=r}}$ and $(\theta_{\i,k})_{k\leq \pi_\i(\alpha p), \i \in \V_{=r}}$, $f_1^{(r)}(s',\Phi^{(r)}(s^{(r)})) = -f_1^{(r)}(s',s^{(r)})$ and $|f_1^{(r)}| \leq f_2^{(r)}$. Indeed, this is true for $r=0$ by assumption, and for $r\geq 1$ one can show inductively that this holds since 
\begin{align*}
    f_1^{(r)} (s',\Phi^{(r)}(s^{(r)})) & = \av [f_1^{(r-1)}(s',\varepsilon)\cE_\beta^{(r-1)}\big(\varepsilon,\Phi^{(r)}(s^{(r)})\big)]\\
    & = \av [f_1^{(r-1)}\big(s',\Phi^{(r-1)}(\varepsilon)\big)\cE_\beta^{(r-1)}\big(\Phi^{(r-1)}(\varepsilon),\Phi^{(r)}(s^{(r)})\big)]\\
    & = -\av [f_1^{(r-1)}(s',\varepsilon)\cE_\beta^{(r-1)}\big(\Phi^{(r-1)}(\varepsilon),\Phi^{(r)}(s^{(r)})\big)]\\
    & = -\av [f_1^{(r-1)}(s',\varepsilon)\cE_\beta^{(r-1)}(s',s^{(r)})] =  - f_1^{(r)}(s',s^{(r)}).
\end{align*}
In the above computation, the second line follows by swapping the replica labels of $(\varepsilon^1_\i)_{\i \in \V_{=r-1}}$ and $(\varepsilon^2_\i)_{\i \in \V_{=r-1}}$ so the averaging in $\varepsilon$ is not altered, and the fourth line is due to $$\cE_\beta^{(r-1)}\big(\Phi^{(r-1)}(\varepsilon),\Phi^{(r)}(s^{(r)})\big) = \cE_\beta^{(r-1)}(\varepsilon,s^{(r)}).$$ Also, $|f_1^{(0)}| \leq f_2^{(0)}$, and for any $r\geq 1$, by induction, we have \begin{align*}
    |f_1^{(r)}(s',s^{(r)})| & \leq \av [|f_1^{(r-1)}(s',\varepsilon)| \cE_\beta^{(r-1)}(\varepsilon, s^{(r)})] \leq \av [f_2^{(r-1)}(s',\varepsilon)\cE_\beta^{(r-1)}(\varepsilon,s^{(r)})] = f_2^{(r)}(s',s^{(r)}).
\end{align*}
Let $\G'$ denote the graph $\G$ with only the edges in $\G_\varnothing$ present, and the rest of the edges deleted. In other words, in $\G'$, we set $\pi_\i(\alpha p) = 0$ for any $\i \not \in \G_\varnothing$. Let $\text{ht}(\G_\varnothing)$ denote the height of $\G_\varnothing$. Since $\alpha p(p-1) \leq 1$, almost surely, $\text{ht}(\G_\varnothing) <\infty$. Hence in $\G'$, there is a certain level $r_0 \geq 0$ such that for $r\geq r_0$, $\pi_\i(\alpha p) = 0$ for all $\i \in \V_{=r}$, and $\cE^{(r)}_\beta = 1$, leading to
    \begin{align*}
    f_1^{(r+1)}(s',s^{(r+1)})& = \av f_1^{(r)}(s',\varepsilon)\cE_\beta^{(r)}(\varepsilon, s^{(r+1)}) = \av f_1^{(r)}(s',\varepsilon) = \av [f_1^{(r)}(s',\Phi^{(r)}(\varepsilon))] \\
    & = -\av f_1^{(r)}(\varepsilon, s') = 0.
\end{align*}
 Let $R \geq 0$ be an integer. On the event $\{\text{ht}(\G_\varnothing) \leq R\}$, we have that $\pi_\i(\alpha p) = 0$ for all $\i \in \V_{=R}$, so using Theorem \ref{thm inv} repeatedly with $n = 1$ and $F(x_1, x_2) = |x_1/x_2|$, we have 
\begin{align*}
    \E \Bigl[\Big| \frac{\E_{u,x} f_1^{(0)}(s',s^{(0)})}{\E_{u,x} f_2^{(0)}(s',s^{(0)})} \Big| \big| \text{ht}(\G_\varnothing) \leq R\Bigl] & = \E \Bigl[\Big| \frac{\E_{u,x} f_1^{(1)}(s',s^{(1)})}{\E_{u,x} f_2^{(1)}(s',s^{(1)})} \Big| \big| \text{ht}(\G_\varnothing) \leq R\Bigl]\\
    & = \ldots = \E \Bigl[\Big| \frac{\E_{u,x} f_1^{(\text{ht}(\G_\varnothing))}(s',s^{(\text{ht}(\G_\varnothing))})}{\E_{u,x} f_2^{(\text{ht}(\G_\varnothing))}(s',s^{(\text{ht}(\G_\varnothing))})} \Big| \big| \text{ht}(\G_\varnothing) \leq R\Bigl] = 0,
\end{align*} where the last step follows from the previous display. Since $\text{ht}(\G_\varnothing) < \infty$ a.s., sending $R\to \infty$ and using the dominated convergence theorem yields that \begin{align*}
    \E \Big| \frac{\E_{u,x} f_1^{(0)}(s',s^{(0)})}{\E_{u,x} f_2^{(0)}(s',s^{(0)})} \Big| = 0.
\end{align*}
\end{proof}
The arguments for the high temperature regime are along the lines as that of \cite[Lemma 2]{EJP2963}. We include the proof for completeness.

\begin{proof}[Proof of Theorem \ref{thm_pure_states} (High temperature regime)]
Borrowing the notation from the previous part and setting $n = 1$, $F(x_1,x_2) = |x_1/x_2|$ in the invariance principle (Theorem \ref{thm inv}), we obtain \begin{align*}
        \E \Big|\frac{\E_{u,x} f_1^{(0)}(s',s^{(0)})}{\E_{u,x} f_2^{(0)}(s',s^{(0)})}\Big| & = \E \Big|\frac{\E_{u,x}f_1^{(1)}(s',s^{(1)})}{\E_{u,x} f_2^{(1)}(s',s^{(1)})}\Big|.
    \end{align*}
   For each $j \leq (p-1)\pi_\varnothing(\alpha p)$, let $\Phi^{(1)}_j$ be the map acting on $s^{(1)}$ that switches $s_j^1$ with $s^2_j$ and leaves all the other coordinates unchanged. Thus, denoting composition of functions by $\bigcirc$, we have $\comp_{j \leq (p-1)\pi_\varnothing(\alpha p)} \Phi^{(1)}_j = \Phi^{(1)}$. In the previous part, we established that $$f_1^{(1)}(s',\Phi^{(1)}(s^{(1)})) = -f_1^{(1)}(s',s^{(1)}).$$ Thus, for $j\leq (p-1)\pi_\varnothing(\alpha p)$, defining
   \begin{align*}
       f^{(1)}_{1,j}(s',s^{(1)}) = \frac{1}{2}\Big(f_1^{(1)}(s', \comp_{k<j}\Phi^{(1)}_k(s^{(1)})) -  f_1^{(1)}(s',\comp_{k\leq j}\Phi^{(1)}_k(s^{(1)}))\Big),
   \end{align*}
   we can write \begin{align*}
       f_1^{(1)}(s',s^{(1)}) & = \frac{1}{2}\big(f_1^{(1)}(s,s^{(1)}) - f_1^{(1)}(s',\Phi^{(1)}(s^{(1)}))\big) \\
       & = \frac{1}{2}\Big(f_1^{(1)}(s,s^{(1)}) - f_1^{(1)}(s', \comp_{j \leq (p-1)\pi_\varnothing(\alpha p)} \Phi^{(1)}_j(s^{(1)})\Big)=  \sum_{j\leq (p-1)\pi_\varnothing(\alpha p)} f^{(1)}_{1,j}.
   \end{align*}
   By definition, we have $f^{(1)}_{1,j}(s',\Phi^{(1)}_j(s^{(1)})) = - f^{(1)}_{1,j}(s',s^{(1)})$. Hence, we obtain that \begin{align}\label{eq: f to tilde f ineq}
       \E \Big|\frac{\E_{u,x} f_1^{(0)}(s',s^{(0)})}{\E_{u,x} f_2^{(0)}(s',s^{(0)})}\Big| & = \E \Big|\frac{\E_{u,x}f_1^{(1)}(s',s^{(1)})}{\E_{u,x} f_2^{(1)}(s',s^{(1)})}\Big| \leq \E \sum_{j \leq (p-1)\pi_\varnothing(\alpha p)} \Big|\frac{\E_{u,x} f^{(1)}_{1,j}(s',s^{(1)})}{\E_{u,x}f_2^{(1)}(s',s^{(1)})}\Big|.
   \end{align}
   From the definition of $f_1^{(1)}$, we have that \begin{align*}
        f_{1,j}^{(1)} (s',s^{(1)}) = \frac{1}{2}\av \Big[f_1^{(0)}(s',\varepsilon)\Big(\cE_\beta^{(0)}(\varepsilon, \comp_{k< j}\Phi^{(1)}_k(s^{(1)})) - \cE^{(0)}_\beta(\varepsilon,  \comp_{k\leq j}\Phi^{(1)}_k(s^{(1)}))\Big)\Big].
   \end{align*}
   Since all the maps $\Phi^{(1)}_j$ switch coordinates in the first and second replicas, if we write \begin{align*}
       \cE'_\beta(\varepsilon, s^{(1)}) = \exp \beta \sum_{l=1,2}\Big(\sum_{k\leq \pi_\varnothing(\alpha p)}\theta_{\varnothing, k}(s^l_{(k-1)(p-1) + 1}, \ldots, s^l_{k(p-1)}, \varepsilon^l) + \psi(\varepsilon^l)\Big)
   \end{align*}
   and $\cE''_\beta(\varepsilon, s^{(1)}) = \cE_\beta^{(1)}(\varepsilon, s^{(1)})/\cE'_\beta(\varepsilon, s^{(1)})$, then we can write 
   \begin{align*}
      f^{(1)}_{1,j} (s',s^{(1)}) = \frac{1}{2}\av \Big[f_1^{(0)}(s',\varepsilon)\cE''_\beta(\varepsilon, s^{(1)}) \Big(\cE'_\beta(\varepsilon, \comp_{k<j}\Phi^{(1)}_k( s^{(1)}))- \cE'_\beta(\varepsilon, \comp_{k\leq j}\Phi^{(1)}_k(s^{(1)})) \Big)\Big].
   \end{align*}
  Note that if $k$ is such that $(k-1)(p-1) + 1\leq j \leq k(p-1)$, then in the above display the only terms affected are $\theta_{k}(s^l_{(k-1)(p-1) + 1}, \ldots, s^l_{k(p-1)}, \varepsilon^l)$ for $l = 1, 2$. Using the inequality $|e^x - e^y| \leq e^{\max(x,y)}|x-y|$ and noting that $\varepsilon^l \in \Sigma \subseteq [-R,R]$, we thus have \begin{align*}
      & \Big|\cE'_\beta(\varepsilon, \comp_{k<j}\Phi^{(1)}_k( s^{(1)}))- \cE'_\beta(\varepsilon, \comp_{k\leq j}\Phi^{(1)}_k(s^{(1)}))\Big| \\
      & \leq 4\beta\|\theta_{\varnothing,k}\|_\infty\exp \Big(2\beta\sum_{k'\leq \pi_\varnothing(\alpha p)}\|\theta_{\varnothing,k'}\|_\infty + 2\beta \|\psi\|_\infty\Big).
  \end{align*}
  and thus \begin{align*}
      |f_{1,j}^{(1)} (s',s^{(1)})| & \leq 2\beta\|\theta_{\varnothing,k}\|_\infty \exp \Big(2\beta\sum_{k'\leq \pi_\varnothing(\alpha p)}\|\theta_{\varnothing,k'}\|_\infty + 2\beta \|\psi\|_\infty\Big)\av \Big(|f_1^{(0)}(s',\varepsilon)|\cE''_\beta(\varepsilon, s^{(1)})\Big) \\
      &  \leq 2\beta\|\theta_{\varnothing,k}\|_\infty \exp \Big(2\beta\sum_{k'\leq \pi_\varnothing(\alpha p)}\|\theta_{\varnothing,k'}\|_\infty + 2\beta \|\psi\|_\infty\Big)\av \Big(f_2^{(0)}(s',\varepsilon)\cE''_\beta(\varepsilon,s^{(1)})\Big).
  \end{align*}
  Also, \begin{align*}
      f_2^{(1)}(s',s^{(1)}) & = \av \Big(f_2^{(0)}(s',\varepsilon)\cE_\beta^{(0)}(\varepsilon, s^{(1)})\Big) = \av \Big(f_2^{(0)}(s',\varepsilon)\cE'_\beta(\varepsilon, s^{(1)})\cE''_\beta(\varepsilon, s^{(1)})\Big)\\
      & \geq \exp \Big(-2\beta\sum_{k'\leq \pi_\varnothing(\alpha p)}\|\theta_{\varnothing,k'}\|_\infty -2\beta \|\psi\|_\infty\Big) \av \Big(f_2^{(0)}(s',\varepsilon)\cE''_\beta(\varepsilon, s^{(1)})\Big)
  \end{align*}
  so \begin{align*}
       |f_1^{(1)} (s',s^{(1)})| & \leq 2\beta\|\theta_{\varnothing,k}\|_\infty \exp \Big(4\beta\sum_{k'\leq \pi_\varnothing(\alpha p)}\|\theta_{\varnothing, k'}\|_\infty + 4\beta \|\psi\|_\infty\Big) f_2^{(1)}(s',s^{(1)}).
  \end{align*}
  Let \begin{align*}
      D = \sup\E \Big|\frac{\E_{u,x}f_1(s',s^{(0)})}{\E_{u,x}f_2(s',s^{(0)})}\Big|,
  \end{align*}
  where the supremum is over all functions $f_1$ and $f_2$ satisfying $$0 <|f_1| \leq f_2\;\; \text{and}\;\; f_1(s',\Phi(s^{(0)})) = -f_1(s',s^{(0)}),$$ where $\Phi$ is an operator that switches only one coordinate in the first and second replicas. Then we have from \eqref{eq: f to tilde f ineq} that \begin{align*}
      D & \leq D\E \sum_{k\leq \pi_\varnothing(\alpha p)}\sum_{j=(k-1)(p-1) + 1}^{k(p-1)} 2\beta\|\theta_{\varnothing, k}\|_\infty \exp \Big( 4\beta\sum_{k'\leq \pi_\varnothing(\alpha p)}\|\theta_{\varnothing, k'}\|_\infty + 4\beta \|\psi\|_\infty\Big)\\
      & = 2\beta D(p-1)e^{4\beta \|\psi\|_\infty} \E L e^{4\beta L},
  \end{align*}
  where $L = \sum_{k\leq \pi(\alpha p)} \|\theta_k\|_\infty$. Now, conditioned on $\pi(\alpha p) = l\geq 1$, we have \begin{align*}
      \E L e^{4\beta L} & = \sum_{k\leq l} \E \|\theta_k\|_\infty e^{4\beta \|\theta_k\|_\infty}\E e^{4\beta \sum_{k'\leq l, k'\neq k}\|\theta_{k'}\|_\infty} = l\E\|\theta\|_\infty e^{4\beta \|\theta\|_\infty} \Big(\E e^{4\beta \|\theta\|_\infty}\Big)^{l-1}.
  \end{align*}
  Hence, integrating out the randomness in $\pi(\alpha p)$, we have \begin{align*}
      \E L e^{4\beta L} & = \alpha p \exp \big(\alpha p(\E e^{4\beta \|\theta\|_\infty} - 1)\big)\E \|\theta\|_\infty e^{4\beta \|\theta\|_\infty} .
  \end{align*}
  From \eqref{ht3} it follows that $D = 0$, thus proving the lemma.
\end{proof}

To use the invariance principle and the uniqueness of the fixed point of $\T_{\nu, \beta}$ to determine the distributions of the spins in the pure state, we need to ensure that the parameters $(\alpha, \beta)$ satisfy \eqref{ht1} and \eqref{ht3}.

\begin{lemma}\label{lem hi_temp}
    If \eqref{ht} holds then \eqref{ht1} and \eqref{ht3} hold.
\end{lemma}

\begin{proof}
    Let $$\kappa: = 6\beta e^{4\beta \|\psi\|_\infty}\E \|\theta\|_\infty e^{4\beta \|\theta\|_\infty}.$$ Suppose first that \eqref{ht} holds with $ \kappa>1$. Then $\alpha p(p-1) \leq 1$, so both \eqref{ht1} and \eqref{ht3} hold. Now suppose that  $\kappa \leq 1$. Then from \eqref{ht}, $\kappa \alpha p(p-1)\leq 1$. Since $\beta \geq 0$, $e^{4\beta \|\psi\|_\infty}\geq 1$ and thus \begin{align*}
 4\beta \E \|\theta\|_\infty e^{2\beta \|\theta\|_\infty}\alpha p(p-1) & < \kappa \alpha p(p-1) \leq 1,
    \end{align*}
    so \eqref{ht1} holds. Moreover, using $\kappa \alpha p(p-1)\leq 1$ twice, we get\begin{align*}
        & 2\beta\alpha p (p-1) \exp \big(4\beta \|\psi\|_\infty + \alpha p(\E e^{4\beta \|\theta\|_\infty} - 1)\big)\E \|\theta\|_\infty e^{4\beta \|\theta\|_\infty} \leq \frac{1}{3} \exp \big(\alpha p (\E e^{4\beta\|\theta\|_\infty} - 1)\big) \\
        & < \frac{1}{3} \exp \big(4\beta \alpha p \E \|\theta\|_\infty e^{4\beta \|\theta\|_\infty}\big) \leq \frac{1}{3}\exp \Big(\frac{2e^{-4\beta \|\psi\|_\infty}}{3(p-1)}\Big) \leq \frac{1}{3}e^{\frac{2}{3(p-1)}} <1
    \end{align*}
    for all $p \geq 2$ so \eqref{ht3} holds.
\end{proof}

Let $m ,q > 1$ be fixed and $I$ be a subset of $\{(l,i):i\leq q, l\leq m\}$ that does not contain $(1,m)$ and $(2,m)$. Let $n=1$, and $\kappa, (\kappa_i^l)_{(l,i)\in I}$ be non-negative integers and \begin{align*}
    f_1 = [(\sigma_m^1)^\kappa - (\sigma_m^2)^\kappa] \prod_{(l,i)\in I} (\sigma_i^l)^{\kappa_i^l}
\end{align*}
 and $f_2 \equiv 1$. Then $f_1(s',\Phi^{(0)}(s^{(0)})) = -f_1(s',s^{(0)})$ and thus from Theorem \ref{thm_pure_states} we have \begin{align}\label{eq: one replica at a time}
    \E (s_m^1)^\kappa \prod_{(l,i)\in I} (s_i^l)^{\kappa_i^l} = \E (s_m^2)^\kappa \prod_{(l,i)\in I} (s_i^l)^{\kappa_i^l}.
\end{align}
For an integer $k\geq 1$, let \begin{align*}
   \bar s^{l,(k)}_{i_1, \ldots, i_n}: = \bar \sigma^{(k)}(w, u_l, v_{i_1, \ldots, i_n}): =\int_0^1 (s_{i_1, \ldots, i_n}^l)^k dx_{l, i_1, \ldots, i_n} = \int_0^1\sigma^k(w, u_l, v_{i_1, \ldots, i_n}, x)dx.
\end{align*}
Since the $x$ variables appearing in \eqref{eq: one replica at a time} are independent of each other for different replica or spin indices, we obtain by integrating over all of them that \begin{align}\label{eq: one replica at a time new} 
     \E \bar s_m^{1, (\kappa)} \prod_{(l,i)\in I} \bar s_i^{l, (\kappa_i^l)}= \E \bar s_m^{2, (\kappa)} \prod_{(l,i)\in I} \bar s_i^{l, (\kappa_i^l)}.
\end{align}
In particular, when $\kappa_i^l = \kappa = k$, using \eqref{eq: one replica at a time new} twice and the symmetry in the spin indices, we obtain \begin{align*}
    \E \bar s_1^{1, (k)} \bar s_1^{2, (k)} \bar s_2^{1, (k)} \bar s_2^{2, (k)} = \E \bar s_1^{1, (k)} \bar s_1^{2, (k)} \bar s_2^{3, (k)} \bar s_2^{4, (k)}.
\end{align*}
Defining the overlaps \begin{align*}
    R_{1,2}^{(k)} = R_{1,2}^{(k)}(u_1, u_2) = \E_v \bar s^{1, (k)} \bar s^{2, (k)} = \E_{v, x} \sigma^k(w, u_1, v, x_1) \sigma^k(w, u_2, v, x_2),
\end{align*}
we have from the previous display that $$\E_w \var_u (R_{1,2}^{(k)} | w) = 0,$$ from which we obtain that for almost all $w, u_1, u_2$, $R_{1,2}^{(k)} = C(k,w)$, where $C(k,w)$ is a constant depending only on $k$ and $w$.

\begin{lemma}\label{lem: ps}
 Let $g:\Sigma\to \R$ be a continuous function. Then for almost all $w, u, v \in [0,1]$, $$\E_x g(\sigma(w, u, v, x)) = \E_{u, x}g(\sigma(w, u, v, x)).$$
\end{lemma}

\begin{proof}
    It suffices to show that the above is true when $g$ is a polynomial. We just showed that for all $k \geq 1$ and almost all $w, u_1, u_2$, \begin{align}\label{eq: ps}
        \E_v \bar \sigma^{(k)}(w, u_1, v) \bar \sigma^{(k)}(w, u_2, v) = C(k, w).
    \end{align}
    Fix $b \in (0,1)$ and $\epsilon > 0$ small enough. Then \begin{align*}
        \E_v\left[\frac{1}{2\epsilon}\int_{b-\epsilon}^{b+\epsilon} \bar \sigma^{(k)}(w, u_1, v)du_1 \cdot \frac{1}{2\epsilon}\int_{b-\epsilon}^{b+\epsilon} \bar \sigma^{(k)}(w, u_2, v)du_2\right] = C(k,w).
    \end{align*}
    Since the spins take values in the compact set $\Sigma$, $\bar \sigma^{(k)}$ is integrable, so by the Lebesgue differentiation theorem, for almost all $w, b, v$ \begin{align*}
        \frac{1}{2\epsilon}\int_{b-\epsilon}^{b+\epsilon} \bar \sigma^{(k)}(w, u, v)du \to \bar \sigma^{(k)} (w, b, v).
    \end{align*}
    By dominated convergence, we have for almost all $b$\begin{align*}
        \E_v (\bar \sigma^{(k)} (w, b, v))^2 = C(k,w).
    \end{align*}
    Together with \eqref{eq: ps} this implies that $\E_v \var_u \bar \sigma^{(k)}(w,u,v) = 0$ from which the statement of the lemma follows.
\end{proof}

We write down an immediate corollary of the above lemma. 
\begin{cor}\label{cor ps}
Let $g:\Sigma^n \to \R$ be a continuous function. For almost all $w, v_1, \ldots, v_n \in [0,1]$ \begin{align*}
\E_{u, x}g\big(\sigma(w, u, v_1, x_1), \ldots, \sigma(w, u, v_n, x_n)\big) = \E_{u, x}g\big(\sigma(w, u_1, v_1, x_1), \ldots, \sigma(w, u_n, v_n, x_n)\big).
\end{align*}
\end{cor}

\begin{proof}
Again, it suffices to show the lemma when $g$ is a polynomial as any continuous function on the compact set $\Sigma^n$ can be uniformly approximated by polynomials. Let $\kappa_1, \ldots, \kappa_n \geq 0$ be integers. Using Lemma \ref{lem: ps} with $g(x) = x^{\kappa_i}$, for almost all $w, u, v_1, \ldots, v_n$, we have \begin{align*}
    \E_{u, x}\sigma^{\kappa_i}(w, u, v_i, x_i) & = \E_x\sigma^{\kappa_i}(w, u, v_i, x_i).
\end{align*}
Hence, \begin{align*}
    \E_x \prod_{i\leq n}\sigma^{\kappa_i}(w, u, v_i, x_i) & =\prod_{i\leq n} \E_x \sigma^{\kappa_i}(w, u, v_i, x_i)= \prod_{i\leq n}\E_{u, x}\sigma^{\kappa_i}(w, u, v_i, x_i) = \E_{u, x}\prod_{i\leq n} \sigma^{\kappa_i}(w, u_i, v_i, x_i).
\end{align*}
Taking expectation in the variable $u$ in the above display yields the result.
\end{proof}

\subsection{Convergence to the fixed point}\label{sec3.2}

Throughout this section, we assume that $\nu$ is supported on a finite set $\Sigma_L=\{t_1,\ldots, t_{L}\}$ for some $L\geq 1.$ In \eqref{eq:inv_new}, taking  $m=0$, $q = 1$, and functions $f_1, f_2, \ldots, f_{nL}$ as $\1_{s_i = t_l}, 1 \le l \le L, 1 \le i \le n$, we have that for any continuous function $F$,
\begin{align} 
    & \E_{w,v} F\bigl( \left\lbrace\E_{u,x} \1_{s_i = t_l}\right\rbrace_{ l\leq L,  i\leq n}\bigr) \nonumber\\
    & = \E_{w,v} F\Bigl( \Big\lbrace\frac{\E_{u,x} \av \1_{\varepsilon_i = t_l} \exp[\beta\sum_{j\leq n}(\sum_{k\leq \pi_j(\alpha p)}\theta_{k,j}(s_{k,j,1}, \ldots, s_{k,j,p-1}, \varepsilon_j) + \psi(\varepsilon_j))]}{\E_{u,x} \av \exp[\beta\sum_{j\leq n}(\sum_{k\leq \pi_j(\alpha p)}\theta_{k,j}(s_{k,j,1}, \ldots, s_{k,j,p-1}, \varepsilon_j)+ \psi(\varepsilon_j))]}\Big\rbrace_{l\leq L, i\leq n}\Bigr). \label{exp:F_pure1}
\end{align}
Note that $s_{k,j,r}=\sigma(w,u,v_{k,j,r},x_{k,j,r}).$ Let $\hat s_{k,j,r}=\sigma(w,u_{k,j,r},v_{k,j,r},x_{k,j,r})$ and 
\[ \hat \cE^j_\beta (\varepsilon):= \exp \Big( \beta\sum_{k \leq \pi_j(\alpha p)}\theta_{k,j}(\hat s_{k,j,1}, \ldots, \hat s_{k,j,p-1}, \varepsilon_j)+\beta \psi(\varepsilon_j)\Big).\]
If we are in the pure state, then by Corollary~\ref{cor ps}, we can substitute each $s_{k,j,r}$ by $\hat s_{k,j,r}$ in \eqref{exp:F_pure1} to obtain 
\begin{align*}
    & \E_{w,v}F\bigl(\{\P_{u,x}(\sigma(w, u, v_i, x_i) = t_l)\}_{l\leq L, i\leq n}\bigr) =  \E_{w,v}F\Bigl( \Big\lbrace \frac{\av \1_{\varepsilon_i = t_l}\E_{u,x} \prod_{j \le n} \hat \cE^j_\beta (\varepsilon) }{ \quad \av \E_{u,x} \prod_{j \le n} \hat \cE^j_\beta (\varepsilon)}\Big\rbrace_{l\leq L,i\leq n}\Bigr)\\
      & =  \E_{w,v}F\Bigl( \Big\lbrace \frac{\av \1_{\varepsilon_i = t_l} \prod_{j \le n} \E_{u,x}  \hat \cE^j_\beta (\varepsilon) }{ \quad \av  \prod_{j \le n} \E_{u,x} \hat \cE^j_\beta (\varepsilon)}\Big\rbrace_{l\leq L,i\leq n}\Bigr) \\     &=  \E_{w,v}F\Bigl( \Big\lbrace \frac{ \Big( \av \1_{\varepsilon_i = t_l} \E_{u,x}  \hat \cE^i_\beta (\varepsilon) \Big) \Big( \av \prod_{j \le n, j \ne i} \E_{u,x}  \hat \cE^j_\beta (\varepsilon) \Big)  }{ \Big( \av \E_{u,x}  \hat \cE^i_\beta (\varepsilon) \Big) \Big(   \av  \prod_{j \le n, j \ne i} \E_{u,x} \hat \cE^j_\beta (\varepsilon) \Big) }\Big\rbrace_{l\leq L,i\leq n}\Bigr)\\
    &=     \E_{w,v}F\Bigl( \Big\lbrace \frac{\nu(t_l)\E_{u,x} \exp [\beta\sum_{k \leq \pi_i(\alpha p)}\theta_{k,j}(\hat s_{k,i,1}, \ldots, \hat s_{k,i,p-1}, t_l) + \beta \psi(t_l)]}{\int\E_{u,x} \exp[\beta\sum_{k \leq \pi_i(\alpha p)}\theta_{k,i}(\hat s_{k,i,1}, \ldots, \hat s_{k,i,p-1}, \varepsilon) + \beta \psi(\varepsilon)]\nu(d\varepsilon)}\Big\rbrace_{l\leq L,i\leq n}\Bigr),
\end{align*}
where the second equality follows from the fact that $(\hat s_{k,j, r})_{k\geq 1,j\leq n,r\leq p-1}$ are independent conditionally on $w$ and $(v_{k,j, r})_{k\geq 1, j\leq n, r\leq p-1}$. It follows that
\begin{align*}
\bigl\{\P_{u,x}(s_i = t_l)\bigr\}_{l\leq L, i\leq n}\stackrel{d}{=}\Big\lbrace \frac{\nu(t_l)\E_{u,x} \exp[\beta\sum_{k \leq \pi_i(\alpha p)}\theta_{k,i}(\hat s_{k,i,1}, \ldots, \hat s_{k,i,p-1}, t_l)+ \beta \psi(t_l)]}{\int\E_{u,x} \exp[\beta \sum_{k \leq \pi_i(\alpha p)}\theta_{k,i}(\hat s_{k,i,1}, \ldots, \hat s_{k,i,p-1}, \varepsilon) + \beta \psi(\varepsilon)]\nu(d\varepsilon)}\Big\rbrace_{l\leq L,i\leq n}.
\end{align*}
In particular, this equation ensures that $\P_{u,x}(\sigma(w, u, v_i, x_i) = t_l)>0$ a.s. and since $\nu(t)>0$ for $t\in \Sigma_L$, these allow us to define 
\begin{align}\label{density}
X_{w,v}(t)=\frac{1}{\beta}\log \frac{\P_{u,x}(\sigma(w,u,v,x)=t)}{\nu(t)}\;\;\ \text{ for all }t\in \Sigma_L,
\end{align}
where $w,u,v,x$ are i.i.d. uniform on $[0,1].$
In other words, $X_{w,v}$ is the exponent of the relative density of $\P_{u,x}(\sigma(w,u,v,x)\in \cdot)$ with respect to $\nu.$ Consequently, 
\begin{align}\label{eq8}
	\bigl(X_{w,v_i}\bigr)_{ i\leq n}\stackrel{d}{=}\bigl(T_{\nu,\beta, i,\pi_i(\alpha p)}\bigl(\bigl(X_{w,v_{k,i,r}}\bigr)_{k\leq \pi_i(\alpha p),r\leq p-1}\bigr)\bigr)_{i\leq n},
\end{align}
where for all $t \in \Sigma_L$, \begin{align*}
    & T_{\nu,\beta, i,\pi_i(\alpha p)}\bigl(\bigl(X_{w,v_{k,i,r}}\bigr)_{k\leq \pi_i(\alpha p),r\leq p-1}\bigr)(t) \\
    & = \frac{1}{\beta}\log  \frac{\E_{u,x} \exp[\beta\sum_{k \leq \pi_i(\alpha p)}\theta_{k,i}(\hat s_{k,i,1}, \ldots, \hat s_{k,i,p-1}, t)+ \beta \psi(t)]}{\int\E_{u,x} \exp[\beta \sum_{k \leq \pi_i(\alpha p)}\theta_{k,i}(\hat s_{k,i,1}, \ldots, \hat s_{k,i,p-1}, \varepsilon) + \beta \psi(\varepsilon)]\nu(d\varepsilon)}.
\end{align*}
From Theorem \ref{uniqueness}, for $(X_i)_{i\geq 1}$ be i.i.d. sampled from  $\lambda_{\nu,\beta}$ we have that
\begin{align}\label{eq9}
	(X_i)_{i\leq n}\stackrel{d}{=}\bigl(T_{\nu,\beta,i,\pi_i(\alpha p)}(X_{i,1},\ldots,X_{i,(p-1)\pi(\alpha p)})\bigr)_{i\leq n},
\end{align}
where $(X_{i,k})_{i,k\geq 1}$ are i.i.d. sampled from $\lambda_{\nu,\beta}.$ The main result of this section is the following theorem, where we show that sequences in \eqref{eq8} and \eqref{eq9} have the same distribution.
 \begin{theorem} \label{thm fixed point}
Suppose that \eqref{ht} holds, i.e.,  $\min(1, 6\beta e^{4\beta \|\psi\|_\infty}\E \|\theta\|_\infty e^{4\beta \|\theta\|_\infty}) \alpha p(p-1)\leq 1$. Then \begin{align}\label{eqdist}
    (X_i)_{i\geq 1}\stackrel{d}{=}(X_{w,v_i})_{i\geq 1}.
\end{align}
\end{theorem}
The proof of this theorem is split into the subcritical (i.e., $\alpha p(p-1) \leq 1)$ and high temperature (i.e., $ 6\beta e^{4\beta\|\psi\|_\infty}\E \|\theta\|_\infty e^{4\beta \|\theta\|_\infty} \alpha p(p-1)\leq 1$) regimes.

\begin{proof}[Proof of Theorem \ref{thm fixed point} (Subcritical regime)]
    Let us regard the random functions $X_{w,v_i}$ and $X_i$ on $\Sigma_L$ as vectors in $\R^L$. For brevity, let us write \eqref{eq8} as \begin{align}\label{eq8.1}
        (X_{w,v_i})_{i\leq n} \stackrel{d}{=} \big(g(\xi_i, X_{w,v_{i,1}}, X_{w,v_{i,2}}, \ldots, X_{w,v_{i,(p-1)\pi_i(\alpha p)}})\big)_{i\leq n},
    \end{align}
    where $g$ is some fixed, measurable function, and the i.i.d. ${\rm Unif}([0,1])$ random variables $\xi_i$, drawn independently of $(\pi_i(\alpha p))_{i\leq n}$, capture the randomness in $(\theta_{k,i})_{k\geq 1, i\leq n}$. When $\pi_i(\alpha p) = 0$, we shall understand \begin{align*}
        g(\xi_i, X_{w,v_{i,1}}, X_{w,v_{i,2}}, \ldots, X_{w,v_{i,(p-1)\pi_i(\alpha p)}}) & =  y_0: = \Big(\psi(t_l) - \frac{1}{\beta}\log \Big[\sum_{\ell=1}^L e^{\beta \psi(t_{\ell})}\nu(t_{\ell})\Big]\Big)_{1\leq l\leq L}.
    \end{align*}
    Note that from \eqref{density}, we can write $X_{w,v_i}= \varphi(w,v_i)$ for some measurable function $\varphi$. Clearly, $(X_{w,v_i})_{i\geq 1}$ is exchangeable, and by de Finetti's Theorem (\cite[Theorem~1.6]{Panchenko2013TheSM}), any exchangeable sequence can be expressed in this form.

    Recall, from Lemma \ref{gen_fixed_pt} the set of vertices $\V$. Associated with the function $\varphi$, let us consider a process $(Y_\i)_{\i \in \V}$ that satisfies the following properties: \begin{enumerate}
        \item[(i)]$(\xi_\i, \pi_\i(\alpha p))_{\i\in \V}$ are i.i.d. copies of $(\xi, \pi(\alpha p))$,
        \item[(ii)] almost surely, for all $\i \in \V$, $Y_\i = g(\xi_\i, Y_{\i 1}, Y_{\i 2} ,\ldots, Y_{\i (p-1)\pi_\i(\alpha p)})$, and
        \item[(iii)] for each $r\geq 1$, $(Y_\i)_{\i \in \V_{=r}} \stackrel{d}{=} (\varphi(w,v_\i))_{\i \in \V_{=r}}$, where $w, v_\i$ are i.i.d. ${\rm Unif}([0,1])$. Also, $(Y_\i)_{\i \in \V_{=r}}$ is independent of $(\xi_\i, \pi_\i(\alpha p))_{\i \in \V_{=r}}$.
    \end{enumerate}
    We claim that such a process exists. For a level $r\geq 1$, let $(Y_\i)_{\i\in \V_{=r}}$ be specified by (iii) and $(\xi_\i, \pi_\i(\alpha p))_{\i \in \V_{\leq r}}$ be i.i.d. copies of $(\xi, \pi(\alpha p))$. For levels $k = r-1, r-2, \ldots, 0$, we recursively define $(Y_\i)_{\i\in V_{=k}}$ using (ii). Since $\varphi$ satisfies \eqref{eq8.1}, we have for each $0\leq k\leq r-1$, \begin{align*}
        (Y_\i)_{\i \in \V_{=k}} = \big( g(\xi_\i, Y_{\i 1}, Y_{\i 2} ,\ldots, Y_{\i (p-1)\pi_\i(\alpha p)})\big)_{\i \in \V_{=k}} \stackrel{d}{=} (\varphi(w,v_\i))_{\i \in \V_{=k}},
    \end{align*}
    so (iii) is satisfied for each $k \leq r-1$. Thus, for any $\i\in \V_{\leq r-1}$, $Y_\i$ is a function of $\{(\xi_{\i'}, \pi_{\i'}(\alpha p)): \i' = \i \text{ or } \i' \prec \i, |\i'| \leq r-1\}$ and $\{Y_{\i'}:|\i'| = r, \i' \prec \i\}$. A finite collection of the variables $(Y_\i)_{\i \in \V_{\leq r}}$ and $(\xi_\i, \pi_\i(\alpha p))_{\i \in \V_{\leq r}}$ specifies a distribution $\mu_r$ on $(\R^L)^{\V_{\leq r}} \times ([0,1]\times \Z_+)^{\V_{\leq r}}$ and the coordinate maps on this probability space satisfy (i), (ii) and (iii) for all $|\i'| \leq r$. By Kolmogorov's consistency theorem, there exists a probability measure $\mu$ on $(\R^L)^\V\times ([0,1] \times \Z_+)^\V$ whose coordinate maps, which we still call $(Y_\i)_{\i \in \V}, (\xi_\i, \pi_\i(\alpha p))_{\i \in \V}$ by an abuse of notation, satisfy (i), (ii) and (iii). 

    Since $\alpha p(p-1) \leq 1$, the tree $\G_\varnothing$ is almost surely finite. On this event we have that for each leaf ${\bm \ell}$ of $\G_\varnothing$, $Y_{\bm \ell} = y_0$ and thus, $Y_\varnothing$ is uniquely defined in terms of $(\xi_\i, \pi_\i(\alpha p))_{\i \in \G_\varnothing}$. For each vertex $j \in \N$ adjacent to $\varnothing$, we have that the tree rooted at $j$, i.e., the connected component of $j$ comprising vertices of generations $\geq 1$, call it  $\G_j$, is finite almost surely. So by a similar argument as that presented for $Y_\varnothing$, we have that $Y_j$ is uniquely defined by $(\xi_\i, \pi_\i(\alpha p))_{\i \in \G_j}$ and has the same distribution as that of $Y_\varnothing$. Because of (i), we obtain that $(Y_j)_{j\geq 1}$ are i.i.d. By (ii), we have that $Y_\varnothing = \T_{\nu,\beta}(Y_\varnothing)$, where by an abuse of notation, we identified $Y_\varnothing$ with its law. By Theorem \ref{uniqueness}, we have that the law of $Y_\varnothing$ is $\lambda_{\nu,\beta}$, and thus \begin{align*}
        (X_j)_{j\geq 1} \stackrel{d}{=} (Y_j)_{j\geq 1} \stackrel{d}{=}(\varphi(w,v_j))_{j\geq 1} = (X_{w,v_j})_{j\geq 1}.
    \end{align*}
\end{proof}

\begin{proof}[Proof of Theorem \ref{thm fixed point} (High temperature regime)]
	Let $n\geq 1$ be fixed. Consider the complete and separable metric space $(\X^n,d_n)$ with the metric $$d_n(f,f'):=\sum_{i=1}^n \|f_i-f_i'\|_\infty$$ for $f=(f_1,\ldots,f_n)$ and $f'=(f_1',\ldots,f_n')\in \X^n.$ Denote by $\mathcal{B}_n$ the corresponding Borel $\sigma$-field. For probability measures $Q_1$ and $Q_2$ defined on $(\X^n,\mathcal{B}_n)$ with $\int d_n(f,0)Q_i(df)<\infty$ for $i=1,2$, we define the Wasserstein $1$-distance between them as 
	$$
	W_1(Q_1,Q_2)=\inf_{\Pi\in \Pi(Q_1,Q_2)}\int d_n(f,f')d\Pi(f,f'),
	$$
	where $\Pi(Q_1,Q_2)$ is the collection of all couplings between $Q_1$ and $Q_2.$ Let $Q$ and $Q'$ be the laws of $(X_i)_{i\leq n}$ and $\bigl(X_{w,v_i}\bigr)_{ i\leq n}$ respectively and $D(n) = W_1(Q,Q')$. Note that from \eqref{eq8} and \eqref{eq9},
	$$
	Q=\sum_{m}p_mP_{m}\;\;\mbox{and}\;\;Q'=\sum_mp_mP_m'
	$$
	 for $m=(m_1,\ldots,m_n)\in \Z_+^n$, where $p_m=\prod_{i=1}^n\P(\pi(\alpha p)=m_i)$  and $P_m$ and $P_m'$ are the laws of 
	$$
	\bigl(T_{\nu,\beta,i,m_i}((X_{i,(k-1)m_i+r})_{k\leq m_i,r\leq p-1})\bigr)_{i\leq n}\;\;\mbox{and}\;\;\bigl(T_{\nu,\beta,i,m_i}((X_{w,v_{k,i,r}})_{k\leq m_i,r\leq p-1}\bigr)_{i\leq  n}
	$$
	respectively. For any integers $m_1,\ldots,m_n\geq 0$, let $$Z=(Z_{k,i,r},Z_{k,i,r}')_{k \leq {m_i},i\leq n,r\leq p-1}$$ be an arbitrary coupling between $(X_{(k-1)(p-1)+r,i})_{{k \leq {m_i},i\leq n,r\leq p-1}}$ and $(X_{w,v_{k,i,r}})_{k \leq {m_i},i\leq n,r\leq p-1}$. Using \eqref{eq8} and \eqref{eq9}, we see from \eqref{almost_contraction} that
\begin{align*}
	W_1(P_m,P_m')&\leq \E\sum_{i=1}^n \|T_{\nu,\beta,i,m_i}\bigl((Z_{k,i,r})_{k \leq {m_i},i\leq n,r\leq p-1}\bigr) - T_{\nu,\beta,i,m_i}\bigl((Z_{k,i,r}')_{k \leq {m_i},i\leq n,r\leq p-1}\bigr)\|_\infty\\
	&\leq \gamma \E\sum_{i=1}^n \sum_{k=1}^{m_i}\sum_{r=1}^{p-1}\|Z_{k,i,r}-Z'_{k,i,r}\|_\infty,
\end{align*}
where $\gamma: = 4\beta\E\|\theta\|_\infty e^{2\beta \|\theta\|_\infty}$. This implies that
\begin{align*}
	W_1(P_m,P_m')\leq \gamma D\Bigl((p-1)\sum_{i=1}^nm_i\Bigr).
\end{align*}
Thus, together with the following convexity property exhibited by the Wasserstein 1-distance
	$$
	W_1\Bigl(\sum_{m}p_mP_{m},\sum_{m}p_mP_{m}'\Bigr)\leq \sum_{m}p_m W_1(P_m,P_m'),
	$$
we obtain that $D(n)\leq \gamma \E D((p-1)\pi(\alpha p n)).$

Let $\delta =4\alpha p\E\|\theta\|_\infty + 4\|\psi\|_\infty$. We claim that $D(n) \leq (\gamma \alpha p (p-1))^j\delta n$ for all $j \geq 0$ and $n\geq 1$. To see this, note that from \eqref{eq8}, \eqref{eq9} and \eqref{l_infty bound} that $D(n)\leq \delta n$, so our claim is true for $j=0$. Suppose that our claim is true for all $j = 0, \ldots, j_0$ and $n\geq 1$. Then \begin{align*}
    D(n) \leq \gamma \E D\bigl((p-1)\pi(\alpha p n)\bigr) \leq \gamma (\gamma \alpha p(p-1))^{j_0}\delta (p-1)\E \pi(\alpha p n) = (\gamma \alpha p (p-1))^{j_0+1}\delta n.
\end{align*}
Note that by our assumption and Lemma \ref{lem hi_temp}, $\gamma \alpha p (p-1) < 1$. Thus taking $j\to \infty$ above, we have that $D(n) = 0$ for all $n \geq 1$, concluding the proof of the lemma.
\end{proof}

We finish this section by identifying the limiting joint density of the spins. 

\begin{cor}\label{rsansatz}
    Let $\Sigma$ be a collection of finitely many points and $\nu(t) > 0$ for all $t \in \Sigma$. Suppose that \eqref{ht} holds. Then for any $n\geq 1$, the following sequence
    \begin{align*}
    \Big( \Bigl(\frac{1}{\beta}\log \frac{G_{N,\beta, \nu}(\sigma_1 = t_1)}{\nu(t_1)}\Bigr)_{t_1\in \Sigma}, \ldots,  \Bigl(\frac{1}{\beta}\log \frac{G_{N,\beta, \nu}(\sigma_n = t_n)}{\nu(t_n)}\Bigr)_{t_n\in \Sigma}\Big) 
    \end{align*}
converges weakly to $$\big((X_1(t_1))_{t_1\in \Sigma}, \ldots, (X_n(t_n))_{t_n\in \Sigma}\big),$$
where $(X_i)_{i\geq 1}$ are i.i.d.\ samples from $\lambda_{\nu,\beta}$, the unique fixed point of $\T_{\nu,\beta}$ guaranteed by Theorem~\ref{thm1}.
\end{cor}
\begin{proof}
    Let $n \ge 1$ be fixed and $m$ be the cardinality of $\Sigma$. $F:\R^{mn} \to \R$ and $f_{1,1}, \ldots, f_{m,n}:\Sigma^n\to \R$ be continuous functions. From \cite[Lemma 1]{Panchenko2010SPINGM} and Theorem \ref{thm fixed point} we obtain that
    \begin{align*}
        \lim_{N\to \infty} \E F \big(\la f_{1,1}(\sigma_1, \ldots, \sigma_n)\ra_{N,\beta,\nu}, \ldots, \la f_{m,n}(\sigma_1, \ldots, \sigma_n)\ra_{N,\beta,\nu}\big) = \E F\big(\la f_{1,1} \ra_{\beta,X} , \ldots, \la f_{m,n} \ra_{\beta,X} \big).
    \end{align*}
    For $t_{1,1}, \ldots, t_{m,n} \in \Sigma$ let $f_{i,j}(\sigma_1, \ldots, \sigma_n) = \1_{\sigma_j = t_{i,j}}$. Then the above display implies the joint convergence \begin{align*}
    \big((G_{N,\nu,\beta}(\sigma_j = t_{i,j}))_{i\leq m, j\leq n}\big) \stackrel{d}{\to} \big((e^{\beta X_j(t_{i,j})}\nu(t_{i,j}))_{i\leq m, j\leq n}\big),
\end{align*}
which by the continuous mapping theorem yields the desired result.
\end{proof}

\section{Free energy of the dilute model}\label{fe_dilute_model}
This section is dedicated to the proof of Theorem \ref{thm1}. We note that the first part has been done in Theorem \ref{uniqueness}. Thus, we shall focus on proving the second part, which consists of three major steps. In the first step,  we assume that $\nu$ is supported on a finite set and obtain an expression for the limiting free energy. In the second step, we extend the validity of the expression to the situation when $\nu$ is supported on a compact set. In the final step, we establish the concentration of the free energy around its mean, thereby completing the proof of the $L^1$ convergence of the free energy to the expression $\cP_{\nu, \beta}(\lambda_{\nu, \beta})$.

\subsection{Finite support case}
    Note that since the perturbation in $\hat H_N$ (defined in \eqref{pert_H}) is $o(N)$, the free energies corresponding to the Hamiltonians $\hat H_N$ and $H_N$ are the same in the limit. Hence, with a slightly abuse of notation, we let $F_N(\beta)$ and $Z_N(\beta)$ be the free energy and partition function corresponding to $\hat H_N$ at temperature $\beta > 0$. We note that \begin{align*}
		\liminf_{N\to \infty} \E \log\frac{Z_{N+1}(\beta)}{Z_N(\beta)} \leq \liminf_{N\to \infty} F_N(\beta) \leq \limsup_{N\to \infty} F_N(\beta) \leq \limsup_{N\to \infty} \E \log\frac{Z_{N+1}(\beta)}{Z_N(\beta)}.
	\end{align*}
To establish our proof, it suffices to show that for any subsequence $(N_k)_{k\geq 1}$, 
$\E \log{Z_{N_k+1}(\beta)}/{Z_{N_k}}(\beta)$ converges to the same limit. The proof of this part is similar to that of \cite[Lemma 4]{Panchenko2010SPINGM}, so we shall only provide a brief sketch here. 

Fist of all, note that $\E \log{Z_{N+1}(\beta)}/{Z_N}(\beta)$ is bounded due to \eqref{bdd theta} and the boundedness of $\psi$. So, by relabelling an appropriate convergent subsequence, we can assume without loss of generality that $\E \log{Z_{N+1}(\beta)}/{Z_N}(\beta)$ converges and further that $(\mu_{N,\beta})_{N\geq 1}$ converges weakly to some $\mu\in \M_{\beta,\nu}$. The basic idea is to use the Aizenman-Sims-Starr scheme \cite{Aizenman_2003} by splitting the perturbed Hamiltonian $\hat H_{N+1}$ into three parts,   \begin{align*}
		\hat H_{N+1}(\sigma) & = \sum_{k\leq \pi(\alpha(N-p+1))}\theta_k(\sigma_{I(k,1)}, \ldots, \sigma_{I(k, p)}) + \sum_{i=1}^N\psi(\sigma_i)\\
  &\qquad + \sum_{k\leq \pi(\alpha p)}\hat \theta_k(\sigma_{J(k, 1)}, \ldots, \sigma_{J(k,p-1)}, \sigma_{N+1}) +\psi(\sigma_{N+1}) \\
  &\qquad + \sum_{l\leq \pi(c_{N+1})}\log \av_\tau \exp\sum_{k\leq \pi_l(\alpha p)}\theta_{l,k}(\sigma_{I(l,k,1)}, \ldots, \sigma_{I(l,k,p-1)}, \tau)
	\end{align*}
 for $\sigma\in \Sigma^{N+1},$
	where $(\theta_k)_{k\geq 1}$, $(\hat \theta_k)_{k\geq 1}$, $(\theta_{l,k})_{l,k\geq 1}$ are copies of $\theta$ and $(I_k)_{k\geq 1}$, $(J_k)_{l,k\geq 1}$, $(I_{l,k})_{l,k\geq 1}$ are uniformly chosen from $\binom{[N]}{p}$, $\binom{[N]}{p-1}$ and $\binom{[N+1]}{p-1}$ respectively, all independently of each other. Since with a high probability, all the indices appearing in $\cup_{l, k} I_{l,k}$ are not larger than $N$ and $|c_{N+1} - c_N| \to 0$, we can replace the $\pi(c_{N+1})$ in the perturbation term of the above display by $\pi(c_N)$ for $\hat H_{N+1}$ without altering the free energy asymptotically. Letting $\rho = (\sigma_1, \ldots, \sigma_N)$ and $\varepsilon = \sigma_{N+1}$, we let $H_N'(\rho)$ be defined as follows:
 \begin{align*}
     H_N'(\rho) & = \sum_{k\leq \pi(\alpha(N-p+1))}\theta_k(\rho_{I(k,1)}, \ldots, \rho_{I(k, p)}) + \sum_{i=1}^N\psi(\rho_i) \\
     & \qquad + \sum_{l\leq \pi(c_N)}\log \av_\tau \exp\sum_{k\leq \pi_l(\alpha p)}\theta_{l,k}(\rho_{I(l,k,1)}, \ldots, \rho_{I(l,k,p-1)}, \tau),
 \end{align*}
 where $I_{l,k}$ are now uniformly chosen samples from $\binom{[N]}{p-1}$.  Let $Z_N'(\beta)$ be the corresponding partition function. This enables us to write \begin{align*}
		\E \log \frac{Z_{N+1}(\beta)}{Z_N'(\beta)} = \E \log \Big\langle \av  \exp\Big(\beta\sum_{k\leq \pi(\alpha p)}\hat \theta_k(\rho_{J(k, 1)}, \ldots, \rho_{J(k,p-1)}, \varepsilon) + \beta \psi(\varepsilon)\Big) \Big\rangle_\beta',
	\end{align*}
    where $\langle \cdot \rangle'$ denotes the Gibbs average corresponding to $H_N'$. Note that $\hat H_N$ differs from $H_N'$ by $\pi(\alpha (p-1))$ many terms, so for any function $g$ depending on finitely many spins, $\E \langle g\rangle_\beta$ and $\E \langle g\rangle'_\beta$ are asymptotically equal, as guaranteed by \cite[Lemma 2]{Panchenko2010SPINGM}. Thus, conditioned on $\pi(\alpha p)$ and $\hat \theta_k$, we can replace the average $\langle \cdot \rangle'_\beta$ in the display above by $\langle \cdot \rangle_\beta$ without affecting the limit. In addition, noting that the indices appearing in $\cup_k J_k$ are all distinct with a high probability, we use the weak convergence of $\mu_{N,\beta}$ to $\mu$, in particular \cite[Lemma 1]{Panchenko2010SPINGM}, and the dominated convergence theorem (note that $\|\theta\|_\infty$ is integrable by \eqref{bdd theta}) to conclude that \begin{align*}
		& \E \log \Big\langle \av  \exp\Big(\beta\sum_{k\leq \pi(\alpha p)}\hat \theta_k(\rho_{J(k, 1)}, \ldots, \rho_{J(k,p-1)}, \varepsilon) + \beta \psi(\varepsilon)\Big) \Big\rangle_\beta \\
  & \to \E \log \E_{u,x}\av \exp\Big(\beta\sum_{k\leq \pi(\alpha p)}\hat \theta_k(s_{k, 1}, \ldots, s_{k,p-1}, \varepsilon) + \beta \psi(\varepsilon)\Big).
	\end{align*}
  
	Integrating over the randomness we conditioned over and using \eqref{bdd theta}, it follows that the above convergence holds unconditionally. A similar argument can be used to show that 
	\begin{align*}
		\E \log\frac{Z_N(\beta)}{Z_N'(\beta)} &\to \E \log \E_{u,x} \exp \beta\sum_{k\leq \pi(\alpha(p-1))}\hat \theta_k(s_{k,1}, \ldots, s_{k, p}) \\
  &= \alpha (p-1) \E \log \E_{u,x} \exp \beta\theta(s_1, \ldots, s_p).
	\end{align*}
	 Putting the two limits above together, we arrive at \begin{align*}
		& \lim_{N\to \infty} \E \log\frac{Z_{N+1}(\beta)}{Z_N(\beta)} = \lim_{N\to \infty} \E \log\frac{Z_{N+1}(\beta)}{Z'_N(\beta)} - \lim_{N\to \infty} \E \log\frac{Z_N(\beta)}{Z'_N(\beta)}\\
        & = \E \log \E_{u,x} \av \exp \Big(\beta\sum_{k\leq \pi(\alpha p)}\theta_k(s_{k, 1}, \ldots, s_{k,p-1}, \varepsilon) + \beta \psi(\varepsilon)\Big)\\
        &\qquad- \alpha (p-1) \E \log \E_{u,x} e^{\beta\theta(s_1, \ldots, s_p)}\\
		& = \E \log \E_{u,x} \av \exp \Big(\beta \sum_{k\leq \pi(\alpha p)}\theta_k(\sigma(w, u, v_{k, 1}, x_{k,1}), \ldots, \sigma(w,u,v_{k,p-1}, x_{k, p-1}), \varepsilon) + \beta \psi(\varepsilon)\Big)\\
	& \qquad - \alpha (p-1) \E \log \E_{u,x} \exp \beta\theta(\sigma(w, u, v_1, x_1), \ldots, \sigma(w, u, v_p, x_p)).
	\end{align*}
Letting $(u_{k,j})_{k\geq 1,j\leq p}$ and $(u_j)_{j\leq p}$ be i.i.d. uniform on $[0,1]$ independent of any other randomness, we can use Corollary \ref{cor ps} to obtain the above limit as 
	 \begin{align*}
		&  \E \log \E_{u,x} \av \exp \Big(\beta\sum_{k\leq \pi(\alpha p)}\theta_k(\sigma(w, u_{k,1}, v_{k, 1}, x_{k,1}), \ldots, \sigma(w,u_{k, p-1},v_{k,p-1}, x_{k, p-1}), \varepsilon) + \beta \psi(\varepsilon)\Big)\\
		&\qquad - \alpha (p-1) \E \log \E_{u,x} \exp \beta\theta(\sigma(w, u_1, v_1, x_1), \ldots, \sigma(w, u_p, v_p, x_p)).
		\end{align*}
Recalling \eqref{density} we finally write the above as 
	 \begin{align*}
		& \E \log \int e^{\beta  \psi(\varepsilon)} \prod_{k\leq \pi(\alpha p)} \Big[\exp \bigl(\beta\theta_k(\rho_{k,1}, \ldots, \rho_{k,p-1}, \varepsilon)\bigr)\prod_{i=1}^{p-1}\exp\bigl(\beta X_{w,v_{k,i}}(\rho_{k,i})\bigr)\nu (d\rho_{k,i})\Big] \nu (d\varepsilon)\\
		& \qquad - \alpha(p-1)\E \log \int \exp \big(\beta \theta(\rho_1, \ldots, \rho_p)\big) \prod_{i=1}^p \exp \big(\beta X_{w,v_i}(\rho_i)\big)\nu (d\rho_i)\\
		& =\cP_{\nu,\beta}(\lambda_{\nu, \beta}),
	\end{align*}
where the last equality used Theorem \ref{thm fixed point}.

\subsection{Compact support case}
We extend the result of the previous section to the case when $\nu$ is supported on the compact set $\Sigma$. Throughout this section, we will use $F_N(\beta)$ to denote the free energy corresponding to the unperturbed Hamiltonian \eqref{eq:orig hamiltonian}. For any $r\geq 1$, let $\Sigma_r \subseteq \Sigma$ be a $1/r$-net (with the $L_1$ metric) of $\Sigma$ of finite cardinality, which exists since $\Sigma$ is compact. For $x \in \Sigma$, let $\varrho_r(x)\in \Sigma_r$ be a point closest to $x$ and for $\sigma \in \Sigma^N$, let $\varrho_r(\sigma) = (\varrho_r(\sigma_1), \ldots, \varrho_r(\sigma_N))$.  Consider the following interpolating Hamiltonian \begin{align*}
       H_{N,r,t}(\sigma) = t H_N(\sigma) + (1-t)H_N(\varrho_r(\sigma))
    \end{align*}
    and the associated free energy $$ F_{N,r}(\beta, t) = \frac{1}{N}  \log \int e^{\beta H_{N,r,t}(\sigma)}\nu^{\otimes N}(d\sigma).$$ Let $\nu_r$ 
    be the push-forward of the measure $\nu$ under the map $\varrho_r: \Sigma \to \Sigma$. Then $\nu_r$ is a discrete measure on $\Sigma$ and $\nu_r\stackrel{d}{\to}\nu$.
    Also, note that $F_{N,r}(\beta,1) = F_N(\beta)$ and \begin{align*}
        F_{N,r}(\beta,0) = \frac{1}{N} \log \int e^{\beta H_N(\varrho_r(\sigma))}\nu^{\otimes N}(d\sigma) = \frac{1}{N} \log \int e^{\beta H_N(\sigma)}\nu_r^{\otimes N}(d\sigma).
    \end{align*}
    Now, \begin{align*}
        \frac{\partial}{\partial t} F_{N,r}(\beta,t) =\frac{\beta}{N} \langle H_N(\sigma) - H_N(\varrho_r(\sigma))\rangle_{\beta, N, r, t},
    \end{align*} 
    where $\langle \cdot\rangle_{\beta, N,r,t}$ is the Gibbs average with respect to the Hamiltonian $H_{N,r,t}$. Hence, \begin{align}  \label{eq disc approx}
        \limsup_{N\to \infty} |\E F_N(\beta) - \E F_{N,r}(\beta,0)| & \leq \limsup_{N\to \infty} \E\Big|\frac{\partial}{\partial t} F_{N,r}(\beta, t)\Big| \leq \frac{\beta}{r} \big(\alpha \E \text{Lip}(\theta) + \text{Lip}(\psi)\big).
    \end{align} 
   Then we have \begin{align}
        \limsup_{N\to \infty} |\E F_N(\beta) - \cP_{\nu,\beta}(\lambda_{\nu,\beta})| & \leq \lim_{r\to \infty}\limsup_{N\to\infty}|\E F_N(\beta) - \E F_{N,r}(\beta,0)| \label{eqn:first term}\\
        & + \lim_{r\to \infty}\limsup_{N\to \infty}|\E F_{N,r}(\beta,0) - \cP_{\nu_r,\beta}(\lambda_{\nu_r,\beta})| \label{eqn:second term} \\
        & + \lim_{r\to \infty} |\cP_{\nu_r,\beta}(\lambda_{\nu_r,\beta}) - \cP_{\nu,\beta}(\lambda_{\nu,\beta})| \label{eqn:last term}
    \end{align}
    We now show that each of the terms above vanishes. By \eqref{bdd theta} and  \eqref{eq disc approx}, it is easy to see that \eqref{eqn:first term} vanishes as $r\to\infty$. The second term \eqref{eqn:second term} vanishes as a consequence of the convergence of the free energy in the discrete case developed in the previous subsection. 
   
   The term  \eqref{eqn:last term} can be handled as follows. Recall that $\mathcal{P}_{\nu_r,\beta}$ and $\mathcal{P}_{\nu,\beta}$ involve two major logarithmic terms. We treat the convergence of their second terms first. Equip the space $\X^p$ with the norm, $\|\f\|_\infty:=\sum_{i=1}^p\|f_i\|_\infty$ for $\f = (f_1, \ldots, f_p) \in \X^p$. Let \begin{align*}
          \Phi_\f =   \Phi_{\theta, \f}(\rho): = \exp \beta \Big(\theta(\rho_1, \ldots, \rho_p) + \sum_{i=1}^p f_i(\rho_i)\Big)
        \end{align*}
        and 
        \[  \nu( \Phi_\f)  :=  \log \int \Phi_{\theta, \f}(\rho) \nu^{\otimes p}(d\rho).  \]
       It is easy to check that \begin{align}\label{bound1}
        |\nu(\Phi_\f)| \leq \beta(\|\theta\|_\infty + \|\f\|_\infty)\;\;\text{and}\;\;  |\nu_r(\Phi_\f)| \leq \beta(\|\theta\|_\infty +\|\f\|_\infty)
    \end{align}
    and \begin{align}\label{bound2}
       |\nu_r(\Phi_\f) - \nu_r(\Phi_{\f'})| \leq \beta \|\f-\f'\|_\infty.
    \end{align}
        From Theorem \ref{continuity},  $W_1(\lambda_{\nu_r,\beta}^{\otimes p},\lambda_{\nu,\beta}^{\otimes p})\to 0$, which is equivalent to $\lambda_{\nu_r,\beta}^{\otimes p}\stackrel{d}{\to} \lambda_{\nu,\beta}^{\otimes p}$ and
  $$
  \int \|\f\|_\infty \lambda_{\nu_r,\beta}^{\otimes p}(d\f)\to\int \|\f\|_\infty \lambda_{\nu,\beta}^{\otimes p}(d\f).
  $$
 Since $\mathcal{X}^p$ is a separable space, the support of $\lambda_{\nu,\beta}^{\otimes p}$ is separable as well. As a result, from Skorokhod's representation theorem, there exist $\mathcal{X}^p$-valued random variables $(Y_r)_{r\geq 1}$ and $Y$ defined on a common probability space such that the laws of $Y_r$ and $Y$ are respectively equal to $\lambda_{\nu_r,\beta}$ and $\lambda_{\nu,\beta}$ and $Y_r\to Y$ a.s. Now write
 \begin{align}
  \nonumber & \Big|\E_\theta\int  \nu_r (\Phi_\f)   \lambda_{\nu_r,\beta}^{\otimes p}(d\f) -  \E_\theta\int  \nu (\Phi_\f) \lambda_{\nu,\beta}^{\otimes p}(d\f)\Big|\\
 \nonumber &=\bigl|\E\nu_r(\Phi_{Y_r})-\E \nu(\Phi_Y)\bigr|\\
    \label{add:eq1} &\leq \E\bigl|\nu_r(\Phi_{Y_r})-\nu_r(\Phi_Y)\bigr|+\E\bigl|\nu_r(\Phi_Y)-\nu(\Phi_Y)\bigr|,
 \end{align}
 where $\E_\theta$ is the expectation with respect to $\theta$ only. 
Note that for every fixed $\f$, conditional on $\theta$, the map $\rho \in \Sigma^p \mapsto \Phi_\f(\rho)$ is bounded and continuous. Hence, $\nu_r(\Phi_Y)\to \nu(\Phi_Y)$ a.s. and from \eqref{bound1}, $|\nu_r(\Phi_Y)-\nu(\Phi_Y)|\leq 2\beta (\|\theta\|_\infty+\E\|Y\|_\infty),$ which is integrable. It follows that the second term in the last display vanishes as $r$ tends to infinity. As for the first term, note that from \eqref{bound2}, 
  \begin{align*}
      \E\bigl|\nu_r(\Phi_{Y_r})-\nu_r(\Phi_Y)\bigr|&\leq \beta \E\|Y_r-Y\|_\infty.
  \end{align*}
  Since $Y_r\to Y$ in probability and $\E \|Y_r\|_\infty\to \E\|Y\|_\infty,$ by \cite[Theorem 4.6.3]{Durrett_2019} or Theorem \ref{thm A1}, we see that $\E \|Y_r-Y\|_\infty\to 0.$ These together imply that \eqref{add:eq1} vanishes and it is the same as 
  \begin{align}\label{add:eq2}
       &\lim_{r\to\infty}\bigl|\E \log \la e^{\beta \theta(\sigma)}\ra_{\beta,X_r}-\E \log \la e^{\beta \theta(\sigma)}\ra_{\beta,X}\bigr|=0
  \end{align}
  for $X_r\sim \lambda_{\nu_r,\beta}$ and $X\sim \lambda_{\nu,\beta}.$ 
  In a similar manner, for $k\geq 0$, let $\f = (f_{j,i})_{j\leq k, i\leq p-1}$ and \begin{align*}
      \Psi_{k,\f} = \Psi_{k, (\theta_j)_{j\leq k}, \f}(\rho, \varepsilon) := \exp \beta\Big(\sum_{j=1}^k \Big(\theta_k(\rho_{k,1}, \ldots, \rho_{k, p-1}, \varepsilon) + \sum_{i=1}^{p-1}f_{k, i}(\rho_{k, i})\Big)  +  \psi(\varepsilon)\Big)
  \end{align*}
  and \begin{align*}
      \nu(\Psi_{k,\f}) := \log \int \Psi_k(\rho,\varepsilon) \nu^{\otimes (k(p-1) + 1)}(d\rho, d\varepsilon).
  \end{align*}
   The same argument as above yields that for each $k \geq 0$, \begin{align*}
      \lim_{r\to\infty}\Big|\E_\Theta\int \nu_r(\Psi_{k,\f}) \lambda_{\nu_r,\beta}^{\otimes k(p-1)}(d\f) - \E_\Theta\int \nu(\Psi_{k,\f}) \lambda_{\nu,\beta}^{\otimes k(p-1)}(d\f)\Big| =0,
  \end{align*}
  where $\E_\Theta$ is the expectation with respect to $\Theta = (\theta_k)_{k\leq \alpha N}$.
  Finally, for any $r \in \Z_+$, since \begin{align*}
      \Big|\E_\Theta\int \nu_r(\Psi_{k,\f}) \lambda_{\nu_r,\beta}^{\otimes k(p-1)}(d\f)\Big| \leq \E_\Theta\int |\nu_r(\Psi_{k,\f})| \lambda_{\nu_r,\beta}^{\otimes k(p-1)}(d\f) \leq \beta (k\E\|\theta\|_\infty + \|\psi\|_\infty),
  \end{align*}
  we can use the dominated convergence theorem with respect to $\pi(\alpha p)$ to get \begin{align*}
      \lim_{r\to\infty}\Bigl|\E\log \Bigl\la \int\mathcal{E}_{\beta,\pi(\alpha p)}(\sigma,\varepsilon)\nu_r(d\varepsilon)\Bigr\ra_{\beta,X_r}-\E\log \int\Bigl\la \mathcal{E}_{\beta,\pi(\alpha p)}(\sigma,\varepsilon)\nu(d\varepsilon)\Bigr\ra_{\beta,X}\Bigr|=0.
  \end{align*}
  Together with \eqref{add:eq2}, \eqref{eqn:last term} vanishes and our proof is completed.

\subsection{Concentration of the free energy}
 In this section, we will show that the free energy is concentrated around its mean. We denote by $F_N(\beta)$ the free energy corresponding to the unperturbed Hamiltonian \eqref{eq:orig hamiltonian} at temperature $\beta \in (0, \infty)$.
 
 \begin{prop}\label{concentration}
      There exists a constant $K$ depending only on $\alpha$ and $\beta$ such that for all $N\geq 1$, \begin{align*}
        \E |F_N(\beta) - \E F_N(\beta)| \leq \frac{K}{\sqrt{N}}(\E\|\theta\|^2_\infty)^{1/2}.
    \end{align*}
 \end{prop}
The proof of this proposition is standard, see \cite{10.1214/22-AOP1597} for example. Nevertheless, we include it here for completeness. Throughout this section, random variables appearing in the subscripts to $\E$ shall mean that the expectation is carried out with respect to these random variables only.

\begin{proof}
   As earlier denote $\Theta = (\theta_k)_{k\leq \pi(\alpha N)}$ and $I = \{I(k,i):k\leq \pi(\alpha N), i\leq p\}$. Then we have from Jensen's inequality that \begin{align*}
    \E|F_N(\beta) - \E F_N(\beta)| & \leq \E|F_N(\beta) - \E_{\pi(\alpha N)}F_N(\beta)| + \E|\E_{\pi(\alpha N)}F_N(\beta) - \E_{\pi(\alpha N),\Theta, I}F_N(\beta)|\\
    & \leq \E|F_N(\beta) - \E_{\pi(\alpha N)}F_N(\beta)| + \E|F_N(\beta) - \E_{\Theta, I}F_N(\beta)|
\end{align*}
so to prove the proposition, it suffices to show that each of the terms in the last line above vanishes as $N\to \infty$. This is the content of the next two lemmas.
\end{proof}

\begin{lemma}
There exists a constant $K>0$ depending only on $\alpha$ and $\beta$ such that for each $N\geq 1$,
    $$\E|F_N(\beta) - \E_{\pi(\alpha N)}F_N(\beta)| \leq \frac{K}{\sqrt{N}}\E\|\theta\|_\infty.$$
\end{lemma}

\begin{proof}
    Let $\hat\pi(\alpha N)$ be a copy of $\pi(\alpha N)$, independent of everything else and $\hat F_N$ be the free energy when $\pi(\alpha N)$ is replaced by $\hat\pi(\alpha N)$. Without loss of generality, assume that $\hat\pi(\alpha N) \geq \pi(\alpha N)$. Using $|\log (x/y)| \leq |x-y|/\min(x,y)$ for $x,y>0$, it is easy to see that $|\hat F_N(\beta) - F_N(\beta)| \leq \beta N^{-1}\sum_{\pi(\alpha N) < k\leq \hat\pi(\alpha N)}\|\theta_k\|_\infty$, so by Jensen's and the Cauchy-Schwarz inequalities, \begin{align*}
        \E|F_N(\beta) - \E_{\pi(\alpha N)}F_N(\beta)| & \leq \E|\hat F_N(\beta) - F_N(\beta)| \leq \frac{\beta\E\|\theta\|_\infty}{N}\E|\pi(\alpha N) - \pi'(\alpha N)| \\
        & \leq \frac{\beta\E\|\theta\|_\infty}{N}\sqrt{2\text{Var}(\pi(\alpha N))} \leq \frac{\beta\sqrt{2\alpha}}{\sqrt{N}}\E\|\theta\|_\infty.
    \end{align*}
\end{proof}

\begin{lemma}
There exists a constant $K>0$ depending only on $\alpha$ and $\beta$ such that for each $N\geq 1$,
    $$\E|F_N(\beta) - \E_{\theta, I}F_N(\beta)| \leq \frac{K}{\sqrt{N}}(\E\|\theta\|^2_\infty)^{1/2}.$$
\end{lemma}

\begin{proof}
    On the event $\pi(\alpha N) \geq 1$, let $H_N^-$ be the Hamiltonian obtained by dropping the term $\theta_1(\sigma_{I(1,1)}, \ldots, \sigma_{I(1, p)})$ from $H_N$, and $\tilde H_N$ be obtained from $H_N$ by replacing $\theta_1(\sigma_{I(1,1)}, \ldots, \sigma_{I(1, p)})$ with an independent copy $\tilde \theta_1(\sigma_{\tilde I(1,1)}, \ldots, \sigma_{\tilde I(1, p)})$. We let $H_N = \tilde H_N$ when $\pi(\alpha N) = 0$. Denote the free energies corresponding to $H_N^-$ and $\tilde H_N$ by $F_N^-(\beta)$ and $\tilde F_N(\beta)$ respectively. It is easy to see that \begin{align*}
        |F_N(\beta) - F_N^-(\beta)| \leq \frac{\beta}{N}\|\theta_1\|_\infty \;\; \text{and}\;\;|\tilde F_N(\beta) - F_N^-(\beta)| \leq \frac{\beta}{N}\|\tilde \theta_1\|_\infty.
    \end{align*}
    Using the triangle inequality, we get
\begin{align*}
        (\tilde F_N(\beta) - F_N(\beta))^2 \leq \frac{\beta^2}{N^2}(\|\theta_1\|_\infty + \|\tilde \theta_1\|_\infty)^2 \leq \frac{2\beta^2}{N^2}(\|\theta_1\|_\infty^2 + \|\tilde \theta_1\|_\infty^2).
    \end{align*}
    Consider the filtration $(\mathcal{F}_r)_{r\geq 1}$ defined as $\mathcal{F}_0= \{\emptyset, \Omega\}$ and for $r\geq 1$, $\mathcal{F}_r$ is the $\sigma$-algebra generated by $\{\theta_k, I(k, i): k\leq r, i\leq p\}$. For $r\geq 1$, denote by $\Delta_r$ the martingale difference $$\Delta_r= \E_{\pi(\alpha N),\, \mathcal{F}_r} F_N(\beta) - \E_{\pi(\alpha N),\, \mathcal{F}_{r-1}}F_N(\beta).$$
   Using Jensen's inequality and symmetry, we have from the previous display that $$\E_{\pi(\alpha N)} \Delta_k^2 \leq \frac{4\beta^2}{N^2}\E\|\theta\|_\infty^2.$$ Thus, 
   \begin{align*}
        \E\big(F_N(\beta) - \E_{\Theta, I}F_N(\beta)\big)^2 & = \E\, \E_{\pi(\alpha N)} \Big(\sum_{k=1}^{\pi(\alpha N)}\Delta_k\Big)^2 = \E \sum_{k=1}^{\pi(\alpha N)}\E_{\pi(\alpha N)} \Delta_k^2 \leq \frac{4\alpha\beta^2}{N}\E\|\theta\|_\infty^2,
    \end{align*}
    and the statement of the lemma follows from Jensen's inequality.
\end{proof}

\section{Ground state energy of the dilute model}\label{gse_dilute_model}
 In this section we present the proof of Theorem \ref{thm:GSE}. Recall the operator $\T_\infty$ defined in \eqref{op_infty}. The first part of Theorem \ref{thm:GSE}, i.e., showing that in the regime \eqref{small alpha}, $\T_\infty$ admits a unique fixed point $\lambda_\infty \in \Pr_1(\X)$ follows readily from Lemma \ref{gen_fixed_pt}. Thus, in the remainder of this section, we prove the second part Theorem \ref{thm:GSE}, i.e., we establish that \begin{align}\label{thm1.2.2}
     \lim_{N\to \infty} \E |\gse_N - \cP_\infty(\lambda_\infty)| = 0.
 \end{align}
 
 The idea of our approach is to approximate the ground state energy by the free energy at a positive temperature. To that extent, we consider a measure $\nu \in {\rm Pr}(\Sigma)$ that satisfies the following property: for any $\epsilon > 0$, there exists a constant $c(\epsilon) > 0$ such that $\nu((a-\epsilon, a + \epsilon)) > c(\epsilon)$ for any $a \in \Sigma$. The existence of $\nu$ can be argued as follows. If $\Sigma$ is a finite set, then the uniform measure on $\Sigma$ obviously satisfies this criterion. If $\Sigma$ is an infinite set, there exists a sequence $(a_n)_{n\geq 1}\subset \Sigma$ that is dense in $\Sigma$. Let $\nu$ be the probability measure defined as $\nu(a_n)=2^{-n}$ for $n\geq 1$. Now, for any $\epsilon>0,$ since $I_n:=(a_n-\epsilon,a_n+\epsilon)$ for $n\geq 1$ forms an open covering of $\Sigma$, it follows from the compactness of $\Sigma$ that there exists some $n_0\geq1$ such that $I_1,\ldots,I_{n_0}$ cover $\Sigma$. Take $c(\epsilon)=\min(\nu(a_1),\ldots,\nu(a_{n_0}))>0.$ Then for any $a\in \Sigma,$ there exists some $1\leq n\leq n_0$ so that $a\in I_n$ and $\nu((a-\epsilon,a+\epsilon))\geq \nu(a_n)\geq c(\epsilon).$

 If $\sigma^*$ is a ground state (maximizer) of $\gse_N$, then  \begin{align*}
    \frac{1}{\beta}F_N(\beta) = \frac{1}{N\beta}\log \int e^{\beta H_N(\sigma)} \nu^{\otimes N}(d\sigma) \leq \frac{1}{N}H_N(\sigma^*) = \gse_N.
\end{align*} On the other hand, for any $\epsilon>0,$
\begin{align*}
    \frac{1}{\beta}F_N(\beta)&\geq \frac{1}{N\beta} \log \int_{B_\infty(\sigma^*,\epsilon)} e^{\beta H_N(\sigma)} \nu^{\otimes N}(d\sigma) \\
    &=\frac{1}{N\beta} \log \frac{\int_{B_\infty(\sigma^*,\epsilon)} e^{\beta (H_N(\sigma)-H_N(\sigma^*))} \nu^{\otimes N}(d\sigma)}{\nu^{\otimes N}(B_\infty(\sigma^*,\epsilon))}+\gse_N+\frac{1}{N\beta}\log \nu^{\otimes N}(B_\infty(\sigma^*,\epsilon))\\
    &\geq -\frac{\epsilon\sqrt{p}}{N}\sum_{k=1}^{\pi(\alpha N)}\mbox{Lip}(\theta_k)+\gse_N+\frac{1}{\beta}\log c(\epsilon),
\end{align*}
where $B_\infty(\sigma^*,\epsilon):=\prod_{i=1}^N(\sigma_i^*-\epsilon,\sigma_i^*+\epsilon)$. It follows that 
\begin{align*}
    \Bigl| \frac{1}{\beta}F_N(\beta)-\gse_N\Bigr|\leq \frac{\epsilon\sqrt{p}}{N}\sum_{k=1}^{\pi(\alpha N)}\mbox{Lip}(\theta_k)+\frac{1}{\beta}|\log c(\epsilon)|
\end{align*}
and from \eqref{eq lip}, we see that 
\begin{align}\label{gse fe approx}
\limsup_{\beta\to\infty}\limsup_{N\to\infty}\E\Bigl| \frac{1}{\beta}F_N(\beta)-\gse_N\Bigr|=0.
\end{align}

In light of establishing \eqref{thm1.2.2}, we write \begin{align}
         \lim_{N\to \infty} \E |\gse_N - \cP_\infty(\lambda_\infty)| & \leq \limsup_{\beta \to \infty}\limsup_{N\to \infty}\E \Big|\gse_N - \frac{1}{\beta}F_N(\beta)\Big| \label{eq1.2.2.1}\\
         & \qquad + \limsup_{\beta \to \infty}\limsup_{N\to \infty}\frac{1}{\beta}\E |F_N(\beta)-\E F_N(\beta)|\label{eq1.2.2.2}\\
         & \qquad + \limsup_{\beta \to \infty}\limsup_{N\to \infty}\frac{1}{\beta}|\E F_{N, \beta}-\cP_{\nu,\beta}(\lambda_{\nu, \beta})\label{eq1.2.2.3}|\\
         & \qquad + \lim_{\beta \to \infty}\Big|\frac{1}{\beta}\cP_{\nu,\beta}(\lambda_{\nu, \beta}) - \cP_\infty(\lambda_\infty)\Big|.\label{eq1.2.2.4}
     \end{align}
     The terms (on the right side of) \eqref{eq1.2.2.1}, \eqref{eq1.2.2.2} and \eqref{eq1.2.2.3} vanish due to \eqref{gse fe approx}, Proposition \ref{concentration} and the second part of Theorem \ref{thm1} respectively. It remains to show that the final term \eqref{eq1.2.2.4} is zero as well. To that extent, we record the following result that will be used later.

\begin{lemma}\label{lem,infty}
    Assume $\alpha p (p-1)\leq 1$. We have that $W_1(\lambda_\infty, \lambda_{\nu, \beta}) \to 0$ as $\beta \to \infty$.
\end{lemma}
The proof of this lemma is similar to that of Theorem \ref{continuity}, so we briefly sketch it below.
\begin{proof}[Sketch of proof]
    We first show that for any $\nu \in \mathrm{Pr}(\Sigma)$ fixed, the set of probability measures \begin{align*}
        \mathcal{M}:=\big\{\T_{\nu,\beta}(\lambda) : 0<\beta < \infty, \lambda \in \Pr_1(\X)\big\} \cup \big\{\T_\infty(\lambda):\lambda \in \Pr_1(\X)\big\}
    \end{align*}
    is tight. Indeed, the proof of this statement follows along the same lines as Theorem \ref{tight}, in particular, given $\epsilon > 0$, we can find $C_0 = C_0(\epsilon)$ and $L_0 = L_0(\epsilon)$ such that $\lambda$-measure of the set $K(C_0,L_0)$ (see \eqref{defn_K} for the definition) exceeds $1-\epsilon$ for all $\lambda \in \mathcal{M}$.
    
    Observe that for every $\beta >0$, by Theorem \ref{uniqueness}, the operator $\T_{\nu,\beta}$ admits a unique fixed point $\lambda_{\nu,\beta} \in \Pr_1(\X)$. Owing to the tightness of $(\lambda_{\nu,\beta})_{\beta>0} \subset \mathcal{M}$ and the completeness and separability of $\Pr_1(\X)$, we may assume that there is a sequence $(\beta_n)_{n\geq 1}$ such that $\lim_{n\to\infty}\beta_n =\infty$ and $\lambda_{\beta_n} \to \lambda_* \in \Pr_1(\mathcal{X})$ in distribution.

    It remains to show that $\T_\infty(\lambda_*) = \lambda_*$, since the uniqueness of the fixed point of $\T_\infty$ would then guarantee that $\lambda_* = \lambda_\infty$. Note that $\lambda_{\nu,\beta_n} \stackrel{d}{\to} \lambda_*$ implies that $\T_\infty(\lambda_{\nu,\beta_n}) \stackrel{d}{\to} \T_\infty(\lambda_*)$. Following an argument similar to that presented in Lemma \ref{lem1}, it can be shown that \begin{align*}
        \|T_{\nu,\beta_n,\pi(\alpha p)}(f_1^n, \ldots, f^n_{(p-1)\pi(\alpha p)}) - T_\infty(f_1^n, \ldots, f^n_{(p-1)\pi(\alpha p)})\|_\infty \to 0
    \end{align*}
    in probability as $n\to \infty$, where $(f_i^n)_{i\geq 1}$ are i.i.d. samples from $\lambda_{\nu,\beta_n}$. These observations, together with Slutsky's theorem, imply that  $\T_{\nu,\beta_n}(\lambda_{\nu,\beta_n}) \stackrel{d}{\to} \T_\infty(\lambda_*)$. This weak convergence can be upgraded to convergence in the Wasserstein 1-distance by utilizing the uniform bounds \eqref{l_infty bound} and \eqref{lip_bound} that hold true for $\T_\infty$ as well.
\end{proof}

Equipped with Lemma \ref{lem,infty}, we handle \eqref{eq1.2.2.4} in a manner similar to that used previously to show that \eqref{eqn:last term} is small. For brevity and completeness, we outline the argument here for the difference of the first terms in $\cP_{\nu,\beta}(\lambda_{\nu,\beta})$ and $\cP_\infty(\lambda_\infty)$, the difference of the second terms can be handled similarly.
For $\f = (f_1, \ldots, f_p) \in \X^p$, we declare \begin{align*}
          \Phi_\f = \Phi_{\theta, \f}(\rho): = \theta(\rho_1, \ldots, \rho_p) + \sum_{i=1}^p f_i(\rho_i)
        \end{align*}
        and let \begin{align*}
             R_\beta(\Phi_\f) := \frac{1}{\beta} \log \int e^{\beta \Phi_\f(\rho)} \nu^{\otimes p}(d\rho)\;\; \text{and} \;\; R_\infty(\Phi_\f) = \sup_{\rho \in \Sigma^p} \Phi_\f.
        \end{align*}
Lemma \ref{lem,infty} implies that $W_1(\lambda_{\nu,\beta}^{\otimes p},\lambda_\infty^{\otimes p})\to 0$, so by Skhorokhod's representation theorem, we may switch to a new probability space where there exist random variables $Y_\beta$ and $Y$ with laws $\lambda_{\nu,\beta}^{\otimes p}$ and $\lambda_\infty^{\otimes p}$ respectively such that $Y_\beta \to Y$ almost surely and $\E \|Y_\beta\|_\infty \to \E \|Y\|_\infty$. For every fixed $\f$ and $\theta$, note that $|R_\beta(\Phi_\f)|$ is uniformly bounded for $\beta \leq \infty$ and the map $\rho \in \Sigma^p \mapsto \Phi_\f(\rho)$ is continuous and bounded. Together with \eqref{bdd theta}, this implies that \begin{align*}
    \E|R_\beta(\Phi_Y)-R_\infty(\Phi_Y)|\to 0.
\end{align*}
Again, $R_\beta(\Phi_\f)$ is 1-Lipschitz in $\f$, so \eqref{bdd theta} and Theorem \ref{thm A1} implies that  \begin{align*}
    \E|R_\beta(\Phi_{Y_\beta}) - R_\beta(\Phi_Y)|\to 0.
\end{align*}
Thus, from the previous two displays we have that \begin{align*}
    \E|R_\beta(\Phi_{Y_\beta}) - R_\infty(\Phi_Y)|\to 0
\end{align*}
as desired.

\section{Counting configurations satisfying a fraction of the constraints}\label{sec6}

Throughout this section, $\upnu$ shall stand for the uniform measure on $\{-1,1\}$. Recall that for $\beta \in \R$, $F(\beta) = \cP_{\upnu, \beta}(\lambda_{\upnu,\beta})$, where $\lambda_{\upnu,\beta}$ is the unique fixed point of the operator $\T_{\upnu,\beta}$ guaranteed by Theorem \ref{thm1}.

\begin{proof}[Proof of Theorem \ref{csp:thm1}]
Recall that the constraint $\theta$ is satisfied by $\sigma \in \{-1,1\}^p$ when $\theta(\sigma) = 0$, otherwise $\theta(\sigma) = -1$ and we say that $\sigma$ does not satisfy $\theta$. The Hamiltonian and the free energy are defined in the usual manner: for $\beta \in \R$, \begin{align*}
    H_N(\sigma) = \sum_{k\leq \pi(\alpha N)}\theta_k(\sigma_{I(k,1)}, \ldots, \sigma_{I(k,p)})\;\;\text{and}\;\; F_N(\beta) = \frac{1}{N}\log \int e^{\beta H_N(\sigma)}\upnu^{\otimes N}(d\sigma).
\end{align*}
Thus, we have \begin{align*}
    \mathcal{AN}_{N,\epsilon} = 2^N \upnu^{\otimes N} \Big(\Big| \frac{H_N(\sigma)}{\pi(\alpha N)} -(t-1)\Big| \leq \epsilon\Big).
\end{align*}
From Griffith's lemma, we obtain that $\lim_{N\to\infty}F_N'(\beta)=(\lim_{N\to \infty}F_N(\beta))' = F'(\beta)$ for every $\beta\in \R$ where $F'(\beta)$ is defined. Together with this, the convexity of $F_N$, and the concentration of the free energy (Proposition \ref{concentration}), for any $\epsilon>0$ we can find $\delta,K>0$ such that for any $N\geq 1,$ with a probability of at least $1-K/N$, the following inequalities hold
	\begin{align*}
		\frac{F_N(\beta+\delta)-F_N(\beta)}{\delta}-F'(\beta)\leq \frac{\epsilon}{2}\;\;\text{and} \;\;\frac{F_N(\beta-\delta)-F_N(\beta)}{-\delta}-F'(\beta)\leq \frac{\epsilon}{2}.
	\end{align*}
Consequently, on the event that the above two inequalities hold, we have
\begin{align*}
\frac{1}{N}\log \int_{H_N(\sigma)\geq N(F'(\beta)+\epsilon )}e^{\beta H_N(\sigma)}\upnu^{\otimes N}(d\sigma)
&\leq \frac{1}{N}\log \int e^{(\beta+\delta)H_N(\sigma)}\upnu^{\otimes N}(d\sigma)-\delta(F'(\beta)+\epsilon )\\
&=F_N(\beta+\delta)-\delta(F'(\beta)+\epsilon )\\
&=F_N(\beta)-\delta \epsilon +\delta\Bigl(\frac{F_N(\beta+\delta)-F_N(\beta)}{\delta}-F'(\beta)\Bigr)\\
&\leq F_N(\beta)-\frac{\delta \epsilon }{2}
\end{align*}
and
\begin{align*}
	\frac{1}{N}\log \int_{H_N(\sigma)\leq N(F'(\beta)-\epsilon )}e^{\beta H_N(\sigma)}\upnu^{\otimes N}(d\sigma)
	&\leq \frac{1}{N}\log \int e^{(\beta-\delta)H_N(\sigma)}\upnu^{\otimes N}(d\sigma)+\delta(F'(\beta)-\epsilon )\\
	&=F_N(\beta-\delta)+\delta(F'(\beta)-\epsilon )\\
	&=F_N(\beta)-\delta \epsilon -\delta\Bigl(\frac{F_N(\beta-\delta)-F_N(\beta)}{-\delta}-F'(\beta)\Bigr)\\
&\leq F_N(\beta)-\frac{\delta \epsilon }{2}.
\end{align*}
Hence, with a probability of at least $1-K/N$,
\begin{align*}
	G_{N,\beta,\upnu}(|H_N(\sigma)-NF'(\beta)|\leq N\epsilon )&\geq 1-2e^{-\delta\epsilon  N/2},
\end{align*}
where $G_{N,\beta,\upnu}$ is the Gibbs measure associated to $H_N$ at temperature $\beta$. From this, by taking log on both sides, we readily have
\begin{align*}
\lim_{\epsilon \downarrow 0}	\limsup_{N\to\infty}\Bigl|\frac{1}{N}\log \int_{|H_N(\sigma)-NF'(\beta)|\leq N\epsilon }e^{\beta H_N(\sigma)}\upnu^{\otimes N}(d\sigma)-F(\beta)\Bigr|=0
\end{align*}
and therefore, for any $\beta\in \R,$
\begin{align*}
\lim_{\epsilon \downarrow 0}\limsup_{N\to\infty}\Bigl|	\frac{1}{N}\log \upnu^{\otimes N}\bigl(|H_N(\sigma)-NF'(\beta)|\leq N\epsilon \bigr)-(F(\beta)-\beta F'(\beta))\Bigr|=0.
\end{align*}
In view of the definition of $\mathcal{AN}_{N,\epsilon}$ and noting that  $\pi(\alpha N)/N$ converges to $\alpha$ almost surely, the above limit is equivalent to
\begin{align*}
\lim_{\epsilon \downarrow 0}\limsup_{N\to\infty}\Bigl|	\frac{1}{N}\log \mathcal{AN}_{N,\epsilon }(\alpha^{-1} F'(\beta))-\bigl(\log 2+F(\beta)-\beta F'(\beta)\bigr)\Bigr|=0.
\end{align*}
Finally, since $F$ is convex, for points $\beta_1 \leq \beta_2$ where $F$ is differentiable, we have $F'(\beta_1) \leq F'(\beta_2)$. Further when $0 < \beta_1 \leq \beta_2$, we have \begin{align*}
    F(\beta_2) - F(\beta_1) \leq F'(\beta_2) (\beta_2 - \beta_1) \leq \beta_2 F'(\beta_2) - \beta_1 F'(\beta_1).
\end{align*}
Rearranging this gives the second assertion in Theorem  \ref{csp:thm1}. The final assertion can be established similarly.
\end{proof}

To establish the proof of Corollary \ref{csp:cor1}, we need the following lemma.

\begin{lemma}\label{cor1}
Suppose that $\alpha p(p-1)\leq 1$ and  $\theta$ is symmetric. For any $\beta \in \R$ we have that 
	\begin{align*}
	F(\beta)=\alpha \E\log \bigl(e^{-\beta}+(1-e^{-\beta})\upnu^{\otimes p}(\theta(\rho)=0)\bigr)\;\;\text{and}\;\;F'(\beta)=-\alpha \E\frac{ 1-\upnu^{\otimes p}(\theta(\rho)=0)}{1+(e^{\beta}-1)\upnu^{\otimes p}(\theta(\rho)=0)}.
\end{align*}
\end{lemma}

\begin{proof} Since $\theta$ satisfies $\theta(\rho)=\theta(-\rho),$ it is easy to see that
		$$
	\theta_k(\rho_{(k-1)(p-1)+1}, \ldots, \rho_{k(p-1)},1)=\theta_k(-\rho_{(k-1)(p-1)+1}, \ldots, -\rho_{k(p-1)},-1),
	$$
	which implies that
	\begin{align*}
		\int  \cE_{\beta,r}(\rho, 1) \upnu^{\otimes (p-1)r}(d\rho)&=\int  \cE_{\beta,r}(\rho, -1) \upnu^{\otimes (p-1)r}(d\rho),
	\end{align*}
 where $ \cE_{\beta,r}$ is as in \eqref{def:E}. Consequently, for $f_1,\ldots,f_{(p-1)r}\equiv 0,$ 
	\begin{align*}
		T_{\upnu,\beta,r}(f_1,\ldots,f_{(p-1)r})(t)&=\frac{1}{\beta}\log\frac{\int   \cE_{\beta,r}(\rho, t) \upnu^{\otimes (p-1)r}(d\rho)}{\int   \cE_{\beta,r}(\rho, s) \upnu^{\otimes (p-1)r}(d\rho)\upnu(ds)}=0.
	\end{align*}
	From this and the unique fixed point $\lambda_{\upnu,\beta}$ of the distributional operator $\T_{\upnu,\beta}$ in Theorem \ref{thm1}, we obtain that $\lambda_{\upnu,\beta}$ is a Dirac measure at the zero function. Plugging this into our limiting free energy, it follows that
	\begin{align*}
	F(\beta) & = \E \log \int \prod_{k\leq \pi(\alpha p)}\Bigl(\int \exp \Big(\beta\theta_k(\rho_{k,1}, \ldots, \rho_{k,p-1}, \varepsilon )\Big) \prod_{i=1}^{p-1}\upnu(d\rho_{k,i})\Bigr)\upnu(d\varepsilon )\\
		&\qquad - \alpha(p-1)\E \log \int e^{\beta\theta(\rho_1, \ldots, \rho_p)}\prod_{i=1}^p\upnu(d\rho_i).
	\end{align*}
	Here, since by the symmetry of $\theta$ again,
	\begin{align*}
		\int e^{\beta\theta_k(\rho_{k,1}, \ldots, \rho_{k,p-1}, 1)} \prod_{i=1}^{p-1}\upnu(d\rho_{k,i})&=\int e^{\beta\theta_k(\rho_{k,1}, \ldots, \rho_{k,p-1},-1)} \prod_{i=1}^{p-1}\upnu(d\rho_{k,i}),
	\end{align*}
	the first term can be written as
	\begin{align*}
		\E\sum_{k=1}^{\pi(\alpha p)}\log \int e^{\beta\theta_k(\rho_{k,1}, \ldots,\rho_{k,p})} \prod_{i=1}^{p}\upnu(d\rho_{k,i})&=\alpha p \E\log \int e^{\beta\theta(\rho_1,\ldots,\rho_p)}\prod_{i=1}^p\upnu(d\rho_i).
	\end{align*}
    Combining with the second term, it leads to
	\begin{align*}
		F(\beta)&=\alpha  \E\log \int e^{\beta\theta(\rho_1,\ldots,\rho_p)}\prod_{i=1}^p\upnu(d\rho_i) =\alpha \E\log \bigl(e^{-\beta}\upnu^{\otimes p}(\theta(\rho)=-1)+\upnu^{\otimes p}(\theta(\rho)=0)\bigr)\\
		&=\alpha \E\log \bigl(e^{-\beta} + (1-e^{-\beta})\upnu^{\otimes p}(\theta(\rho)=0)\bigr).
	\end{align*}
 The expression for $F'$ follows from differentiating the above expression.
\end{proof}

\begin{proof}[Proof of Corollary \ref{csp:cor1}]
The first part of Corollary \ref{csp:cor1} follows directly from Lemma \ref{cor1}. As for the second part, owing to the form of $\theta$ given by \eqref{sym_theta} and the fact that $\upnu$ is a uniform measure on $\{-1,1\}^p$, we have that $$\varrho=\upnu^{\otimes p}\bigl(\phi(x_1,\ldots,x_p)=0\bigr) = \upnu^{\otimes p}(\theta(x_1, \ldots, x_p)=0)$$ for almost every choice of $\theta$. Thus from the above lemma, we have that
	\begin{align*}
	F(\beta)=\alpha \log \bigl(e^{-\beta}+(1-e^{-\beta})\varrho\bigr)\;\;\text{and}\;\;F'(\beta)=-\frac{\alpha e^{-\beta} (1-\varrho)}{e^{-\beta}+(1-e^{-\beta})\varrho}.
\end{align*}
Now, if we denote $t=1+\alpha^{-1}F'(\beta)$, then 
\begin{align*}
    t=\frac{\varrho }{e^{-\beta}+(1-e^{-\beta})\varrho} \;\;\text{or equivalently} \;\; e^{-\beta}&=\frac{(1-t)\varrho}{(1-\varrho)t}.
\end{align*}
This enables us to express
\begin{align*}
	F(\beta)-\beta F'(\beta)&=\alpha(1-t)\log \frac{1-\varrho}{1-t}+\alpha t\log \frac{\varrho}{t},
\end{align*}
so for any $0<t<1,$ we have
\begin{align*}
	\lim_{\epsilon \downarrow 0}	\limsup_{N\to \infty}\Bigl|\frac{1}{N}\log\mathcal{AN}_{N,\epsilon }(t)-\Bigl(\log 2+\alpha(1-t)\log \frac{1-\varrho}{1-t}+\alpha t\log \frac{\varrho}{t}\Bigr)\Bigr|=0.
\end{align*}

\end{proof}

\section{Free energy of the dilute continuous hardcore model}\label{proof_hc}
This section is dedicated to the proof of Theorem \ref{thm_hardcore}. Note that Lemma \ref{gen_fixed_pt} ensures that $\T_\infty^\h$ admits a unique fixed point, so we only need to prove the second part of the assertion, namely, the limit of the logarithmic volume of the hardcore model. Our approach introduces a free energy associated the hardcore model, namely, for any $\beta>0,$ \begin{align*}
F_N(\beta)& = - \log \upsilon_0+\frac{1}{N}  \log \int \eta^{\sum_{i=1}^N \sigma_i} e^{\beta\sum_{k=1}^{\pi(\alpha N)}\theta_k} d\sigma,
\end{align*}
where $\theta(\sigma_1,\sigma_2)=-(\sigma_1+\sigma_2-1)\1_{\sigma_1+\sigma_2\geq 1}$ and $\theta_k=\theta(\sigma_{I(k,1)},\sigma_{I(k,2)})$.
It is easy to see that the free energy converges to the logarithmic volume of the hardcore model as $\beta\to\infty$,
\begin{align*}
F_N(\infty)&: = \frac{1}{N} \log \nu^{\otimes N}(\mathfrak{P}_G) = - \log \upsilon_0+\frac{1}{N} \log \int \eta^{\sum_{i=1}^N \sigma_i}\prod_{k =1}^{\pi(\alpha N)} \1_{\theta_k\le 0} d\sigma .
\end{align*}
In the meantime, we also know that the limiting free energy $F_N(\beta)$ for any $\beta>0$ exists and is described by Theorem \ref{thm1}. Our major task will be to establish the interchangeability of the limits between $N$ and $\beta$ for $F_N(\beta).$ We proceed in three steps as follows.

\subsection{Approximation of the free energy}
Our first step establishes that  $F_N(\beta)$ converges to $F_N(\infty)$ in $L^1$ distance uniformly over all $N$ as $\beta$ tends to infinity.
\begin{lemma}\label{hc_1} We have that
    \begin{align*}
        \lim_{\beta \to \infty} \limsup_{N}\E|F_N(\beta) - F_N(\infty)| = 0.
    \end{align*}
\end{lemma}

\begin{proof} To lighten our expressions, we consider the case $\eta\leq 1$ only; the proof for $\eta >1$ is similar. 
A direct computation yields \begin{align*}
    \partial_\beta F_{N}(\beta)& = -\frac{1}{N}\frac{\sum_{j\leq \pi(\alpha N)}\int \eta^{\sum_{i\leq N} \sigma_i}\prod_{k\neq j}(\1_{\omega_k \leq 0} + e^{-\beta\omega_k}\1_{\omega_k > 0}) d\sigma}{\int \eta^{\sum_{i\leq N} \sigma_i}\prod_{k\leq \pi(\alpha N)}(\1_{\omega_k \leq 0} + e^{-\beta\omega_k}\1_{\omega_k > 0}) d\sigma}
\end{align*}
for $\omega_k:=\sigma_{I(k,1)}+\sigma_{I(k,2)}-1.$
Since every term in the sum is nonnegative, $ \E|\partial_\beta F_{N}(\beta)|=  -\E\partial_\beta F_{N}(\beta)$ and using symmetry gives
\begin{align}
    \E|\partial_\beta F_{N}(\beta)|
    &=\E \frac{\pi(\alpha N)}{N} \frac{\int f_{N, \beta}(\sigma) (\sigma_1 + \sigma_2 - 1) e^{-\beta (\sigma_1 + \sigma_2 - 1)} \1_{\sigma_1 + \sigma_2 > 1}\eta^{\sum_{i\leq N}\sigma_i}d\sigma}{\int f_{N, \beta}(\sigma) (\1_{\sigma_1 + \sigma_2 \leq 1} + e^{-\beta (\sigma_1 + \sigma_2 - 1)} \1_{\sigma_1 + \sigma_2 > 1})\eta^{\sum_{i\leq N}\sigma_i}d\sigma}, \label{eq:a0f}
\end{align}
where $ f_{N,\beta}(\sigma) = \prod_{k\leq \pi(\alpha N)-1} (\1_{\omega_k\leq 0} + e^{-\beta \omega_k} \1_{\omega_k > 0})$. For the rest of the proof, we will show that  there exists a positive constant $C$ depending only on $\alpha$ and $\eta$ such that 
\begin{align}\label{add:eq99}
\sup_{N}\E |\partial_\beta F_{N}(\beta)|\leq C(\beta^{-2}+e^{-\beta/2})
\end{align}
for any $\beta\geq 2.$ If this holds, then we readily have that as $\beta\to\infty,$
\begin{align*}
 \sup_N \E|F_N(\beta) - F_N(\infty)| 
  &\le \int_{\beta}^\infty \sup_N\E  |\partial_{\beta'} F_N(\beta')| d\beta'\to 0,
\end{align*}
and this completes our proof.

Now we turn to the proof of \eqref{add:eq99}.  Consider the integral in the numerator of \eqref{eq:a0f} with all spins but $\sigma_1$ and $\sigma_2$ fixed. This can be written as
 \begin{align}
& \int_0^1 \int_0^1 f_{N, \beta}(\sigma)  (\sigma_1 + \sigma_2 - 1) e^{-\beta (\sigma_1 + \sigma_2 - 1)}\1_{\sigma_1 + \sigma_2 > 1}  \eta^{\sum_{i\leq N}\sigma_i} d \sigma_1 d\sigma_2 \nonumber\\
&= \int_0^1  \int_{1- \sigma_2}^1  f_{N, \beta}(\sigma) (\sigma_1 + \sigma_2 - 1) e^{-\beta (\sigma_1 + \sigma_2 - 1)} \1_{\sigma_1 + \sigma_2 >1} \eta^{\sum_{i\leq N}\sigma_i}d \sigma_1 d\sigma_2 \nonumber \\
&\le \int_0^1  \int_{0}^{\sigma_2} f_{N, \beta}(\sigma)  \sigma_1 e^{-\beta \sigma_1} \eta^{\sum_{i\neq 2}\sigma_i} d \sigma_1 d\sigma_2,  \label{eq:a1f} 
\end{align}
where in the inequality above, we performed a change of variables from $\sigma_1+\sigma_2-1$ to $\sigma_1$ in the inner integral and used that $f_{N, \beta}(\sigma)$ is nonincreasing in $\sigma_1$ and $\eta\leq 1.$ Decomposing the inner integral of \eqref{eq:a1f} over disjoint intervals $[0, \min(1/2, \sigma_2)]$ and $[\min(1/2, \sigma_2), \sigma_2]$, we have \begin{align*}
    \eqref{eq:a1f} & = \int_0^1 \int_0^{\min(1/2,\sigma_2)} f_{N, \beta}(\sigma)  \sigma_1 e^{-\beta \sigma_1} \eta^{\sum_{i\neq 2}\sigma_i} d \sigma_1 d\sigma_2  +  \int_{1/2}^1\int_{1/2}^{\sigma_2}f_{N, \beta}(\sigma)  \sigma_1 e^{-\beta \sigma_1} \eta^{\sum_{i\neq 2}\sigma_i} d \sigma_1 d\sigma_2.
\end{align*}
In the second integral, bounding $\sigma_1 e^{-\beta \sigma_1} \le (1/2) e^{-\beta/2}$ (note that for $\beta \geq 2$, $xe^{-\beta x}$ achieves a global maximum at $x =1/\beta \leq 1/2$), and then replacing the variable $\sigma_1$ with $\sigma_1 -1/2$, we deduce that
\begin{align*}
    \eqref{eq:a1f} &\le  \int_0^1  \int_{0}^{1/2} f_{N, \beta}(\sigma)  \sigma_1 e^{-\beta \sigma_1}  \eta^{\sum_{i\neq 2}\sigma_i} d \sigma_1 d\sigma_2  + \frac{1}{2}e^{-\beta/2} \int_{1/2}^1  \int_{0}^{\sigma_2 - 1/2} f_{N, \beta}(\sigma)  \eta^{\sum_{i\neq 2}\sigma_i}  d \sigma_1 d\sigma_2 \\
    &\le  \int_0^1  \int_{0}^{1/2} f_{N, \beta}(\sigma)  \sigma_1 e^{-\beta \sigma_1} \eta^{\sum_{i\neq 2}\sigma_i} d \sigma_1 d\sigma_2+ \frac{1}{2}e^{-\beta/2} \int_0^{1/2}  \int_{0}^{1/2} f_{N, \beta}(\sigma) \eta^{\sum_{i\neq 2}\sigma_i} d \sigma_1 d\sigma_2.
\end{align*}
Integrating over the rest of the coordinates, we obtain
\begin{align} \label{eq:a2f}
& \int f_{N, \beta}(\sigma) (\sigma_1 + \sigma_2 - 1) e^{-\beta (\sigma_1 + \sigma_2 - 1)}\1_{\sigma_1 + \sigma_2 > 1}  \eta^{\sum_{i\leq N}\sigma_i }d \sigma \nonumber\\
&\le   \int_{ [0, 1/2] \times [0, 1]^{N-1} } f_{N, \beta}(\sigma)   \sigma_1 e^{-\beta \sigma_1}  \eta^{\sum_{i\neq 2}\sigma_i}d \sigma + \frac{1}{2} e^{-\beta/2} \int_{ [0, 1/2]^2 \times [0, 1]^{N-2} }f_{N, \beta}(\sigma) \eta^{\sum_{i\neq 2}\sigma_i} d \sigma \nonumber \\
&=:  R_1 + R_2.
\end{align}
On the other hand, by dropping the exponential term below and restricting the domain of integration, we obtain that the integral appearing in the denominator of \eqref{eq:a0f} is bounded as 
\begin{align*}
 & \int f_{N, \beta}(\sigma) \big ( \1_{\sigma_1 + \sigma_2 \leq 1} +e^{-\beta(\sigma_1 + \sigma_2 - 1)} \1_{\sigma_1 + \sigma_2 > 1} \big)  \eta^{\sum_{i\leq N}\sigma_i}d\sigma  \ge \eta^{1/2}\int_{ [0, 1/2]^2 \times [0, 1]^{N-2} } f_{N, \beta}(\sigma) \eta^{\sum_{i\neq 2}\sigma_i}d \sigma.
\end{align*}
It follows that 
\begin{align} \label{eq:a3f}
    R_2\cdot\Big( \int f_{N, \beta}(\sigma) \big ( \1_{\sigma_1 + \sigma_2 \leq 1} +e^{-\beta(\sigma_1 + \sigma_2 - 1)} \1_{\sigma_1 + \sigma_2 > 1} \big)  \eta^{\sum_{i\leq N}\sigma_i} d\sigma \Big)^{-1} \le \frac{e^{-\beta/2}}{2\eta^{1/2}}.
\end{align}

Next, to handle $R_1$ in a similar manner, observe that with high probability, $f_{N,\beta}$ does not contain the $(1,2)$ edge and we can rewrite $f_{N,\beta}$ as
 \begin{align*}
f_{N, \beta}(\sigma)  =  g_{N, \beta}(\sigma) \prod^*_{k} \big ( \1_{ \theta_k \le 0} +e^{-\beta \theta_k} \1_{ \theta_k > 0} \big),
 \end{align*}
 where the product $\prod^*_{k}$ is taken over all edges $k \le \pi(\alpha N)-1$ such that  $1 \in \{  I(k, 1), I(k, 2) \}$. Let $J'$ denote the collection of vertices that appear as a neighbor of $1$ among these edges and let $J  = J' \cup \{2\}$. Thus, the set $J$ represents all neighbors of the vertex $1$.
To handle the integral $R_1$, observe that $f_{N,\beta}$ is nonincreasing in each coordinate. Together with $\eta\leq 1$, we have
\begin{align*}
    \int_{[0, 1]^{J} } f_{N, \beta}(\sigma)   \eta^{\sum_{i\neq 2}\sigma_i} d\sigma_J
    &\le 2^{ | J| }   \int_{[0, 1/2]^{J}} f_{N, \beta}(\sigma)  \eta^{\sum_{i\neq 2}\sigma_i}  d\sigma_J \\
    &= 2^{ | J| }   \int_{[0, 1/2]^{J}} g_{N, \beta}(\sigma)   \eta^{\sum_{i\neq 2}\sigma_i}\prod^*_{k} \big ( \1_{ \omega_k \le 0} +e^{-\beta \omega_k} \1_{ \omega_k > 0} \big) d\sigma_J,
\end{align*}
where we adopted the notation $\sigma_J: = (\sigma_j)_{j\in J}$ for $J \subseteq [N]$. Note that $\prod^*_{k} \big ( \1_{ \omega_k \le 0} +e^{-\beta \omega_k} \1_{ \omega_k > 0} \big) \equiv 1$ if $\sigma_1 \le 1/2$ and $\sigma_j \le 1/2$ for $j \in J' \subseteq J$. Thus, integrating over the remaining variables, we obtain
\begin{align*}
    R_1 &\le 2^{ | J| }   \int_{  [0, 1/2]^{J \cup 
    \{1\}}   \times [0, 1]^{[N] \setminus (J \cup \{1\})} } \sigma_1 e^{-\beta \sigma_1} g_{N, \beta}(\sigma) \eta^{\sum_{i\neq 2}\sigma_i}  d\sigma.
\end{align*}
Note that $g_{N, \beta}(\sigma) $ does not depend on $\sigma_1$, $\int_0^{1/2} \sigma_1 e^{-\beta \sigma_1} d\sigma_1 \le \beta^{-2},$ and $\eta^{\sigma_1} \leq 1$. It follows that 
\begin{align} \label{eq:a4f}
    R_1 \le  2^{ | J| } \beta^{-2} \int_{  [0, 1/2]^{J}   \times [0, 1]^{[N] \setminus (J \cup \{1\})} } g_{N, \beta}(\sigma) \eta^{\sum_{i>2}\sigma_i}  d\sigma_{[N]\setminus\{1\}}.
\end{align}
On the other hand, for the integral in the denominator of \eqref{eq:a0f}, by dropping the exponential term below and restricting the domain of integration, we obtain that
    \begin{align} \label{eq:a5f}
    & \int  f_{N, \beta}(\sigma)  \big ( \1_{\sigma_1 + \sigma_2 \leq 1} +e^{-\beta (\sigma_1 + \sigma_2 - 1)} \1_{\sigma_1 + \sigma_2 > 1} \big)  \eta^{\sum_{i\leq N}\sigma_i} d\sigma \nonumber \\
        &\geq \eta\int_{  [0, 1/2]^{J}   \times [0, 1]^{[N] \setminus (J \cup \{1\})} } g_{N, \beta}(\sigma) \eta^{\sum_{i> 2}\sigma_i}  d\sigma_{[N]\setminus \{1\}}.
    \end{align}
Combining \eqref{eq:a4f} and \eqref{eq:a5f}, we have 
\begin{align}\label{eq:a6f}
    R_1 \cdot \Big(\int f_{N, \beta}(\sigma) \big ( \1_{\sigma_1 + \sigma_2 \leq 1} +e^{-\beta (\sigma_1 + \sigma_2 - 1)} \1_{\sigma_1 + \sigma_2 > 1} \big) \eta^{\sum_{i\leq N}\sigma_i}  d\sigma \Big)^{-1} \le \frac{2^{|J|}}{\eta\beta^2}.
\end{align}
Therefore, combining \eqref{eq:a0f}, \eqref{eq:a2f}, \eqref{eq:a3f} and \eqref{eq:a6f} we get
\begin{align*}
\E |\partial_\beta F_N(\beta)| &\le \frac{1}{\eta \beta^2} \E \Big[ \frac{\pi(\alpha N)}{N} 2^{|J|} \Big] +   \frac{\alpha e^{-\beta/2}}{2\eta^{1/2}} \le \frac{1}{\eta \beta^2} \left ( \E \Big[ \frac{\pi(\alpha N)^2}{N^2}\Big]  \E[ 2^{2|J|} ] \right)^{1/2} +    \frac{\alpha e^{-\beta/2}}{2\eta^{1/2}}.
\end{align*}
This readily implies \eqref{add:eq99} since $|J|$ is Poisson distributed with mean $ \pi(2 \alpha)$.

\end{proof}

\subsection{Convergence of the free energy}
 From Theorem \ref{thm1}, we know that $\lim_{N\to\infty}F_{N}(\beta)$ is given by $\mathcal{P}_{\nu,\beta}(\lambda_{\nu,\beta})$, where $\nu(dt)\propto \eta^{t}dt$ on $[0,1].$ This limit can be reformulated as follows. Let $\X_0$ be the set of all continuous functions $f\geq 0$ that satisfy $\int_0^1 f(x) dx = 1$. For $\lambda\in \Pr_1(\X_0)$, let\begin{align*}
    \cP_\beta^\h(\lambda) & = -\log \upsilon_0 + \E\log \int \eta^t\prod_{k\leq \pi(2\alpha)}\bigl(F_k(1-t) +G_{\beta, k}(1-t)\bigr) dt \\
    & \qquad - \alpha \E \log \int \bigl(F_1(1-t) + G_{\beta, 1}(1-t)\bigr)f_2(t) dt,
\end{align*}
where $(f_k)_{k\geq 1}$ are i.i.d. copies with law $\lambda$, $F_k$ is the cumulative distribution function of $f_k$, and \begin{align}\label{defn G}
    G_{\beta,k}(1-t) &= \int_{1-t}^1  e^{-\beta( \rho  - ( 1-t) ) }   f_k (\rho)  d \rho = \int_0^{t}  e^{-\beta u }   f_k (u  + (1-t))  d u.
\end{align}
Define the operator $\T^\h_\beta: \Pr_1(\X_0) \to \Pr_1(\X_0)$ as follows: for $\lambda \in \Pr_1(\X_0)$, $\T^\h_\beta(\lambda)$ is the law of the function \begin{align}
     T^\h_\beta(f_1, \ldots, f_{\pi(2\alpha)})(t)  & = \frac{\int  \eta^t\prod_{k\leq \pi(2\alpha)} \big ( \1_{\rho_k + t\le 1 } + e^{-\beta( \rho_k + t-1) } \1_{ \rho_k + t >1 } \big)  f_k(\rho_k)  d \rho  }{
     \int_0^1  \eta^s\int \prod_{k\leq \pi(2\alpha)} \big ( \1_{\rho_k + s \le 1} + e^{-\beta( \rho_k + s-1) } \1_{ \rho_k + s>1 } \big)  f_k (\rho_k)  d \rho    ds
    } \notag\\
    & =  \frac{ \eta^t\prod_{k\leq \pi(2\alpha)} \big(  F_k (1- t)  + G_{\beta, k}(1-t) \big) }{ \int_0^1 \eta^s\prod_{k\leq \pi(2\alpha)} \big(  F_k (1-s)  + G_{\beta, k}(1-s) \big) ds }. \label{fp_hc}
\end{align}
From Lemma \ref{gen_fixed_pt}, $\T_\beta^\h$ admits a unique fixed point, called $\lambda_\beta^\h$. From Theorem \ref{thm1}, we have 

\begin{lemma}\label{hc_2}
    Let $\alpha \leq 1/2$. For every $\beta <\infty$, $
    \lim_{N\to \infty}\E | F_N(\beta) - \cP_\beta^\h(\lambda_\beta^\h)| = 0.
$
\end{lemma}
\subsection{Approximation of the energy functional}
In the third step, we approximate the zero-temperature functional $\mathcal{P}_\infty^\h$ defined in \eqref{add:eq01} by the positive-temperature one $\mathcal{P}_\beta^\h$ defined in the previous section. 

\begin{lemma}\label{hc_3}
We have that 
$\lim_{\beta\to\infty}\cP_{\beta}^\h(\lambda_{\beta}^\h) =\cP_\infty^\h(\lambda_\infty^\h),$
    where $\lambda_{\beta}^\h$ and $\lambda_\infty^\h$ are the unique fixed points of $\T_{\beta}^\h$ and $\T_\infty^\h$ respectively.
\end{lemma}

From Lemmas \ref{hc_1}-\ref{hc_3}, the proof of Theorem \ref{thm_hardcore} is completed by the triangle inequality. Thus, for the rest of this section, we prove Lemma \ref{hc_3}. First, we need a result concerning the convergence of the fixed point measures.

\begin{lemma}
    As $\beta \to \infty$,  $\lambda_{\beta}^\h \stackrel{d}{\to}\lambda_\infty^\h$.
\end{lemma}

\begin{proof}
    Let $f, (f_k)_{k\geq 1}$ and $f_\beta, (f_{\beta, k})_{k\geq 1}$ be the samples drawn from $\lambda_\infty^\h$ and $\lambda_\beta^\h$ respectively, all independent of each other. If $F_{\beta, k}$ is the cumulative distribution function of $f_{\beta, k}$ and \begin{align}\label{defn_G}
        G_{\beta, k}(1-t) = \int_{1-t}^1 e^{-\beta(\rho - (1-t))} f_{\beta,k}(\rho) d\rho = \int_0^t e^{-\beta u} f_{\beta, k}(u + (1-t))du,
    \end{align}
    then for $0 \le t_1 \le t_2 \le 1$, we have $F_{\beta,k}(1-t_1) \geq F_{\beta,k}(1-t_2)$ and $G_{\beta,k}(1-t_1) \geq G_{\beta,k}(1-t_2)$, so from \eqref{fp_hc}, the map $ t \mapsto \eta^{-t} f_{\beta,k}(t)$ is nonincreasing on $[0, 1].$ We therefore have  \begin{align*}
        F_{\beta,k}(1/2) =  \int_0^{1/2} f_{\beta,k} (s) ds  \ge \int_0^{1/2} f_{\beta,k}(s+1/2) \eta^{-1/2} ds  = \eta^{-1/2} (1 - F_{\beta,k}(1/2)),
    \end{align*}
     yielding that 
     \[ F_{\beta,k}(1/2) \ge (1+ \eta^{1/2})^{-1} =: c_0.\]
     Denote, for $m \ge 0$,
\[ C_{\beta, m} =  \int_0^1 \eta ^t \prod_{k\leq m} \big(  F_{\beta, k} (1-t)  + G_{\beta, k}(1-t) \big)  dt. \]
Let $c_1 := \min( \eta^{1/2} , 1)/2$. Then we have, by restricting the integral from $[1/2,1]$, that
\begin{align}\label{eq:deno_hardcore_lb1}
   C_{\beta, \pi(2 \alpha) }  &\ge  \int_0^1 \eta ^t \prod_{k\leq \pi(2 \alpha)} F_{\beta, k} (1-t) dt \ge 
   c_1 \prod_{k\leq \pi(2\alpha)}  F_{\beta, k} (1/2) \ge c_1 c_0^{\pi(2\alpha)}.
\end{align}
By a similar argument, we can obtain 
\begin{align}\label{eq:deno_hardcore_lb2}
   \int_0^1 \eta ^t \prod_{k\leq \pi(2 \alpha)} F_{ k} (1-t)  dt 
    \ge c_1 c_0^{\pi(2\alpha)},
\end{align}
where $F_k$ is the cumulative distribution function of $f_k$. Recalling $G_{\beta,k}$ from \eqref{defn_G}, we trivially have that $F_{\beta, k}(t) + G_{\beta, k}(t) \leq 1$ for any $t \in [0,1]$. Thus, denoting $c_2 := \max( \eta , 1)$, we obtain from \eqref{fp_hc}, \eqref{eq:deno_hardcore_lb1} and \eqref{eq:deno_hardcore_lb2} that
\begin{equation} \label{eq:f_beta_uniform_bdd}
  \| f_{\beta,k} \|_\infty \preccurlyeq   \frac{c_2}{c_1  c_0^{\pi(2\alpha)}}\;\;\text{and} \;\;\| f_k \|_\infty \preccurlyeq \frac{c_2}{c_1  c_0^{\pi(2\alpha)}},
\end{equation}
where $\preccurlyeq$ denotes (first order) stochastic domination. 

Next, using the nonincreasing property of $t \mapsto \eta^{-t}f_{\beta,k}(t)$, we see that for any $0\leq t_1,t_2\leq 1$,
$$
|f_{\beta,k}(t_1)-f_{\beta,k}(t_2)|\leq \|f_{\beta,k}\|_\infty\max\bigl(|\eta^{|t_1-t_2|}-1|,|\eta^{-|t_1-t_2|}-1|\bigr)\leq c_3\|f_{\beta,k}\|_\infty|t_1-t_2|
$$
for $c_3:=\max(\eta,1/\eta)|\log \eta|$. Let $
        \phi_\beta(t,s)= \1_{s \le 1-t } + e^{-\beta( s - (1-t))}  \1_{s > 1-t }.
   $
    For any $0\leq t_2\leq t_1\leq 1,$ note that
   $
\phi_\beta(t_1,s)=\phi_\beta(t_2,s+\delta)
$ for $\delta=t_1-t_2.$ 
From this and a change of variable,
\begin{align*}
    &\int_0^1\phi_\beta(t_1,s)f_{\beta,k}(s)ds-\int_0^1\phi_\beta(t_2,s)f_{\beta,k}(s)ds\\
    &= \int_0^1\phi_\beta(t_2,s+\delta)f_{\beta,k}(s+\delta-\delta)ds-\int_0^1\phi_\beta(t_2,s)f_{\beta,k}(s)ds\\
&=\int_\delta^{1+\delta}\phi_\beta(t_2,s)f_{\beta,k}(s-\delta)ds-\int_0^1\phi_\beta(t_2,s)f_{\beta,k}(s)ds\\
    &=\Bigl(\int_1^{1+\delta}-\int_0^\delta\Bigr)\phi_\beta(t_2,s)f_{\beta,k}(s-\delta)ds+\int_0^1\phi_\beta(t_2,s)\bigl(f_{\beta,k}(s-\delta)-f_{\beta,k}(s)\bigr)ds.
\end{align*}
It follows that
\begin{align}\label{eq:f_Lip}
    \Bigl|\int_0^1\phi_\beta(t_1,s)f_{\beta,k}(s)ds-\int_0^1\phi_\beta(t_2,s)f_{\beta,k}(s)ds\Bigr|&\leq (2+c_3)\|f_{\beta,k}\|_\infty|t_1-t_2|.
\end{align}
From \eqref{fp_hc}, \eqref{eq:f_beta_uniform_bdd} and \eqref{eq:f_Lip}, we then have that for any $0 \le t_1 , t_2 \le 1,$ 
\begin{align} 
    |f_\beta(t_1) - f_\beta(t_2)| &\preccurlyeq  \frac{ c_2} {C_{\beta, \pi(2 \alpha) }} \sum_{ k \le \pi(2\alpha)} \Big | \int_0^1  \phi_\beta(t_1 , \rho_k )f_{\beta,k}(\rho_k) d \rho_k  -  \int_0^1\phi_\beta(t_2 , \rho_k )  f_{\beta,k}(\rho_k) d \rho_k \Big| \nonumber\\
    &\le \frac{|t_1 - t_2|(2+ c_3) c_2}{ c_1 c_0^{\pi(2\alpha)}} \sum_{ k \le \pi(2\alpha)} \|   f_{\beta, k} \|_\infty  \le \frac{|t_1 - t_2|(2+ c_3) c_2}{ c_1 c_0^{\pi(2\alpha)}} \sum_{ k \le \pi(2\alpha)}  \frac{c_2}{c_1  c_0^{ \pi_k(2\alpha)}}, \label{eq:f_beta_uniform_lip}
\end{align}
    where $(\pi_k(2\alpha))_{k\geq 1}$ are independent Poisson random variables with mean $2\alpha$.
    
    Now, by \eqref{eq:f_beta_uniform_bdd}, \eqref{eq:f_beta_uniform_lip}, and the Arzela-Ascoli theorem, $\{ f_\beta: \beta > 0 \}$ is tight. Hence, there exists a subsequence $\beta_n \to \infty $ such that 
    $ f_{\beta_n} \stackrel{d}{\to} f_\infty$. We claim that $f_\infty$ satisfies \eqref{eq:fixpt_beta_infty_hardcore}. If this holds, then from the uniqueness of fixed point equation \eqref{eq:fixpt_beta_infty_hardcore}, it will follow that 
    \[ f_\infty \stackrel{d}{=} f\;\;\text{and}\;\; f_\beta \stackrel{d}{\to} f \text{ as } \beta \to \infty. \]  To establish our claim, by the Skorokhod representation theorem, we can assume without loss of generality that 
    $ \lim_{\beta\to\infty}\| f_\beta  -  f_\infty \|_\infty = 0$ a.s.
Let $f_{\infty, k} $ be i.i.d.\ copies of $f_{\infty},  $ independent of $\pi(2 \alpha)$, and assume that 
$\lim_{\beta\to\infty}\| f_{\beta, k}  -  f_{\infty, k} \|_\infty =0$ a.s.\ for each $k$. 
We need to argue that as $\beta \to \infty$, uniformly in $t\in [0,1]$,
\[ \frac{  \eta^t\prod_{k\leq \pi(2\alpha)} \big( \int_0^1 \phi_{\beta}(t, \rho_k) 
    f_{\beta, k}(\rho_k) d \rho_k \big)   }{\int_0^1 \eta^s\prod_{k\leq \pi(2\alpha)} \big( \int_0^1 \phi_{\beta}(s, \rho_k) 
    f_{\beta, k}(\rho_k) d \rho_k \big)    ds} \to 
    \frac{  \eta^t\prod_{k\leq \pi(2\alpha)} \big( \int_0^1 \1_{ \rho_k \le 1-t}
    f_{\infty, k}(\rho_k) d \rho_k \big)   }{\int_0^1 \eta^s\prod_{k\leq \pi(2\alpha)} \big( \int_0^1 \1_{ \rho_k \le 1-s }
    f_{\infty, k}(\rho_k) d \rho_k \big)    ds}. \]
 The desired convergence follows since after conditioning on $\pi(2\alpha)$, for each $k$,
\begin{align}\label{eq:intermediate_dct_hardcore}
&  \sup_{t \in [0, 1]}\Big|  \int_0^1 \phi_{\beta}(t, \rho_k) 
    f_{\beta, k}(\rho_k) d \rho_k   - \int_0^1 \1_{ \rho_k \le 1-t}
    f_{\infty, k}(\rho_k) d \rho_k  \Big|  \nonumber \\
    &\le \| f_{\beta, k} -  f_{\infty , k}\|_{\infty} + \| f_{\infty , k}\|_{\infty} \int_0^1 e^{-\beta u} du 
     \le \| f_{\beta, k} -  f_{\infty , k}\|_{\infty} + \beta^{-1} \| f_{\infty , k}\|_{\infty} \to 0.
\end{align}
\end{proof}

Equipped with the previous lemma, let us complete the proof of Lemma \ref{hc_3}.
\begin{proof}[Proof of Lemma \ref{hc_3}]
Let $(f_k)_{k\geq 1}$ and $(f_{\beta,k})_{k\geq 1}$ be independent samples from $\lambda_\infty^\h$ and $\lambda_\beta^\h$ respectively. We write \begin{align*}
    \cP_\beta^\h(\lambda_\beta^\h) - \cP_\infty^\h(\lambda_\infty^\h) & = \E \log \frac{ \int_0^1  \eta ^t\int \prod_{k\leq \pi(2\alpha)} (\1_{ \rho_k \leq 1-t } + e^{-\beta( \rho_k + t-1) } \1_{ \rho_k >1-t  } )  f_{\beta, k}(\rho_k)  d \rho  dt }{
    \int_0^1 \int \prod_{k\leq \pi(2\alpha)}  \1_{ \rho_k \leq 1-t }   f_k(\rho_k)  d \rho \eta^t  dt}\\
    & \qquad  + \E \log \frac{\int (\1_{\rho_1 + \rho_2 \leq 1} + e^{-\beta (\rho_1 + \rho_2 -1 )}\1_{\rho_1 + \rho_2 >1}) f_{\beta, 1}(\rho_1)f_{\beta,2}(\rho_2)d\rho}{\int\1_{\rho_1 + \rho_2 \leq 1} f_1(\rho_1)f_2(\rho_2)d\rho}.
\end{align*}
    Let us show the convergence of the first term since the other term can be handled similarly. We want to show 
    \begin{align} \label{eq:ratio_a1}
    \E \log \frac{ \int_0^1  \eta ^t\int \prod_{k\leq \pi(2\alpha)} ( \1_{ \rho_k \leq 1-t } + e^{-\beta( \rho_k + t-1) } \1_{ \rho_k>1-t  } )  f_{\beta, k}(\rho_k)  d \rho dt }{
    \int_0^1  \eta ^t\int \prod_{k\leq \pi(2\alpha)}  \1_{ \rho_k \leq 1-t }   f_k(\rho_k)  d \rho  dt} \to 0.
\end{align}
From the previous lemma, we may assume that without loss of generality, almost surely, for each $k$
\[ \| f_{k, \beta}  - f_k \|_\infty \to 0. \]
Write the expression in \eqref{eq:ratio_a1} as
\begin{align*}
   & \E \log \frac{ \int_0^1  \eta ^t\prod_{k\leq \pi(2\alpha)} \big(  F_{\beta, k} (1- t)  + G_{\beta, k}(1-t) \big) dt }{ \int_0^1 \eta ^t \prod_{k\leq \pi(2\alpha)}  F_{\beta, k} (1-t) dt } +  \E \log \frac{ \int_0^1 \eta ^t \prod_{k\leq \pi(2\alpha)}  F_{\beta, k} (1-t)   dt }{  \int_0^1 \eta ^t \prod_{k\leq \pi(2\alpha)}  F_{ k} (1-t)   dt} \\
   &=: \E \log T_1 + \E \log T_2.
   \end{align*}
Recall the definition of $G_{\beta,k}$ from \eqref{defn_G}. For $\beta \geq 1$, we have $$\| G_{\beta, k} \|_\infty \le \| f_{\beta, k}\|_\infty \int_0^1 e^{-\beta u} du = \| f_{\beta, k}\|_\infty.$$ 
Using this and the lower bound \eqref{eq:deno_hardcore_lb1}, we obtain
\begin{align*}
    1 \le T_1 &\le 1 + \sum_{k \le \pi(2 \alpha)} \frac{\| G_{\beta, k} \|_\infty}{\int_0^1 \eta ^t\prod_{k\leq \pi(2 \alpha)} F_{\beta, k} (1-t)  dt} \cdot 
    { \int_0^1  \eta ^t\prod_{i\leq \pi(2\alpha), \, i \ne k} \big ( F_{\beta, i}(1-t) + G_{\beta, i}(1-t) \big)  dt } \\
    &\le 1 + \sum_{k \le \pi(2 \alpha)}  \frac{c_2\| f_{\beta, k}\|_\infty}{\beta c_1  c_0^{\pi_k(2\alpha)}} \prod_{i\leq \pi(2\alpha), \, i \ne k}  \Big(1  +   \| f_{\beta, i} \|_\infty  \Big) \\
     &\le 1 + \frac{1}{\beta} \sum_{k \le \pi(2 \alpha)}  \Big(\frac{c_2}{c_1 c_0^{\pi_k(2\alpha)}}\Big)^2 \prod_{i\leq \pi(2\alpha), \, i \ne k}  \Big(1  +   \frac{c_2}{ c_1  c_0^{\pi_i(2\alpha)}} \Big), 
\end{align*}
where in the last inequality, we used the bound \eqref{eq:f_beta_uniform_bdd}. It follows that $\E \log T_1 \to 0$ as $\beta \to \infty$. By a reasoning similar to \eqref{eq:intermediate_dct_hardcore}, for each $k$, 
\[ \| F_{\beta, k}  -  F_{k} \|_\infty \to 0 \ \ \text{a.s.} \]
which yields that $T_2 \ge 1$. On the other hand,
\begin{align*}
    \max\Big(T_2, \frac{1}{T_2}\Big) &\le c_2\max \Bigl[ \Big( \int_0^1\eta ^t\prod_{k\leq \pi(2\alpha)}  F_{\beta, k} (1-t)    dt \Big)^{-1},  \Big (\int_0^1\eta ^t\prod_{k\leq \pi(2\alpha)}  F_{ k} (1-t) dt \Big)^{-1} \Bigr] \\
    & \le \frac{c_2}{c_1  c_0^{\pi(2\alpha)}} 
\end{align*} 
by  \eqref{eq:deno_hardcore_lb1} and \eqref{eq:deno_hardcore_lb2}. Therefore, by the dominated convergence theorem, 
$\E \log T_2 \to 0.$
\end{proof}

{
\footnotesize
\bibliographystyle{siam}
\bibliography{diluted_ref}

\begin{thebibliography}{10}

\bibitem{achlioptas2003randomksatmomentssuffice}
{\sc D.~Achlioptas and C.~Moore}, {\em Random {$k$}-{SAT}: two moments suffice
  to cross a sharp threshold}, SIAM J. Comput., 36 (2006), pp.~740--762.

\bibitem{Achlioptas2005RigorousLO}
{\sc D.~Achlioptas, A.~Naor, and Y.~Peres}, {\em Rigorous location of phase
  transitions in hard optimization problems}, Nature, 435 (2005), pp.~759--764.

\bibitem{AIZENMAN1980281}
{\sc M.~Aizenman and B.~Simon}, {\em A comparison of plane rotor and {I}sing
  models}, Phys. Lett. A, 76 (1980), pp.~281--282.

\bibitem{Aizenman_2003}
{\sc M.~Aizenman, R.~Sims, and S.~L. Starr}, {\em Extended variational
  principle for the {S}herrington-{K}irkpatrick spin glass model}, Phys. Rev.
  B, 68 (2003), p.~214403.

\bibitem{aldous2005survey}
{\sc D.~J. Aldous and A.~Bandyopadhyay}, {\em A survey of max-type recursive
  distributional equations}, Ann. Appl. Probab., 15 (2005), pp.~1047--1110.

\bibitem{bandyopadhyay2006counting}
{\sc A.~Bandyopadhyay and D.~Gamarnik}, {\em Counting without sampling: new
  algorithms for enumeration problems using statistical physics}, in
  Proceedings of the Seventeenth Annual ACM-SIAM Symposium on Discrete
  Algorithm, SODA '06, USA, 2006, Society for Industrial and Applied
  Mathematics, p.~890.

\bibitem{bandyopadhyay2008counting}
\leavevmode\vrule height 2pt depth -1.6pt width 23pt, {\em Counting without
  sampling: asymptotics of the log-partition function for certain statistical
  physics models}, Random Struct. Alg., 33 (2008), pp.~452--479.

\bibitem{basak2012ferromagnetic}
{\sc A.~Basak and A.~Dembo}, {\em Ferromagnetic {I}sing measures on large
  locally tree-like graphs}, Ann. Probab., 45 (2017), pp.~780--823.

\bibitem{bates2023parisiformulabalancedpotts}
{\sc E.~Bates and Y.~Sohn}, {\em Parisi formula for balanced {P}otts spin
  glass}, Comm. Math. Phys., 405 (2024), p.~228.

\bibitem{bencs2023random}
{\sc F.~Bencs, M.~Borb\'enyi, and P.~Csikv\'ari}, {\em Random cluster model on
  regular graphs}, Comm. Math. Phys., 399 (2023), pp.~203--248.

\bibitem{bencs2024approximating}
{\sc F.~Bencs and G.~Regts}, {\em Approximating the volume of a truncated
  relaxation of the independence polytope}, arXiv preprint arXiv:2404.08577,
  (2024).

\bibitem{berezinskii1971destruction}
{\sc V.~L. Berezinski\u{i}}, {\em Destruction of long-range order in
  one-dimensional and two-dimensional systems having a continuous symmetry
  group. {I}. {C}lassical systems}, Soviet Physics JETP, 32 (1971),
  pp.~493--500.

\bibitem{10.1214/22-AOP1597}
{\sc R.~Biswas, W.-K. Chen, and A.~Sen}, {\em Free energy of a diluted spin
  glass model with quadratic {H}amiltonian}, Ann. Probab., 51 (2023),
  pp.~359--395.

\bibitem{BLANCA_CHEN_GALVIN_RANDALL_TETALI_2019}
{\sc A.~Blanca, Y.~Chen, D.~Galvin, D.~Randall, and P.~Tetali}, {\em Phase
  coexistence for the hard-core model on {$\mathbb{Z}^2$}}, Combin. Probab.
  Comput., 28 (2019), pp.~1--22.

\bibitem{bovier1992rigorous}
{\sc A.~Bovier and V.~Gayrard}, {\em Rigorous bounds on the storage capacity of
  the dilute {H}opfield model}, J. Statist. Phys., 69 (1992), pp.~597--627.

\bibitem{coja2016potts}
{\sc A.~Coja-Oghlan and N.~Jaafari}, {\em On the {P}otts antiferromagnet on
  random graphs}, Electron. J. Combin., 23 (2016).

\bibitem{coja2018information}
{\sc A.~Coja-Oghlan, F.~Krzakala, W.~Perkins, and L.~Zdeborov\'a}, {\em
  Information-theoretic thresholds from the cavity method}, Adv. Math., 333
  (2018), pp.~694--795.

\bibitem{Coja_Oghlan_2016}
{\sc A.~Coja-Oghlan and K.~Panagiotou}, {\em The asymptotic {$k$}-{SAT}
  threshold}, Adv. Math., 288 (2016), pp.~985--1068.

\bibitem{coja2019bethe}
{\sc A.~Coja-Oghlan and W.~Perkins}, {\em Bethe states of random factor
  graphs}, Comm. Math. Phys., 366 (2019), pp.~173--201.

\bibitem{coja2019spin}
\leavevmode\vrule height 2pt depth -1.6pt width 23pt, {\em Spin systems on
  {B}ethe lattices}, Comm. Math. Phys., 372 (2019), pp.~441--523.

\bibitem{Coja_Oghlan_2012}
{\sc A.~Coja-Oghlan and L.~Zdeborov\'a}, {\em The condensation transition in
  random hypergraph 2-coloring}, in Proceedings of the {T}wenty-{T}hird
  {A}nnual {ACM}-{SIAM} {S}ymposium on {D}iscrete {A}lgorithms, ACM, New York,
  2012, pp.~241--250.

\bibitem{Contucci_2013}
{\sc P.~Contucci, S.~Dommers, C.~Giardin\`a, and S.~Starr}, {\em
  Antiferromagnetic {P}otts model on the {E}rd\"os-{R}\'enyi random graph},
  Comm. Math. Phys., 323 (2013), pp.~517--554.

\bibitem{dembo2010gibbs}
{\sc A.~Dembo and A.~Montanari}, {\em Gibbs measures and phase transitions on
  sparse random graphs}, Braz. J. Probab. Stat., 24 (2010), pp.~137--211.

\bibitem{Dembo_2010}
\leavevmode\vrule height 2pt depth -1.6pt width 23pt, {\em Ising models on
  locally tree-like graphs}, Ann. Appl. Probab., 20 (2010), pp.~565--592.

\bibitem{dembo_2012}
{\sc A.~Dembo, A.~Montanari, A.~Sly, and N.~Sun}, {\em The replica symmetric
  solution for {P}otts models on {$d$}-regular graphs}, Comm. Math. Phys., 327
  (2014), pp.~551--575.

\bibitem{Dembo_2013}
{\sc A.~Dembo, A.~Montanari, and N.~Sun}, {\em Factor models on locally
  tree-like graphs}, Ann. Probab., 41 (2013), pp.~4162--4213.

\bibitem{ding2013satisfiabilitythresholdrandomregular}
{\sc J.~Ding, A.~Sly, and N.~Sun}, {\em Satisfiability threshold for random
  regular {NAE}-{SAT}}, Comm. Math. Phys., 341 (2016), pp.~435--489.

\bibitem{ding2021proof}
\leavevmode\vrule height 2pt depth -1.6pt width 23pt, {\em Proof of the
  satisfiability conjecture for large {$k$}}, Ann. of Math. (2), 196 (2022),
  pp.~1--388.

\bibitem{ding2018capacity}
{\sc J.~Ding and N.~Sun}, {\em Capacity lower bound for the {I}sing
  perceptron}, in S{TOC}'19---{P}roceedings of the 51st {A}nnual {ACM} {SIGACT}
  {S}ymposium on {T}heory of {C}omputing, ACM, New York, 2019, pp.~816--827.

\bibitem{Dommers_2010}
{\sc S.~Dommers, C.~Giardin\`a, and R.~van~der Hofstad}, {\em Ising models on
  power-law random graphs}, J. Stat. Phys., 141 (2010), pp.~638--660.

\bibitem{Durrett_2019}
{\sc R.~Durrett}, {\em Probability: Theory and Examples}, vol.~49 of Cambridge
  Series in Statistical and Probabilistic Mathematics, Cambridge University
  Press, Cambridge, fifth~ed., 2019.

\bibitem{friedli2017statistical}
{\sc S.~Friedli and Y.~Velenik}, {\em Statistical mechanics of lattice systems:
  A concrete mathematical introduction}, Cambridge University Press, Cambridge,
  2018.

\bibitem{frohlich1981kosterlitz}
{\sc J.~Fr\"ohlich and T.~Spencer}, {\em The {K}osterlitz-{T}houless transition
  in two-dimensional abelian spin systems and the {C}oulomb gas}, Comm. Math.
  Phys., 81 (1981), pp.~527--602.

\bibitem{GALVIN_KAHN_2004}
{\sc D.~Galvin and J.~Kahn}, {\em On phase transition in the hard-core model on
  {$\mathbb{Z}^d$}}, Combin. Probab. Comput., 13 (2004), pp.~137--164.

\bibitem{gamarnik2017uniqueness}
{\sc D.~Gamarnik and K.~Ramanan}, {\em Uniqueness of {G}ibbs measures for
  continuous hardcore models}, Ann. Probab., 47 (2019), pp.~1949--1981.

\bibitem{gamarnik2023computing}
{\sc D.~Gamarnik and D.~Smedira}, {\em Computing the volume of a restricted
  independent set polytope deterministically}, arXiv preprint arXiv:2312.03906,
   (2023).

\bibitem{Guerra_2004}
{\sc F.~Guerra and F.~L. Toninelli}, {\em The high temperature region of the
  {V}iana-{B}ray diluted spin glass model}, J. Statist. Phys., 115 (2004),
  pp.~531--555.

\bibitem{helmuth2023finite}
{\sc T.~Helmuth, M.~Jenssen, and W.~Perkins}, {\em Finite-size scaling, phase
  coexistence, and algorithms for the random cluster model on random graphs},
  Ann. Inst. H. Poincar\'e{} Probab. Stat., 59 (2023), pp.~817--848.

\bibitem{huang2024capacity}
{\sc B.~Huang}, {\em Capacity threshold for the {I}sing perceptron}, 2024.
\newblock arXiv preprint arXiv:2404.18902.

\bibitem{kelly1985stochastic}
{\sc F.~P. Kelly}, {\em Stochastic models of computer communication systems},
  J. R. Statist. Soc. B, 47 (1985), pp.~379--395.

\bibitem{kosterlitz1973ordering}
{\sc J.~M. Kosterlitz and D.~J. Thouless}, {\em Ordering, metastability and
  phase transitions in two-dimensional systems}, J. Phys. C, 6 (1973), p.~1181.

\bibitem{Krz_aka_a_2007}
{\sc F.~Krzakala, A.~Montanari, F.~Ricci-Tersenghi, G.~Semerjian, and
  L.~Zdeborov\'a}, {\em Gibbs states and the set of solutions of random
  constraint satisfaction problems}, Proc. Natl. Acad. Sci. USA, 104 (2007),
  pp.~10318--10323.

\bibitem{makarychev_et_al:DFU.Vol7.15301.287}
{\sc K.~Makarychev and Y.~Makarychev}, {\em Approximation algorithms for
  {CSP}s}, in The constraint satisfaction problem: complexity and
  approximability, vol.~7 of Dagstuhl Follow-Ups, Schloss Dagstuhl.
  Leibniz-Zent. Inform., Wadern, 2017, pp.~287--325.

\bibitem{mcbryan1977decay}
{\sc O.~A. McBryan and T.~Spencer}, {\em On the decay of correlations in {${\rm
  SO}(n)$}-symmetric ferromagnets}, Comm. Math. Phys., 53 (1977), pp.~299--302.

\bibitem{mermin1967absence}
{\sc N.~D. Mermin}, {\em Absence of ordering in certain classical systems}, J.
  Math. Phys., 8 (1967), pp.~1061--1064.

\bibitem{mermin1966absence}
{\sc N.~D. Mermin and H.~Wagner}, {\em Absence of ferromagnetism or
  antiferromagnetism in one- or two-dimensional isotropic heisenberg models},
  Phys. Rev. Lett., 17 (1966), pp.~1133--1136.

\bibitem{mezard2009information}
{\sc M.~M\'ezard and A.~Montanari}, {\em Information, physics, and
  computation}, Oxford Graduate Texts, Oxford University Press, Oxford, 2009.

\bibitem{mezard2001bethe}
{\sc M.~M\'ezard and G.~Parisi}, {\em The {B}ethe lattice spin glass
  revisited}, Eur. Phys. J. B, 20 (2001), pp.~217--233.

\bibitem{Mzard1986SpinGT}
{\sc M.~M\'ezard, G.~Parisi, and M.~A. Virasoro}, {\em Spin glass theory and
  beyond}, vol.~9 of World Scientific Lecture Notes in Physics, World
  Scientific Publishing Co., Inc., Teaneck, NJ, 1987.

\bibitem{montanari2012weak}
{\sc A.~Montanari, E.~Mossel, and A.~Sly}, {\em The weak limit of {I}sing
  models on locally tree-like graphs}, Probab. Theory Relat. Fields, 152
  (2012), pp.~31--51.

\bibitem{Montanari-Shah}
{\sc A.~Montanari and D.~Shah}, {\em Counting good truth assignments of random
  {$k$}-{SAT} formulae}, in Proceedings of the {E}ighteenth {A}nnual
  {ACM}-{SIAM} {S}ymposium on {D}iscrete {A}lgorithms, ACM, New York, 2007,
  pp.~1255--1264.

\bibitem{Panchenko2013TheSM}
{\sc D.~Panchenko}, {\em The {S}herrington-{K}irkpatrick model}, Springer
  Monographs in Mathematics, Springer, New York, 2013.

\bibitem{Panchenko2010SPINGM}
\leavevmode\vrule height 2pt depth -1.6pt width 23pt, {\em Spin glass models
  from the point of view of spin distributions}, Ann. Probab., 41 (2013),
  pp.~1315--1361.

\bibitem{EJP2963}
\leavevmode\vrule height 2pt depth -1.6pt width 23pt, {\em On the replica
  symmetric solution of the {$K$}-sat model}, Electron. J. Probab., 19 (2014),
  pp.~no. 67, 17.

\bibitem{panchenko2018free}
\leavevmode\vrule height 2pt depth -1.6pt width 23pt, {\em Free energy in the
  mixed {$p$}-spin models with vector spins}, Ann. Probab., 46 (2018),
  pp.~865--896.

\bibitem{c6f03298-21dc-337c-b65d-4c74d6e7603a}
\leavevmode\vrule height 2pt depth -1.6pt width 23pt, {\em Free energy in the
  {P}otts spin glass}, Ann. Probab., 46 (2018), pp.~829--864.

\bibitem{peled2019lectures}
{\sc R.~Peled and Y.~Spinka}, {\em Lectures on the spin and loop {$O(n)$}
  models}, in Sojourns in probability theory and statistical physics. {I}.
  {S}pin glasses and statistical mechanics, a {F}estschrift for {C}harles {M}.
  {N}ewman, vol.~298 of Springer Proc. Math. Stat., Springer, Singapore, 2019,
  pp.~246--320.

\bibitem{perkins2021frozen}
{\sc W.~Perkins and C.~Xu}, {\em Frozen 1-{RSB} structure of the symmetric
  {I}sing perceptron}, Random Struct. Alg., 64 (2024), pp.~856--877.

\bibitem{shcherbina2001rigorous}
{\sc M.~Shcherbina and B.~Tirozzi}, {\em Rigorous solution of the {G}ardner
  problem}, Comm. Math. Phys., 234 (2003), pp.~383--422.

\bibitem{Sinclair2014SpatialMA}
{\sc A.~Sinclair, P.~Srivastava, D.~\v{S}tefankovi\v{c}, and Y.~Yin}, {\em
  Spatial mixing and the connective constant: optimal bounds}, Probab. Theory
  Relat. Fields, 168 (2017), pp.~153--197.

\bibitem{sly2023number}
{\sc A.~Sly, N.~Sun, and Y.~Zhang}, {\em The number of solutions for random
  regular {NAE}-{SAT}}, Probab. Theory Relat. Fields, 182 (2022), pp.~1--109.

\bibitem{Talagrand_high_temp}
{\sc M.~Talagrand}, {\em The high temperature case for the random {$K$}-sat
  problem}, Probab. Theory Relat. Fields, 119 (2001), pp.~187--212.

\bibitem{book}
\leavevmode\vrule height 2pt depth -1.6pt width 23pt, {\em Mean field models
  for spin glasses. {V}olume {I}: Basic examples}, vol.~54 of Results in
  Mathematics and Related Areas. 3rd Series. A Series of Modern Surveys in
  Mathematics, Springer-Verlag, Berlin, 2011.

\bibitem{book2}
\leavevmode\vrule height 2pt depth -1.6pt width 23pt, {\em Mean field models
  for spin glasses. {V}olume {II}: Advanced replica-symmetry and low
  temperature}, vol.~55 of Results in Mathematics and Related Areas. 3rd
  Series. A Series of Modern Surveys in Mathematics, Springer, Heidelberg,
  2011.

\bibitem{Talagrand2016AMS}
\leavevmode\vrule height 2pt depth -1.6pt width 23pt, {\em A mean-field spin
  glass model based on diluted {$V$}-statistics}, Probab. Theory Relat. Fields,
  165 (2016), pp.~401--445.

\bibitem{viana1985phase}
{\sc L.~Viana and A.~J. Bray}, {\em Phase diagrams for dilute spin glasses}, J.
  Phys. C, 18 (1985), p.~3037.

\bibitem{villani2008optimal}
{\sc C.~Villani}, {\em Optimal transport: Old and new}, vol.~338 of Fundamental
  Principles of Mathematical Sciences, Springer-Verlag, Berlin, 2009.

\bibitem{weitz2006counting}
{\sc D.~Weitz}, {\em Counting independent sets up to the tree threshold}, in
  S{TOC} `06: {P}roceedings of the 38th {A}nnual {ACM} {S}ymposium on {T}heory
  of {C}omputing, ACM, New York, 2006, pp.~140--149.

\end{thebibliography}
}

\appendix
\section{Appendix}
Recall that for $\f = (f_1, \ldots, f_p) \in \X^p$, $\|\f\|_\infty := \sum_{i=1}^p \|f_i\|_\infty$.
\begin{theorem}\label{thm A1}
    Let $X, (X_n)_{n\geq 1}$ be $\X^p$-valued random variables such that $\E \|X\|_\infty, \E \|X_n\|_\infty < \infty$ for all $n$. If $X_n \to X$ in probability and $\E \|X_n\|_\infty \to \E \|X\|_\infty$, then $\E \|X_n - X\|_\infty \to 0$.
\end{theorem}

\begin{proof}

For $M>0$, let $\phi_M: \X \to \X$ be defined as \begin{align*}
    (\phi_M(f))(x) = \begin{cases}
        M, & \text{if }f(x) \geq M,\\
        f(x), & \text{if }-M \leq f(x) \leq M,\\
        -M, & \text{if }f(x) \leq -M,
    \end{cases}
\end{align*}
and for $\f \in \X^p$, we let $\phi_M(\f) = (\phi_M(f_1), \ldots, \phi_M(f_p))$.

Fix $\epsilon >0$ and let $n_0, M_0$ be large integers depending on $\epsilon$ that will be chosen later. Noting that $\|\phi_M(f) - f\|_\infty \leq \|f\|_\infty \1_{\|f\|_\infty > M}$, we write \begin{align}
    \E \|X_n - X\|_\infty & \leq \E \|\phi_{M_0}(X_n) - \phi_{M_0}(X)\|_\infty + \E \|\phi_{M_0}(X_n) - X_n\|_\infty + \E \|\phi_{M_0}(X) - X\|_\infty \notag\\
    & \leq \E \|\phi_{M_0}(X_n) - \phi_{M_0}(X)\|_\infty + \E \|X_n\|_\infty \1_{\|X_n\|_\infty > {M_0}} + \E \|X\|_\infty \1_{\|X\|_\infty > {M_0}} \label{eqn}
\end{align}
It is easy to see that $\phi_M$ is continuous on $\X^p$, so by the continuous mapping theorem, $\phi_M(X_n) \to \phi_M(X)$ in probability, and by the bounded convergence theorem, we can choose $n_0$ large enough so that the first term in \eqref{eqn} is $\leq \epsilon/3$ for all $n \geq n_0$. Owing to the dominated convergence theorem, the third term can be made $\leq \epsilon/3$ by choosing $M_0$ large enough. It remains to show that second term in \eqref{eqn} is $\leq \epsilon/3$.

For $M>1$, consider the continuous function $\psi_M:\R_+ \to \R_+$ defined as \begin{align*}
    \psi_M(t) = \begin{cases}
        t, & \text{on } [0,M-1],\\
        0, & \text{on }[M, \infty),\\
        \text{linear}, & \text{on }[M-1,M].
    \end{cases}
\end{align*}
For every fixed choice of $M$, the continuous mapping theorem and the bounded convergence theorem yield that $\E \psi_M(\|X_n\|_\infty) \to \E \psi_M(\|X\|_\infty)$. Hence, by increasing $n_0$ and $M_0$ if necessary, we can write\begin{align*}
        \E \|X_n\|_\infty \1_{\|X_n\|_\infty \geq M_0} & \leq \E \|X_n\|_\infty - \E \psi_{M_0}(\|X_n\|_\infty) \leq \E \|X\|_\infty - \E \psi_{M_0}(\|X\|_\infty) + \epsilon/4 \leq \epsilon/3.
    \end{align*}
\end{proof}
\end{document}